%% file: sawwav.tex
\documentclass[10pt,a4paper,reqno]{amsart} 

\usepackage{latexsym}
\usepackage{verbatim}
\usepackage{epsfig}
\usepackage{rotating}
\usepackage{amssymb}
\usepackage[T1]{fontenc}
\usepackage{afterpage}
\usepackage{color}

\usepackage{amssymb,amsfonts,amsmath,stmaryrd}

\graphicspath{{Figures/}}

\catcode`\@=11
\def\section{\@startsection{section}{1}%
 \z@{.7\linespacing\@plus\linespacing}{.5\linespacing}%
 {\normalfont\bfseries\scshape\centering}}

\def\subsection{\@startsection{subsection}{2}%
  \z@{.5\linespacing\@plus\linespacing}{.5\linespacing}%
  {\normalfont\bfseries\scshape}}

\def\subsubsection{\@startsection{subsubsection}{3}%
 \z@{.5\linespacing\@plus\linespacing}{-.5em}
  {\normalfont\bfseries\itshape}}
\catcode`\@=12

\newtheorem{Theorem}{Theorem}
\newtheorem{Lemma}[Theorem]{Lemma}
\newtheorem{Proposition}[Theorem]{Proposition}

\def\qed{$\hfill{\vrule height 3pt width 5pt depth 2pt}$}

\newfont{\bbold}{msbm10 scaled \magstep1}
\newfont{\bbolds}{msbm7 scaled \magstep1}

\newcommand{\ns}{\mathbb{N}}

\newcommand{\zs}{\mathbb{Z}}

\newcommand{\qs}{\mathbb{Q}}

\newcommand{\rs}{\mathbb{R}}

\newcommand{\bz}{\bar z}

\newcommand{\GP}{\mathbb{P}}

\newcommand{\E}{\mathbb{E}}
\newcommand{\Var}{\mathbb{V}}

\newcommand{\D}{\mathcal D}

\newcommand{\tU}{\tilde U}

\newcommand{\C}{\mathcal C}

\newcommand{\beq}{\begin{equation}}
\newcommand{\eeq}{\end{equation}}
\newcommand{\gf}{generating function}
\newcommand{\gfs}{generating functions}

\def\emm#1,{{\em #1}}

\newcommand{\R}{R}
\newcommand{\Ptr}{P}
\newcommand{\Tc}{T}

\newcommand{\G}{\mathcal G}

\newcommand{\Pc}{P}

\newcommand{\Rt}{R}

\newcommand{\Tt}{T}

\newcommand{\Pt}{P}
\newcommand{\Td}{T}

\newcommand{\Pd}{P}
\newcommand{\Po}{P}

\newcommand{\Ex}{{ \rm Ex}}

\newcommand{\De}{{\rm Ch}}
\begin{document}
\title[Prudent self-avoiding walks]
{Families of  prudent self-avoiding  walks}

\author{Mireille Bousquet-M\'elou}

\address{CNRS, LaBRI, Universit\'e Bordeaux 1, 351 cours de la Lib\'eration,
  33405 Talence Cedex, France}
\email{mireille.bousquet@labri.fr}
\thanks{MBM was  partially supported by the French ``Agence Nationale
de la Recherche'', project SADA ANR-05-BLAN-0372.} 

\begin{abstract}
A self-avoiding walk (SAW) on the square lattice is \emm prudent, if it
never takes a step  towards a vertex it has already visited. Prudent
walks differ from most classes of SAW that have been counted so far in that
they can wind around their starting point. 

Their enumeration was first addressed by Pr\'ea 
in 1997. He defined 4 classes of prudent walks, of increasing generality, and
wrote a system of recurrence relations for each of them. However,
these relations involve more and more parameters as the generality of
the class increases.

The first class actually consists of \emm partially directed walks,,
and its \gf\ is well-known to be rational. The second class was proved
to have an algebraic (quadratic) \gf\ by Duchi 
(2005). Here, we solve exactly the third class, which turns out to
be much more complex: its \gf\ is not algebraic, nor even D-finite. 

The fourth class --- general prudent walks --- is the only 
isotropic one, and still defeats us. However, we design an isotropic
family of  prudent  walks on the triangular lattice, which we
count exactly. Again, the \gf\ is proved to be non-D-finite.

We also study the asymptotic properties of these classes of walks, with
the (somewhat disappointing) conclusion that their endpoint moves away
from the origin at a \emm positive speed,. This is confirmed visually
by the random generation procedures we have designed.

\end{abstract}
\maketitle

\date{\today}


\section{Introduction}

\subsection{Families of self-avoiding walks}

The study of self-avoiding walks  is a famous ``elementary'' problem in
combinatorics, which is also of interest in probability theory and in
statistical physics~\cite{madras-slade}.  Recall that, given  a lattice with some
origin $O$, a self-avoiding walk (SAW) is a lattice path starting from
$O$ that does not 
visit the same vertex  twice (Fig.~\ref{fig:saw}).

\vskip -5mm 
\begin{figure}[htb]
\includegraphics[height=3cm]{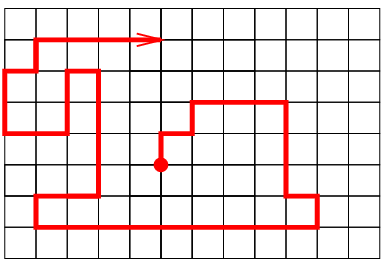}\hskip 20mm
\includegraphics[height=4cm]{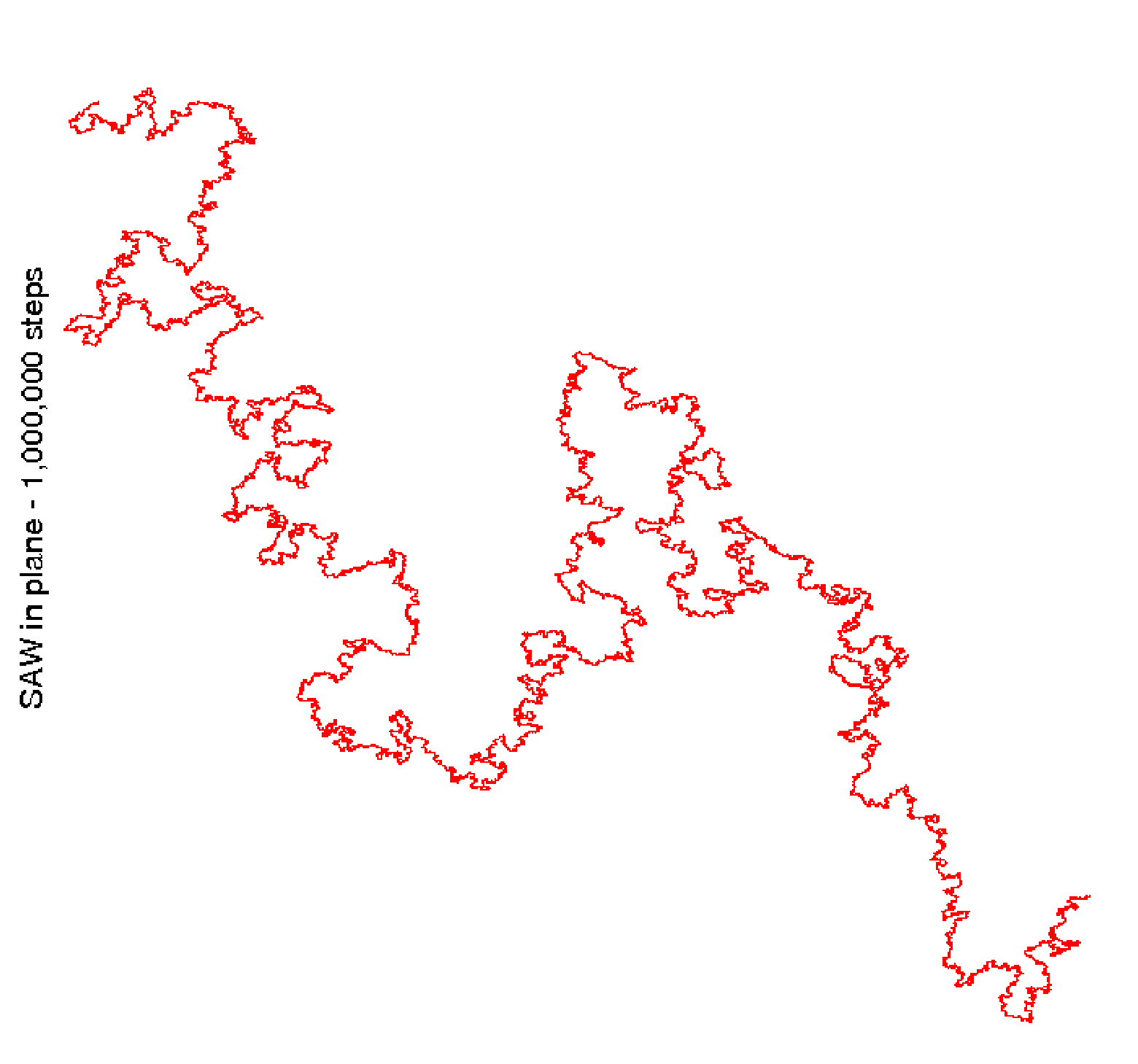}
\caption{A self-avoiding walk on the square lattice, and a
  (quasi-)random SAW of length 1,000,000, constructed by Kennedy using a
  pivot algorithm~\cite{kennedy-pivot-SAW}.} 
\label{fig:saw}
\end{figure}

It is strongly believed that, for two-dimensional lattices, the number
$c(n)$ of $n$-step SAW and the average end-to-end distance of these
walks satisfy
$$
c(n)\sim \alpha \mu^n n^\gamma \quad \hbox{and} \quad \E(D_n) \sim \kappa n^{\nu}
$$
where $\gamma=11/32$ and $\nu=3/4$. The growth constant $\mu$ is
lattice-dependent. Several independent but so far non-rigorous
methods  predict that $\mu=\sqrt{2+\sqrt 2}$ on the 
honeycomb lattice. Moreover,  numerical 
studies suggest that $\mu$ may also be a bi-quadratic number for the square
lattice~\cite{jensen-guttmann}.
On the probability side, it has been proved that, if the scaling limit
of SAW exists and has some conformal invariance property,  it must
be described by the process SLE(8/3) (stochastic Loewner
evolution)~\cite{lawler-schramm-werner}. This would imply that the predicted values of $\gamma$ and $\nu$ are correct.

The fact that all these conjectures only deal with \emm asymptotic,
properties of SAW  tells us how  far 
the problem is from the reach of \emm exact, enumeration. The followers of
this discipline 
thus focus on the study of sub-classes of SAW. 
A simple family consists of
 \emm  partially directed, walks, that is, self-avoiding walks formed of North, East
 and West steps. It is easy to see that their \gf\ is 
 rational~\cite[Example~4.1.2]{stanley-vol1}, 
\beq\label{part-dir-sol}
\sum_n c(n) t^n= \frac{1+t}{1-2t-t^2},
\eeq
which gives 
$c(n) \sim \alpha \mu^n$, with $\mu=1+\sqrt 2=2.41...$
The above series can be refined by taking into account
the coordinates  $(X_n,Y_n)$ of the endpoint, and the analysis of the result
gives:
$$
\E(X_n)=0, \quad \E(X_n^2) \simeq n  \quad \hbox{and}\quad \E(Y_n)\simeq
n.
$$
We use the notation $a_n\simeq b_n$ as a shorthand for $a_n\sim \alpha
b_n$ for some positive constant $\alpha$.

\medskip
The  prudent walks studied in this paper form a more
general class of SAW which have a natural kinetic description: a walk
is \emm prudent, if it
never takes a step pointing towards a vertex it has already
visited. In other words, the walk is so cautious that it only takes
steps in directions where the road is perfectly clear. Various
examples are shown on Fig.~\ref{fig:square}. In particular,  partially directed walks are prudent. 
Note that it is equivalent to require that the new step avoids the past
convex hull of the walk:  
hence prudent walks are the discrete counterpart of the walks
evolving in discrete time but  continuous  space  studied by
Angel \emm et al.,~\cite{angel-benjamini-virag}.

\begin{figure}[pbth]
  \begin{center}
    \includegraphics[scale=0.6]{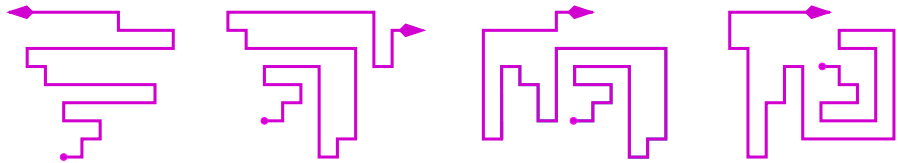}
    \caption{Four prudent walks: the first one is  partially directed
      (or: 1-sided), the other    ones are respectively     
      2-sided, 3-sided and 4-sided.}
    \label{fig:square}
  \end{center}
\end{figure}

These walks have already been studied in the past under different
names: self-directed walks~\cite{turban-debierre1}, outwardly directed
SAW~\cite{santra}, exterior SAW~\cite{prea}, and finally prudent
walks~\cite{duchi,guttmann}. We refrain from the temptation of inventing one more name
and stick to the latter terminology. The
first two papers above deal with Monte-Carlo
simulations. Pr\'ea~\cite{prea} was, to our knowledge, the first to
address enumerative questions. He wrote recurrence relations
defining an array of numbers $c(n;i,j,h)$ that count 
prudent walks  according to their length ($n$) and to three additional
\emm catalytic, parameters ($i, j, h$).  By this, we mean that these
parameters are essential to the existence of these recurrence
relations, and that it is far from obvious how to derive from them a
recursion for, say,  some numbers $c(n;i,j)$ that would only take into
account two of the catalytic parameters (or for $c(n)$).  Pr\'ea also defined four natural
families of prudent walks of increasing generality, called \emm
$k$-sided,, for $k$ 
ranging from 1 to 4. In particular, $1$-sided walks coincide with partially
directed walks, and  $4$-sided walks coincide with general prudent
walks. (Precise definitions will be given below.)
He wrote recurrence relations for each of these classes: three
catalytic parameters are needed for general
($4$-sided) prudent walks, but two suffice for $3$-sided 
walks, while one is enough for $2$-sided walks. No catalytic parameter
is needed for 1-sided walks. This reflects
the increasing generality of these four classes of walks.

 Recall that the \gf \  of 1-sided walks (partially
directed walks) is rational~\eqref{part-dir-sol}. 
Duchi~\cite{duchi} proved that  2-sided walks have an algebraic
(quadratic) \gf. She also found an algebraic \gf\ for 3-sided walks,
but there was a subtle flaw in her derivation, which was detected by
Guttmann~\cite{guttmann-comm}.  He and Dethridge performed a
numerical study of prudent 
walks, in order to get an idea of their asymptotic enumeration, and of
the properties of the associated \gfs~\cite{guttmann}. In particular,
they conjectured that the 
length \gf\ of general prudent walks
 is not \emm D-finite,, that is, does not satisfy any linear
 differential equation with polynomial coefficients. This implies that
 it is not algebraic.


\subsection{Contents}\label{sec:contents}
In Section~\ref{sec:equations} of this paper, we collect
functional equations  that define the \gfs\ of the four classes of
prudent walks introduced by Pr\'ea~\cite{prea}. This is not really original, as these
 equations are basically equivalent to Pr\'ea's recurrence
relations. Moreover, similar equations were
written by Duchi~\cite{duchi}. Our presentation may  be a bit more systematic.

In Sections~\ref{sec:2-sided} to~\ref{sec:solution-triangle} we
address the solution of these 
equations. The case of  1-sided walks is immediate and leads the
above rational \gf. We then recall how  the \emm kernel method, solves
linear equations with \emm one, catalytic variable. In particular, it 
provides the \gf\ of 2-sided walks (Section~\ref{sec:2-sided}). The
extension of this 
method to linear equations with \emm two, catalytic variables is not yet completely
understood, although a number of papers have been devoted to instances
of such equations recently~\cite{mbm-petkovsek2,mbm-kreweras,mbm-mishna,mbm-xin,rechni,marni-rechni,Mishna-jcta}.
Underlying the
\emm bi-variate kernel method, is a certain group, which depends on the
equation. Roughly speaking, the instances that have been studied
suggest that the solution is ``nice'' if the group is
finite. This belief is confirmed by the example of 3-sided walks. The
corresponding equation is associated with an infinite group, but we
can still solve it,
and prove that the \gf\ of these walks is \emm 
not D-finite,, having a rather complex singularity structure 
(Section~\ref{sec:3-sided}). We also
prove that the growth constants of 3-sided walks and 2-sided walks
are the same.   It is actually predicted from numerical
experiments that the growth constant of \emm general, prudent walks is also
the same~\cite{guttmann}.

The final equation, which deals with general prudent walks and involves
\emm three, catalytic variables, still defeats us. This is a bit annoying,
as the other classes are by definition anisotropic. However, our
understanding of the role of catalytic variables leads us to introduce
a new isotropic class of prudent walks, on the triangular lattice, which are
described by (only) two catalytic variables (Fig.~\ref{fig:triangle},
left). Again, the associated equation corresponds to an infinite
group. We  solve it, and prove that the \gf\ of  triangular prudent
 walks is not D-finite, having a natural boundary (Section~\ref{sec:solution-triangle}).

\begin{figure}[pbth]
  \begin{center}
     \scalebox{0.6}{   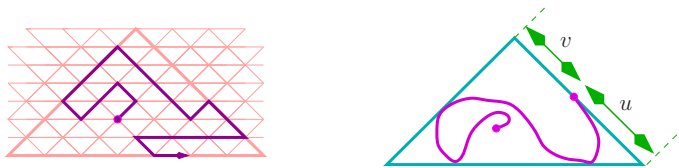}
    \caption{Left: A triangular prudent walk in a box of size 7. 
Right: The two catalytic parameters involved in the enumeration.}
    \label{fig:triangle}
  \end{center}
\end{figure}

We also refine our equations to take
into account parameters related to the end-to-end distance of a
prudent walk: for instance the coordinates of the endpoint, or the size of
the smallest  rectangle containing the walk. Extending our solutions
to these refined equations is harmless, but the conclusion we draw
from these results is somewhat disappointing:  the prudent
walks we can solve drift away from
the origin at a \emm positive speed,. In other words, the end-to-end
distance grows linearly with the length of the walk.  We do
not know what happens for general (4-sided) prudent walks.

Finally, in Section~\ref{sec:random}, we address the uniform random
generation of $n$-step prudent walks. 
Their step-by-step recursive structure
allows for a standard recursive approach, in the spirit
of~\cite{nijenhuis-wilf-book}. 
This approach first requires a precomputation stage,  followed by a
generation stage, which is usually 
linear. We emphasize that the (costly) precomputation may require less
information than the generation itself, and we use this to optimize the
precomputation stage.   Our final procedures involve the
precomputation and storage of  up to $O(n^4)$ numbers (for general
prudent walks), so that the
typical length we can reach is a few hundred. This still provides
interesting pictures
(Figs.~\ref{fig:random-4s},~\ref{fig:random-2s},~\ref{fig:random-3s},~\ref{fig:random-t}).

\begin{figure}[pbth]
\includegraphics[height=3cm,width=3cm]{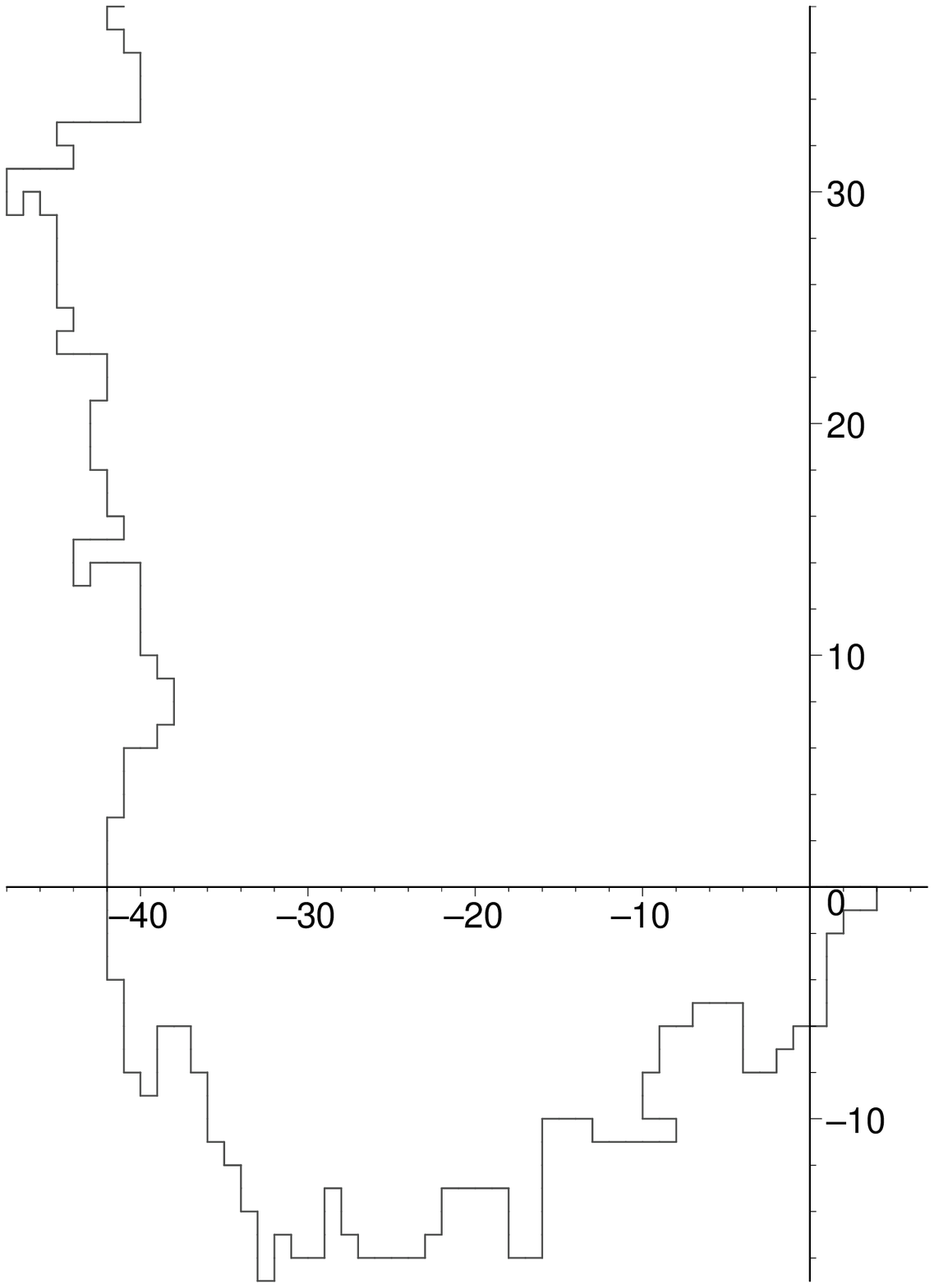}
\hskip 10mm
\includegraphics[height=3cm,width=3cm]{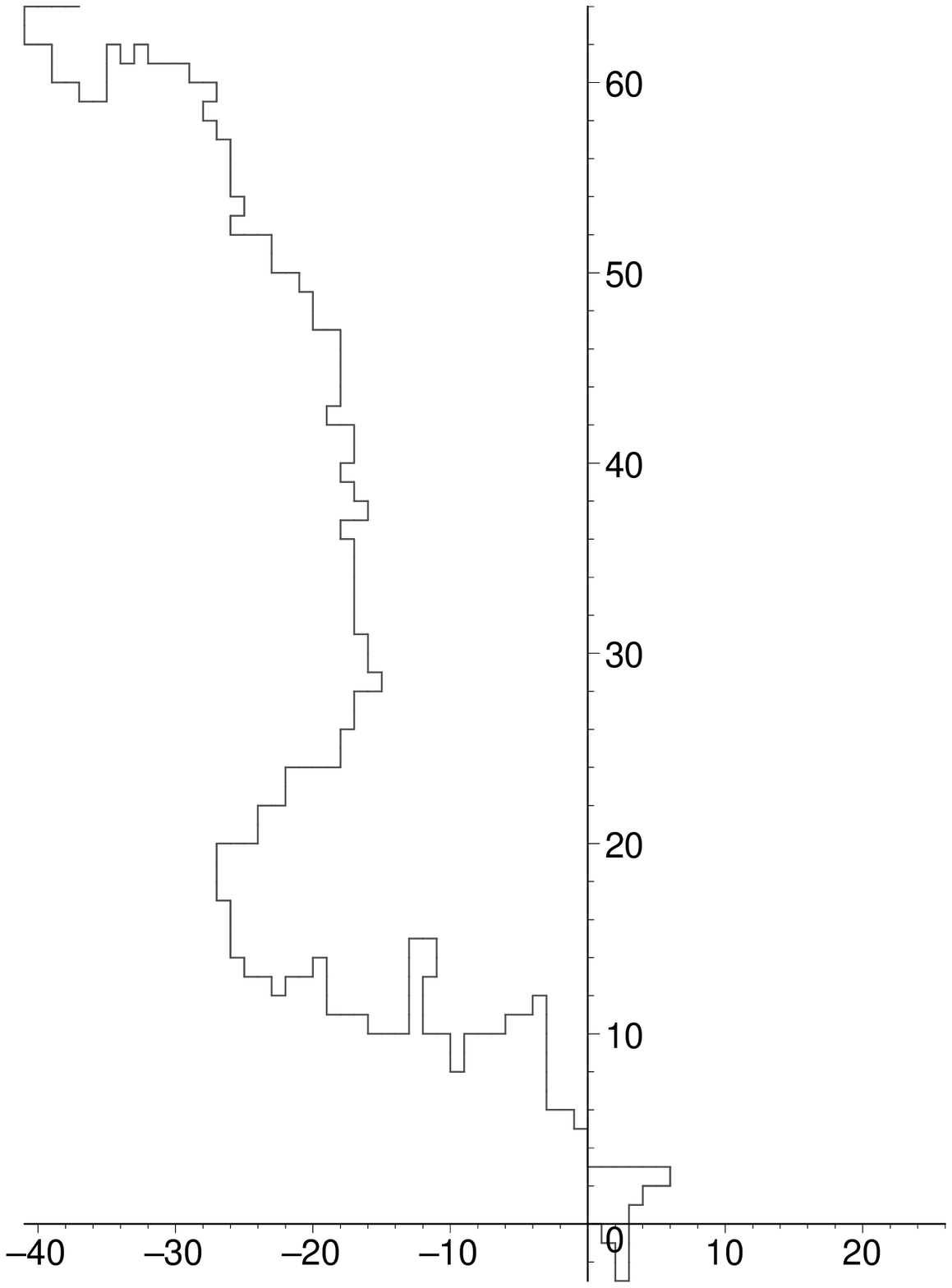}

\vskip 5mm
\includegraphics[height=3cm,width=3cm]{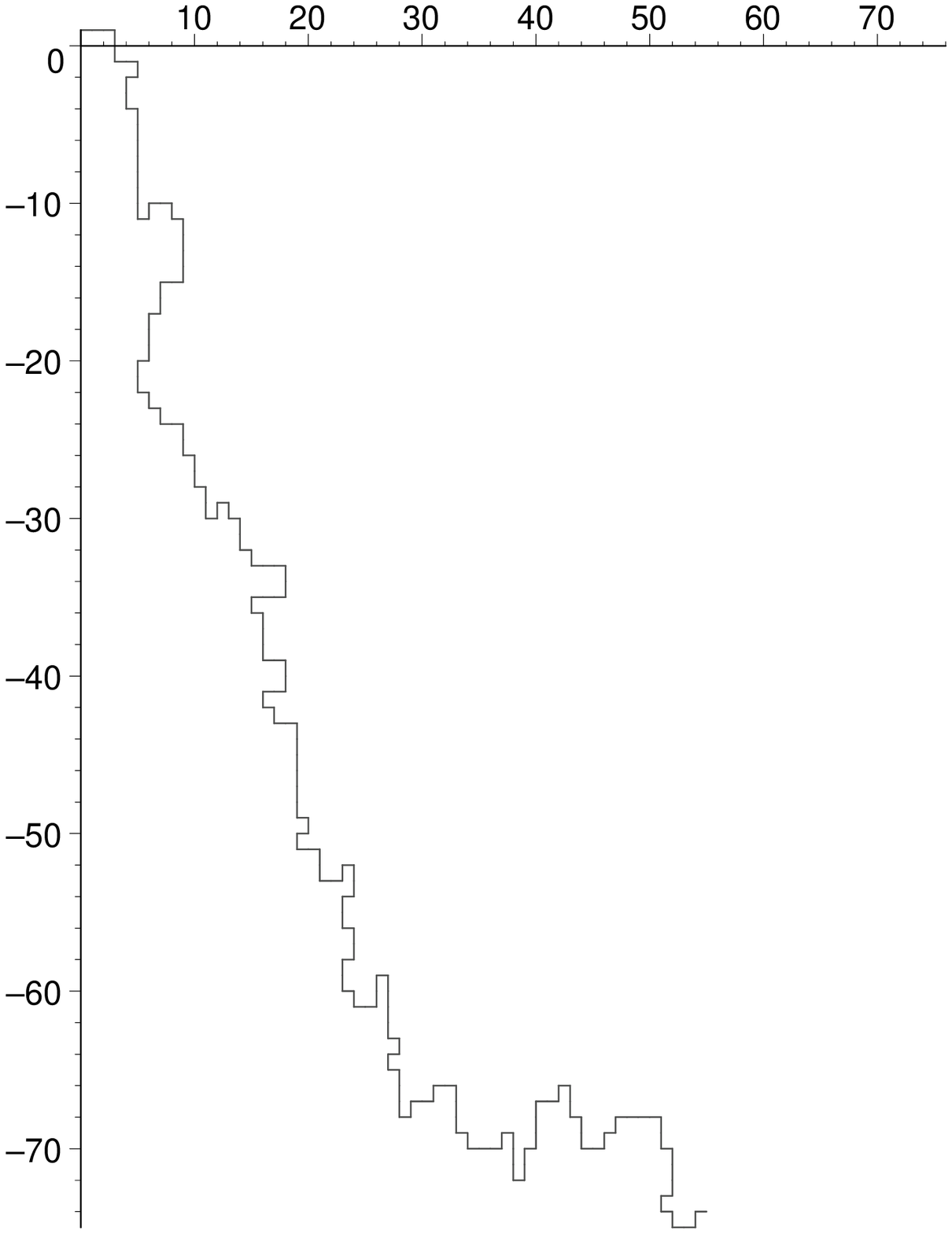} 
\hskip 10mm
\includegraphics[height=3cm,width=3cm]{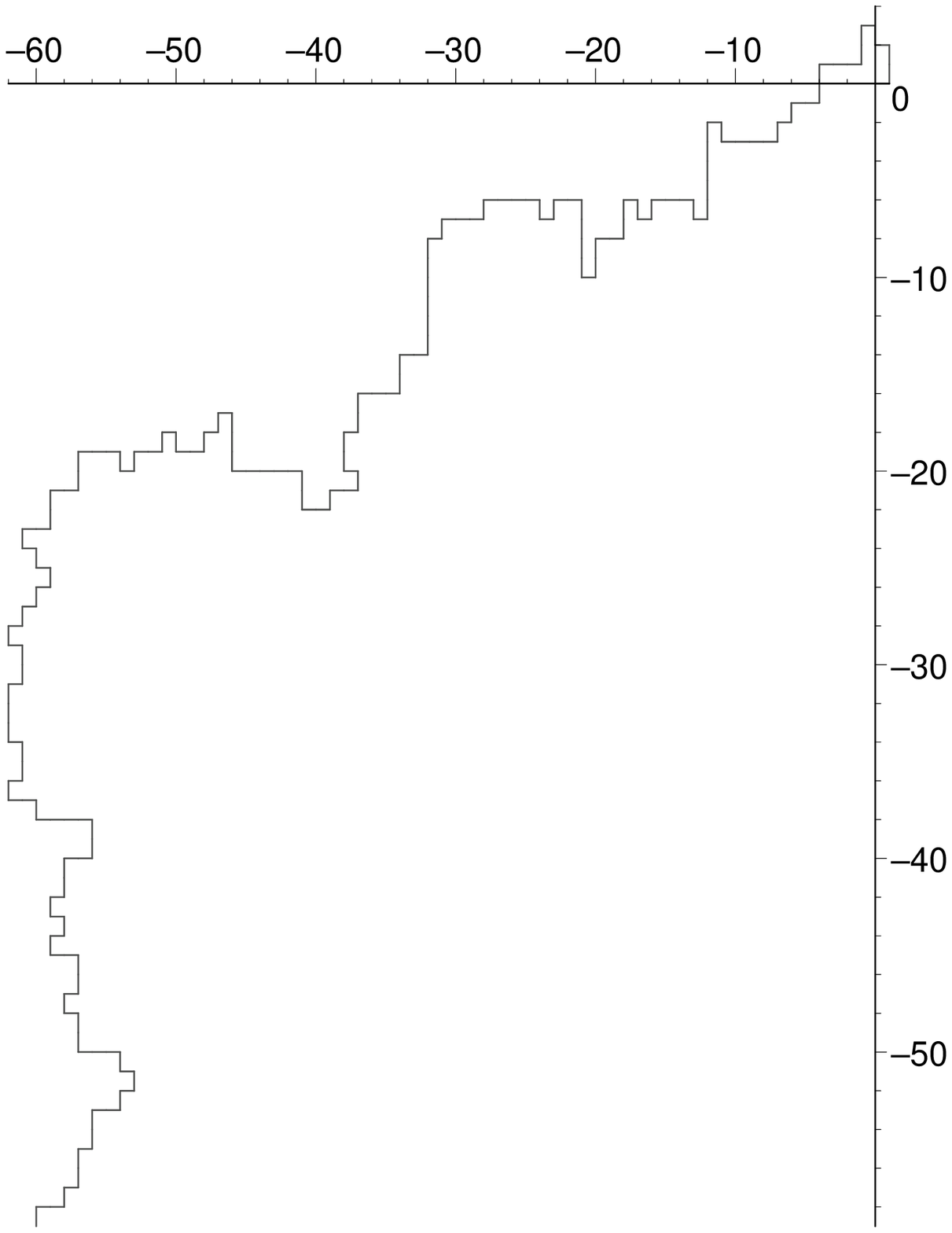}
\caption{Random prudent walks. 
}
\label{fig:random-4s}
\end{figure}

\subsection{Families of prudent walks}
Let us conclude this long introduction with some definitions and notations.
The \emm box, of a square lattice walk is
the smallest rectangle that contains it. It is not hard to see that
the endpoint of a prudent walk is always on the border of the
box. This means that every new step either walks on the border of the
box, or moves one of its four edges.

Using this box, we can give a kinetic description of 
partially directed walks: 
a prudent walk is partially directed if its endpoint, as the walk
grows,  always lies on the top  of the box. This is why partially
directed walks  are also
called \emm $1$-sided,. The generalisation of this terminology is natural: a
prudent walk is \emm $2$-sided, if its endpoint always  lies either on the top, or on the right edge of the box. It is \emm $3$-sided, if its endpoint
is always on the top, right or left edge of the box\footnote{As
  noticed by Uwe Schwerdtfeger, we
actually  require this to hold as well for non-integral points, that
is, when the walk is considered as a continuous process. The aim of
this is to avoid considering the walk ESW, among others, as a 3-sided walk.}. Of course, \emm $4$-sided, walks coincide with
general prudent walks (Fig.~\ref{fig:square}).  

Consider now a walk on the triangular lattice. Define the (triangular)
\emm box, of the walk as the smallest triangle pointing North that
contains the walk. The walk is a \emm triangular prudent walk, if
each new step either inflates the box, or walks along one edge of the
box in a prudent 
way (that is, not pointing to an already visited vertex). An example
is shown in Fig.~\ref{fig:triangle}.
Note that this is \emm not, the natural counterpart of a square lattice prudent
walk: this counterpart would just
require the walk to avoid pointing to an
already visited vertex. But then  the natural ``box''  would be a hexagon:
 every new step would either 
inflate the hexagonal box, or walk along its border. 
As we will see below, the number
of edge lengths of the box  (3 for a hexagon, 1 for a triangle) is
directly related to the number of 
catalytic parameters we have to introduce, and this is what makes
triangular prudent walks relatively easy to handle. 

Given a class of walks $\C$, the \emm \gf,\ of walks of $\C$, counted by
their length, is
$$
C(t)=\sum_{w \in \C} t^{|w|},
$$
where $|w|$ denotes the length of the walk $w$. The generalisation of
this definition to the series $C(t;u_1,\ldots, u_k)$ counting walks according to
their length and to $k$ additional parameters is immediate. We will
often drop the length variable $t$, denoting this series $C(u_1,
\ldots, u_k)$. Recall that a one-variable series $C(t)$ is \emm algebraic,
if it satisfies a polynomial equation $P(t,C(t))=0$, and \emm D-finite, if
it satisfies a linear ODE with polynomial coefficients, $P_k(t)
C^{(k)}(t) +\cdots + P_1(t) C'(t)+ P_0(t) C(t)=0$. Every algebraic
series is D-finite. We refer
to~\cite{stanley-vol2} for generalities about these classes of power series. 

\section{Functional equations}\label{sec:equations} 

The construction of functional equations for all the families of
prudent walks we study rely on the same principle, which we first
describe on 1-sided (partially directed) walks. 

Consider a 1-sided walk. If it ends with a horizontal step, we can
extend it in two different ways: either we repeat the last step, or we
change direction and add a North (N) step. Otherwise, the walk is
either empty or ends with a North step, and we have three
ways (N, E and W) to extend it. This shows that North
steps, which move the top edge of the box, play a special role in
these walks. Our functional equation is obtained by answering the
following question: where is the last North step, and what has
happened since then? 

More specifically, let $\Po(t)$ denote the length \gf\ of 1-sided
prudent walks.  The contribution to $\Po(t)$ of walks that contain no
North step  (horizontal walks) is
$$
1+2\sum_{n\ge 1} t^n = \frac{1+t}{1-t}.
$$
The other walks are obtained by concatenating a 1-sided walk, a North
step, and then a horizontal walk. Their contribution is thus
$$
\Po(t) \ t \ \frac{1+t}{1-t}.
$$
Adding these two contributions gives the equation 
$$
\Po(t)= \frac{1+t}{1-t}+ t\  \frac{1+t}{1-t}\Po(t) ,
$$from which we readily derive the rational
expression~\eqref{part-dir-sol}.

The principle of this recursive description extends to $k$-sided walks
for each $k$. We say that a step of a $k$-sided walk is \emm inflating,
if, at the time it was added to the walk, it shifted one of the $k$
edges of the box that are relevant in the definition of $k$-sided
walks. For instance, when $k=2$, an inflating step  moves the top
or right edge of the box. We write our equations by answering the
question: where is the last inflating step, and what has happened
since then?

Since then, the walk has grown \emm without
creating a new inflating step,. What does that mean? 
Assume $k\ge 2$, that the last inflating step was 
North, and that, since then, the walk has taken $m$ East steps. Then
$m$  cannot be arbitrarily large, otherwise one or several of these East
steps would be inflating, having moved the right edge of the box. This
observation explains why we have to take into account other ``catalytic''
parameters in our  enumeration of prudent walks. For instance, for a
2-sided walk, we  keep track of
the distance between the endpoint and the NE corner of the box, using a
new variable $u$ (Fig.~\ref{fig:measures}). For
3-sided walks ending on the top of the box, we keep track of the
distances between the endpoint and the NE and NW corners of the box
(variables $u$ and $v$). For 3-sided walks ending on the right edge of
the box, we keep track of the distance  between the endpoint and the
NE  corner (variable $u$) and of the width of the box
(variable $w$). For
4-sided walks ending on the top of the box, we keep track of the
distances between the endpoint and the NE and NW corners
(variables $u$ and $v$), and of the height of the box (variable
$w$). These parameters, and the names of the corresponding variables,
are schematized in Fig.~\ref{fig:measures}. They give rise to series
with one, two or three catalytic variables. For instance, for 4-sided
walks ending on the top of their box, we will use the series 
$$
\Tc(t;u,v,w) \equiv \Tc(u,v,w)=\sum_{i,j,h} \Tc_{i,j,h} u^i v^j w^h,
$$
where $\Tc_{i,j,h}\equiv\Tc_{i,j,h}(t)$ counts 4-sided walks
ending on the top of their box, at a distance $i$ (resp. $j$) from the
NE  (resp. NW) corner, such that the height of the box is $h$. 
Similar notation will be used for the other classes of walks.

\begin{figure}[pbth]
  \begin{center}
  \scalebox{0.45}{   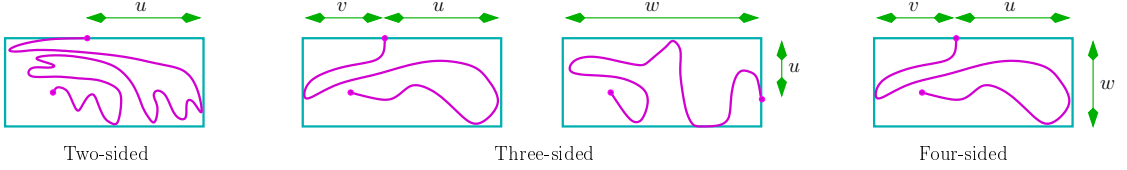}
    \caption{Catalytic variables for $k$-sided walks, $k\ge 2$.}
    \label{fig:measures}
  \end{center}
\end{figure}

Finally, for triangular prudent walks ending on the right edge of
their (triangular) box, we keep track of the distances between the
endpoint and  the SE and N corners of the box (variables $u$ and
$v$, see Fig.~\ref{fig:triangle}).

\subsection{Two-sided prudent walks}

\begin{Lemma}\label{lem:2s}
  The \gf\ $\Td(t;u)\equiv \Td(u)$ of $2$-sided walks ending
  on the top of their    box  satisfies
$$
\left(1- \frac{tu(1-t^2)}{(1-tu)(u-t)}\right)\Td(u)=\frac 1{1-tu}
+ t\ \frac{u-2t}{u-t} \,\Td(t).
$$
The \gf\ of $2$-sided walks, counted by their length
and the distance of the endpoint to the NE corner of the  box, is
$$
\Pd(t;u)= 2\Td(t;u)-\Td(t;0).
$$
\end{Lemma}
\begin{proof}
  We partition   the set of  2-sided walks ending on the top 
  of their  box into 3 classes, depending on the
  existence and   direction of the last inflating step (LIS). This
  step, if it exists,  has moved the right or top   edge of the box.  
   \begin{itemize}
    \item[1.] Neither the top  nor the right edge has ever moved:
    the walk is a sequence of West steps. The \gf\ for this class is 
$$
\frac 1 {1-tu}.
$$
 \item[2.] The LIS goes East. This implies that the endpoint of the
   walk was on the right edge of the box before that step. After that
   East step,  the walk has made a sequence of N steps to reach the
   top  of the box. Observe that, by symmetry,
   the series $\Td(t;u)$ also counts walks ending on the right edge of
   the box by
   the length and the distance between the endpoint and the NE
   corner. These two observations   give the  \gf\ for this class as
$$
t\sum_{i\ge 0} \Td_i t^i= t \Td(t).
$$
\item[3.]  The LIS goes North. After this step, there is either
    an (unbounded) sequence of West steps, or a bounded sequence of
    East steps. This gives the \gf\ for this class as
 $$
\frac{t^2u}{1-tu} \Td(u) + t \sum_{i\ge 0} \Td_i\sum_{k=0}^i t^k
u^{i-k}
=
\frac{t^2u}{1-tu} \Td(u) + \frac{t}{u-t} \left(u\Td(u)- t\Td(t)\right).
$$
  \end{itemize}
Adding the 3 terms  gives the  functional equation satisfied by
$\Td(u)$.

The expression of $\Pd(t;u)$ relies  on an inclusion-exclusion
argument: we first double  the contribution of $\Td(u)$ to take into
account walks ending on the right edge of the box, and then
subtract the series $\Td(0)$ counting those that end at the NE corner.
\end{proof}


\subsection{Three-sided prudent walks}


\begin{Lemma}\label{lem:3sided}
  The \gfs\ $\Tt(t;u,v)\equiv \Tt(u,v)$ and $\Rt(t;u,w)\equiv
  \Rt(u,w)$ that count respectively  $3$-sided walks ending on
  the top  and on the   right edge of their 
  box  are related by
  \begin{eqnarray}
\left(1- \frac{tuv(1-t^2)}{(u-tv)(v-tu)}\right) \Tt(u,v)=
1+tu \Rt(t,u) +tv \Rt(t,v)
-\frac{t^2v}{u-tv} \Tt(tv,v) -\frac{t^2u}{v-tu} \Tt(u,tu)
\label{eq:Tt}\\
\left(1-\frac{tuw(1-t^2)}{(u-t)(1-tu)}\right)\Rt(u,w)
=\frac 1 {1-tu} +t\Tt(tw,w) -\frac{t^2w}{u-t}\Rt(t,w).\hskip 37mm\label{eq:Rt}
  \end{eqnarray}
The \gf\ of $3$-sided walks, counted by the length and by the
width of the  box, is
$$
\Pt(t;u)=  \Tt(t;u,u)+2\Rt(t;1,u)-2 \Tt(t;u,0)-\frac t{1-t} .
$$ 
\end{Lemma}
\begin{proof} 
   We partition   the set of  3-sided walks ending on the top 
  of their  box into 4 subsets, depending on the
  existence and direction of the LIS, which has moved the right,
  left or top edge of the box.
   \begin{itemize}
    \item[1.] There is no inflating step at all: the walk is empty and
    contributes $1$ to the \gf. 
\item[2.] The LIS goes East. This case is analogous to Case (2) of
  2-sided walks, with \gf\ 
$$
tv \sum_{i,j} \Rt_{i,j} t^i v^j =tv\Rt(t,v).
$$
\item[3.]  Symmetrically, the  case where the LIS goes West is counted by
$$ tu \Rt(t,u).$$
\item[4.]  If the LIS is a North step, it is followed by a
  bounded number of West steps, or by a bounded number of East steps.
This case is counted by:
   \begin{multline*}
    t\sum_{i,j \ge 0} \Tt_{i,j} \left(
\sum _{k=0}^i t^k u^{i-k} v^{j+k} + \sum_{k=0}^j t^k u^{i+k}
v^{j-k} -u^iv^j\right) 
\\=
\frac{t}{u-tv} \left( u\Tt(u,v) -tv \Tt(tv,v)\right)
+\frac{t}{v-tu} \left( v\Tt(u,v) -tu \Tt(u,tu)\right)
-t \Tt(u,v).
  \end{multline*}
  \end{itemize}
Adding the 4 terms  gives the first equation of the lemma.

For 3-sided walks ending on the right edge of their  box,
the last inflating step cannot go West. Three cases remain:
  \begin{itemize}
       \item[1.] There is no inflating step at all: the walk consists
    of South steps. The \gf\ for this class is
$$
\frac 1 {1-tu}.
$$
\item[2.] The LIS goes East. This case is analogous to Case (3) of
  2-sided walks: the LIS is followed by an unbounded number of South
  steps, or by a bounded number   of North steps. The \gf \ is
$$
\frac{t^2uw}{1-tu} \Rt(u,w)+ \frac{tw}{u-t}\left( u\Rt(u,w)-t\Rt(t,w)\right).
$$
\item[3.]  The LIS goes North. This case is analogous
  to Case (2) of 2-sided walks.  The \gf\ is found to be
$$
t\Tt(tw,w).
$$
  \end{itemize}
Adding the 3 terms gives the functional equation for $\Rt(u,w)$. 

The expression of $\Pt(t;u)$ again  relies on an inclusion-exclusion
argument, based on the enumeration of walks ending on a prescribed set
of edges of their box.
\end{proof}


\subsection{General prudent walks on the square lattice}


\begin{Lemma}\label{lem:4-sided}
  The \gf\ $\Tc(t;u,v,w)\equiv \Tc(u,v,w)$ of prudent walks ending
  on the top of their    box  satisfies
$$
\left(1- \frac{tuvw(1-t^2)}{(u-tv)(v-tu)}\right) \Tc(u,v,w)=
1+\G(w,u)+\G(w,v)-\frac{tv}{u-tv} \G(v,w) -\frac{tu}{v-tu} \G(u,w)
$$
with $\G(u,v)\equiv \G(t;u,v)= tv\Tc(t;u,tu,v)$.

The \gf\ of  prudent walks, counted by the length and the
half-perimeter of the  box, is
$$
\Pc(t;u)= 1+4\Tc(t;u,u,u)-4\Tc(t;0,u,u).
$$
\end{Lemma}
\noindent We have learnt from~\cite{guttmann} that Andrew Rechnitzer has
independently obtained the functional equation satisfied by
$\Tc(u,v,w)$.
\begin{proof}
  Again, we partition the set of  prudent walk ending on the top 
  of their  box into 4 subsets, depending on the
  existence and   direction of the last inflating step.
Note that the LIS cannot be a South step.
  \begin{itemize}
    \item[1.] There is no inflating step: the walk is empty,
    and contributes 1 to the \gf.
\item[2.] The LIS goes East. This case is analogous to Case (2) of
  2-sided walks. Using the obvious symmetry between
    prudent walks ending on the top and on the right edge of their
     box, we find that the \gf\ for this class is
$$
tv \sum_{i,j,h} \Tc_{i,j,h} t^j w^{i+j} v^h =tv\Tc(w,tw,v).
$$
\item[3.]  Symmetrically, the \gf\ of prudent walks in which the LIS
  goes West is 
$$ tu \Tc(w,tw,u).$$
\item[4.]  Finally, the \gf\ of prudent walks in which the LIS goes
  North is analogous to Case (4) of 3-sided walks ending on the top, with \gf: 
  \begin{multline*}
\frac{wt}{u-tv} \left( u\Tc(u,v,w) -tv \Tc(tv,v,w)\right)
+\frac{wt}{v-tu} \left( v\Tc(u,v,w) -tu \Tc(u,tu,w)\right)
-tw \Tc(u,v,w).
  \end{multline*}
  \end{itemize}
Adding the 4 terms provides the functional equation for $\Tc(u,v,w)$,
given the obvious symmetry $\Tc(u,v,w)=\Tc(v,u,w)$.
The expression for $\Pc(t;u)$   again relies on an inclusion-exclusion
argument, based on the enumeration of walks ending on a prescribed set
of edges of their box.
\end{proof}

\subsection{Triangular prudent walks}

\begin{Lemma}\label{lem:triangle}
The \gf\ $\R(t;u,v)\equiv \R(u,v)$ of triangular prudent walks ending
on the right edge of their  box  satisfies
\beq\label{R-rec}
\left(1-\frac{tuv(1-t^2)(u+v)}{(u-tv)(v-tu)}\right) \R(u,v)
=1+tu(1+t)\frac{v-2tu}{v-tu} \R(u,tu)
+ tv(1+t)\frac{u-2tv}{u-tv} \R(tv,v).
\eeq
The \gf\ of triangular prudent walks, counted by the length and the
size of the  box, is 
\beq\label{Ptr-R}
\Ptr(t;u)= 1+ 3\R(t;u,u) -3\R(t;u,0).
\eeq
\end{Lemma}
\begin{proof}
  We partition the set of prudent walks ending on the right
  edge of their    box into 7 subsets, depending on the
  existence and direction of  the last inflating step.
  \begin{itemize}
     \item[1.] There is no inflating step: the walk is empty,
    and contributes 1 to the \gf.
\item[2.] The LIS is a NW step. This implies that the
    endpoint of the walk was on the left edge of the  box
    before that step.  Thanks to the obvious symmetry between walks
    ending on the left edge and right edge of their  box,
    we obtain the \gf\ of walks of this type as
$$
tu\sum_{i,j} \R_{i,j} t^i u^{i+j}= tu\R(tu,u).
$$

 \item[3.] The LIS is a W step. Again, the
    endpoint of the walk was on the left edge of the  box
    before that step.  The \gf\ of walks of this type is
$$
tu\sum_{i,j} \R_{i,j} t^{i+1} u^{i+j}= t^2u\R(tu,u).
$$ 
  \end{itemize}
The cases where the LIS goes SE or SW are very
    similar to the two previous cases. The
    endpoint of the walk was on the bottom edge of the  box
    before the last inflating step.    
 \begin{itemize}
\item[4.]  The \gf\ of walks such that the LIS is a SE step is
$$
tv\sum_{i,j} \R_{i,j} t^j v^{i+j}= tv\R(v,tv).
$$
 \item[5.]  The \gf\ of walks such that the LIS is a SW step is
$$
tv\sum_{i,j} \R_{i,j} t^{j+1} v^{i+j}= t^2v\R(v,tv).
$$ 
 \end{itemize}
We are left with the two richer cases where the LIS goes East or
North-East.   The endpoint of the walk was already on the right
edge of the  box  before the last inflating step, and the LIS is
followed by a bounded sequence of SE or NW steps.
 \begin{itemize}
\item[6.] The \gf\ of walks such that the LIS is an E step  is
  \begin{multline*}
 tv \sum_{i,j} \R_{i,j}
\left( 
\sum_{k=0}^i t^k u^{i-k} v^{j+k} 
+\sum_{k=1}^{j+1} t^k u^{i+k} v^{j-k}\right)
\\= 
\frac{tv}{u-tv} \left( u\R(u,v)-tv\R(tv,v)\right)
+ \frac{t^2u}{v-tu} \left( v\R(u,v)-tu\R(u,tu)\right).
     \end{multline*}
 \item[7.] The \gf\ of walks such that the LIS goes NE  is
  \begin{multline*}
tu \sum_{i,j} \R_{i,j}
\left( 
\sum_{k=0}^j t^k v^{j-k} u^{i+k} 
+\sum_{k=1}^{i+1} t^k u^{i-k} v^{j+k}\right)
\\= 
\frac{tu}{v-tu} \left( v\R(u,v)-tu\R(u,tu)\right)
+ \frac{t^2v}{u-tv} \left( u\R(u,v)-tv\R(tv,v)\right).
     \end{multline*}
 \end{itemize}
We add the  7 terms above and  note that $\R(u,v)=\R(v,u)$ to obtain
the functional equation for $\R(u,v)$.
The expression of $\Ptr(t;u)$  again relies on an inclusion-exclusion
argument.
\end{proof}

\section{Enumeration and asymptotic properties of 2-sided prudent  walks}
\label{sec:2-sided}
In this section, we recall how the \emm kernel method, works on linear
equations with one catalytic
variable~\cite{bousquet-petkovsek-1,hexacephale,prodinger}, using the
example of 2-sided walks. We first recover Duchi's algebraic \gf, and
then refine our enumeration to keep track of other statistics like the
end-to-end distance of the walk. This section is a sort of warm-up
before solving the more difficult equations of Sections~\ref{sec:3-sided}
and~\ref{sec:solution-triangle}. 
\begin{Proposition}
\label{prop-sol-2sided}
The \gf\ $\Pd(t;u)$ of $2$-sided walks, counted by their length and by
the distance between the endpoint and the NE corner of the box, is
$$
\Pd(t;u)=
\frac{2(1-t^2)(1-t)U}{(1-uU)(1-tU)(2t-U)}-1
$$
where 
$$
U\equiv U(t)= \frac {1-t+t^2+t^3 - \sqrt{(1-t^4)(1-2t-t^2)}}{2t}.
$$
In particular, the length \gf\ is
$$
\Pd(t;1)= \frac 1{1-2t-2t^2+2t^3} \left(
1+t-t^3 + t(1-t) \sqrt{\frac{1-t^4}{1-2t-t^2}}
\right).
$$
\end{Proposition}
\begin{proof}
We start from the functional equation of Lemma~\ref{lem:2s}, written
as
\beq\label{Td-eq}
\left((1-tu)(u-t)-tu(1-t^2)\right)\Td(u)=u-t
+ t (1-tu)(u-2t) \Td(t).
\eeq
The  series $U\equiv U(t)$ given in the proposition is the only power series in
$t$ that cancels the 
\emm kernel, of this equation, that is, the polynomial
$\left((1-tu)(u-t)-tu(1-t^2)\right)$. The series $\Td(U)\equiv\Td(t;U)$ is
well-defined. Replacing $u$ by $U$ in the equation cancels the
left-hand side, and hence the right-hand side, giving
$$
\Td(t)=\frac{U-t}{t(1-tU)(2t-U)}.
$$
Then~\eqref{Td-eq} gives $\Td(u)$, and the second equation of
Lemma~\ref{lem:2s} provides $\Pd(t;u)$.
\end{proof}

\begin{figure}[pbth]
\includegraphics[height=3cm,width=3cm]{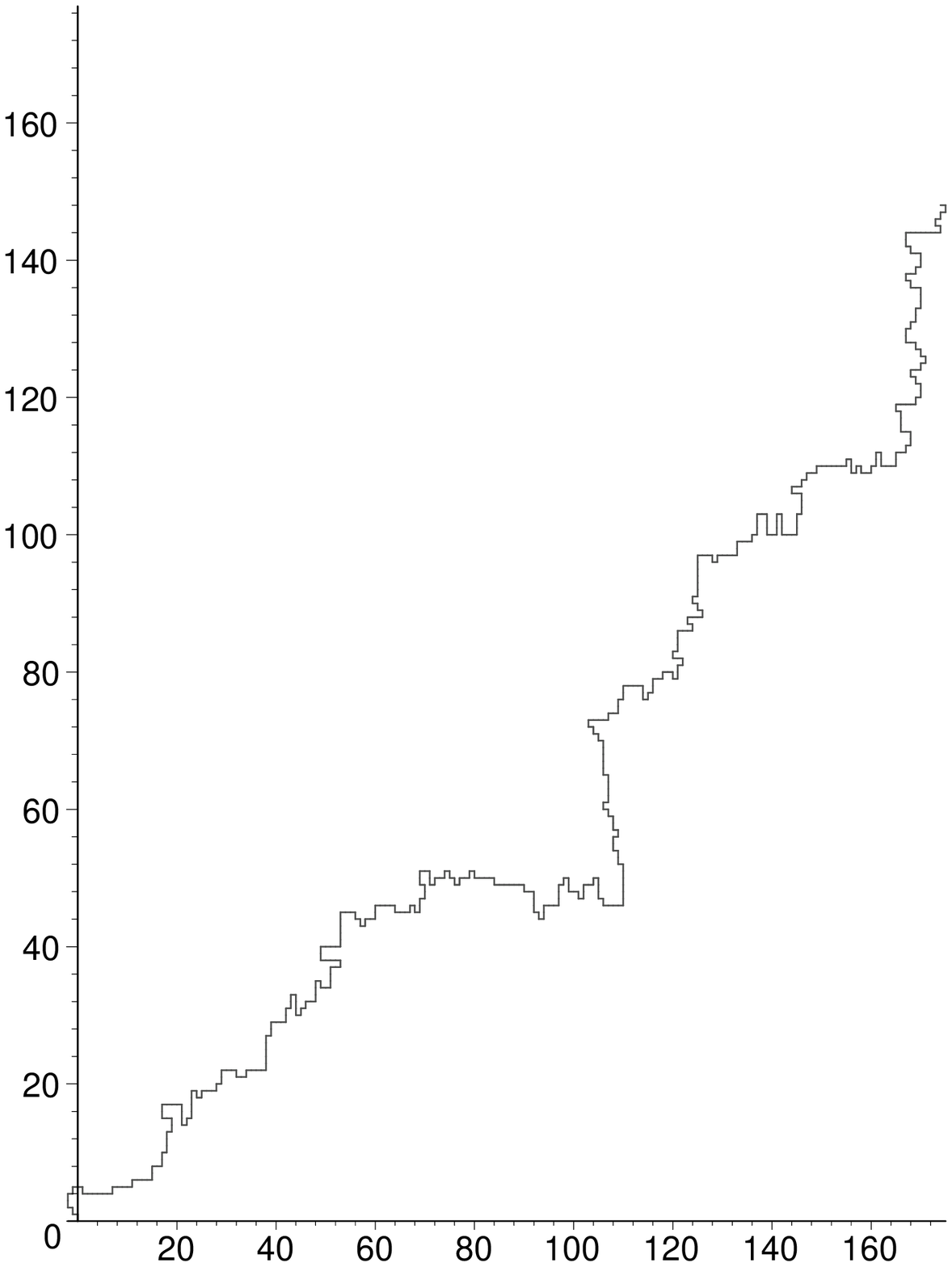}
\hskip 10mm
\includegraphics[height=3cm,width=3cm]{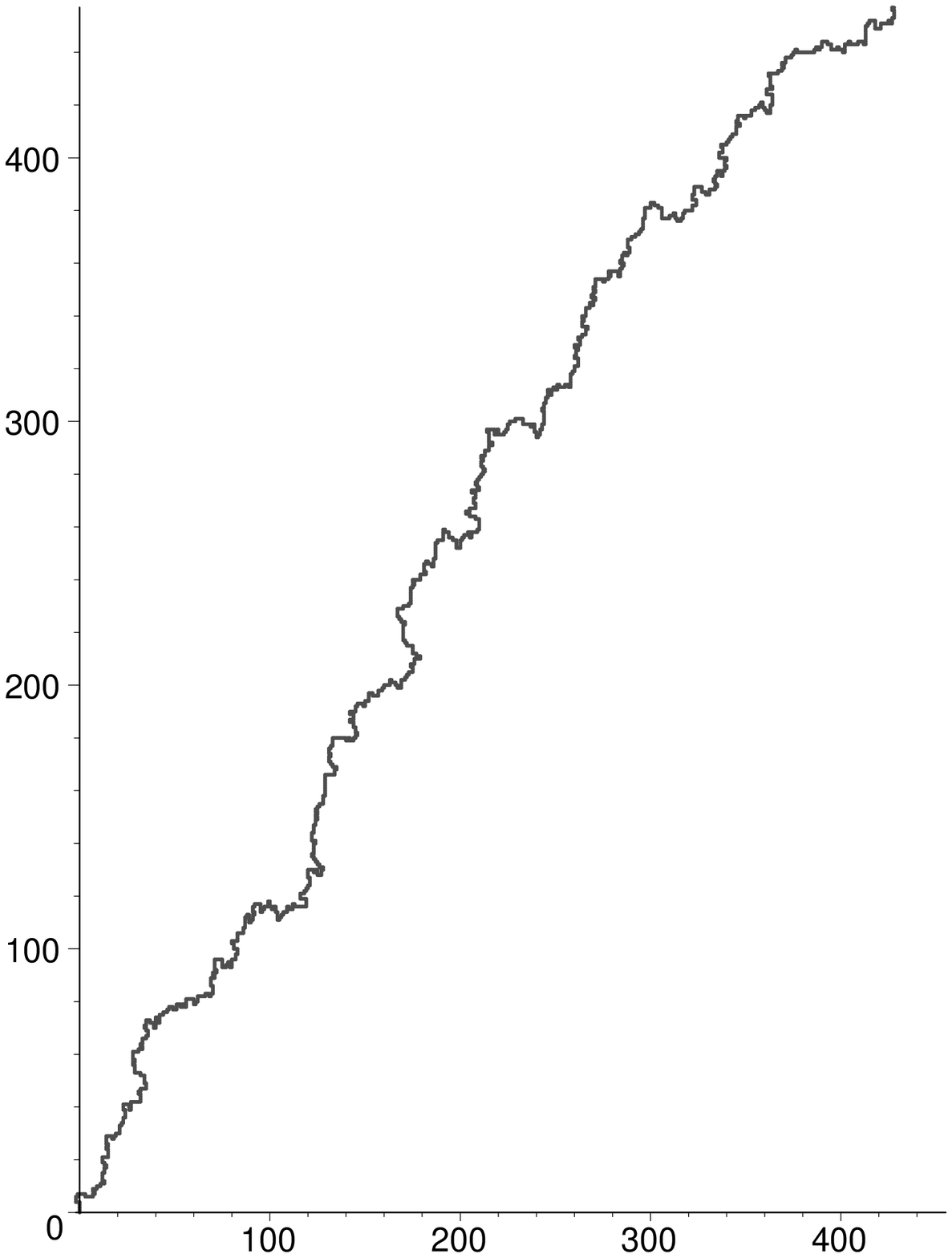}
\hskip 10mm
\includegraphics[height=3cm,width=3cm]{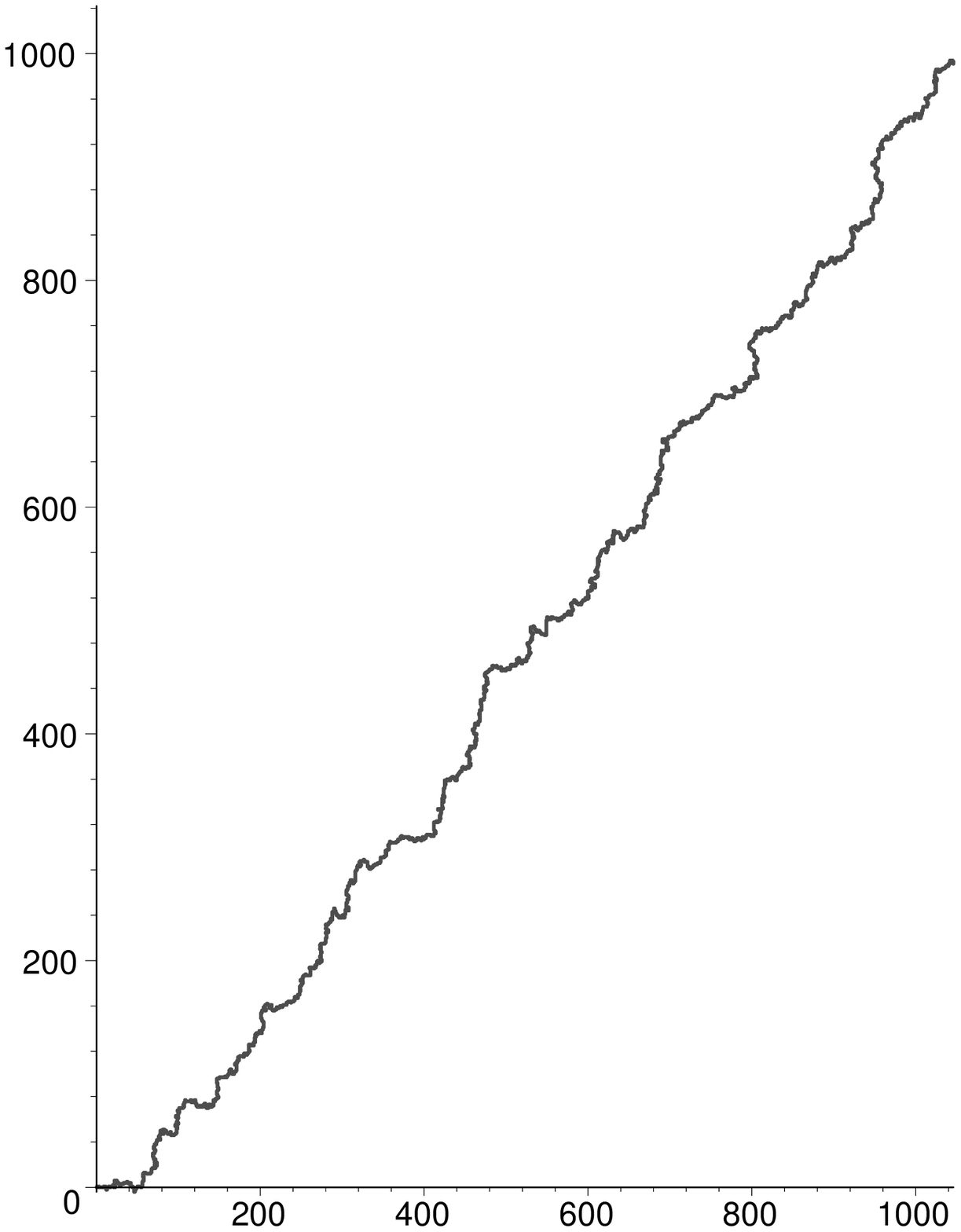}
%
\caption{Random 2-sided  walks of length 500, 1354 and 3148.}
\label{fig:random-2s}
\end{figure}

\begin{Proposition}[{\bf Asymptotic properties of 2-sided walks}]
\label{prop-2-sides-asympt}
The length \gf\, of $2$-sided walks
has a unique singularity of minimal modulus,
$\rho\simeq 0.403$, which is a simple pole and satisfies $1-2\rho-2\rho^2+2\rho^3=0$. 
The number of $n$-step $2$-sided walks satisfies
$$
p_n \sim \kappa \mu^n \quad \hbox{with} \quad \mu=\frac 1\rho\simeq 2.48
\quad \hbox{and} \quad \kappa=
{\frac {\rho \left( 3\rho -1 \right) }{ \left( 3\rho+1 \right) 
 \left( 5\rho-2 \right) }}\simeq 2.51 .
$$

 The  distance between the endpoint and the NE corner of the box in a
 random $n$-step $2$-sided walk follows asymptotically a geometric law
 of parameter $2\rho$. In particular, the average value of this
 distance is asymptotically constant, equal to 
$
\frac{2\rho}{1-2\rho}
\simeq 4.15.$

Let $X_n$ (resp. $Y_n$) denote the abscissa (resp. ordinate)
of the endpoint of a random $n$-step $2$-sided walk. Then the mean and
variance of $X_n+Y_n$ satisfy
$$
\E(X_n+Y_n) \sim  \frak m  \, n,  \quad  \quad
\Var(X_n+Y_n) \sim  \frak s ^2 \, n ,  
$$
where 
$$
 \frak m = {\frac {\rho+1}{3\,\rho+1}}\simeq 0.63
\quad \hbox{and} \quad
\frak s ^2= {\frac {4 \left( \rho+1 \right) ^{2}\rho}{ \left( 3\,\rho+1
 \right) ^{3} \left(1- \rho \right) }}\simeq 0.49,
$$
 and the variable
$
\frac{X_n+Y_n-  \frak m \, n}{\frak s \sqrt n}
$
converges in law to  a standard normal distribution.

Finally,
$$
\E(X_n-Y_n) =0,  \quad  \quad
\Var(X_n-Y_n) \sim  \frak s ^2 \, n ,  
$$
where 
$$
\frak s ^2= 
{\frac {\rho \left( {\rho}^{2}-2 \right)  \left( 1+\rho \right) }{
 \left( {\rho}^{2}+\rho-1 \right)  \left( 3\,\rho-1 \right)  \left( 1+
3\,\rho \right) }}
\simeq 5.17,
$$
and the variable
$
\frac{X_n-Y_n}{\frak s \sqrt n}
$
converges in law to a standard normal distribution.

These asymptotic properties are in good agreement with the random $2$-sided
walks of Fig.~\rm\ref{fig:random-2s}.
\end{Proposition}
\begin{proof}
We start from the expression of $\Pd(t;1)$ given in
Proposition~\ref{prop-sol-2sided}. The singularities of $\Pd(t;1)$ are
found among the roots of the denominator 
$(1-2t-2t^2+2t^3)$ and those of the discriminant $(1-t^4)(1-2t-t^2)$. 
It is not hard to see that the
smallest one (in modulus) is $\rho$, corresponding to a simple pole of
the series. In the expression of $P(t;u)$ in terms of $U$ given
in Proposition~\ref{prop-sol-2sided}, this pole is found when $U=2t$.
This implies $\sqrt{(1-\rho^4)(1-2\rho-\rho^2)}=(1-4\rho^2)/2$.
The asymptotic behaviour  $p_n \sim \kappa \mu^n$ easily follows, with
$\mu=\rho^{-1}$.

\smallskip
To study the distance to the NE corner, consider
$\Pd(t;u)$ (given in Proposition~\ref{prop-sol-2sided}). For $u$ in a
neighborhood of $1$, $\Pd(t;u)$ still admits $\rho$ as its unique
dominant singularity, which remains a simple pole. This gives
$$
[t^n] \Pd(t;u) \sim \frac{1-U(\rho)}{1-uU(\rho)}\,[t^n] \Pd(t;1)
,
$$
with $U(\rho)=2\rho$, and the result follows using a continuity
theorem for probability  \gfs~\cite[Thm.~IX.1]{flajolet-sedgewick}.

\smallskip
Then, we enrich our enumeration by taking into account the sum of
the coordinates of the endpoint, using a new variable $z$. The
functional equations of Lemma~\ref{lem:2s} become:
\begin{eqnarray*}
\left(1- \frac{tuz^2(1-t^2)}{(z-tu)(u-tz)}\right)\Td(t,z;u)&=&\frac z{z-tu}
+ tz\ \frac{u-2tz}{u-tz} \Td(t,z;tz),\\
\Pd(t,z;u)&=& 2\Td(t,z;u)-\Td(t,z;0).
\end{eqnarray*}
We solve them in the same way that we solved the case $z=1$, and
obtain:
$$
\Pd(t,z;u)=
\frac{2z^3(1-t^2)(1-tz)U}{(z^2-uU)(z-tU)(2tz-U)}-1
$$
with
$$
U\equiv U(t,z)= z\,
\frac{1-tz+t^2+t^3z-\sqrt{(1-t^2)(1+t-tz +t^2z)(1-t-tz-t^2z)}}
{2t}.
$$
For $z$ in a neighborhood of $1$, the radius of convergence $\rho_z$ of
$\Pd(t,z;1)$ is reached when $U(\rho_z,z)=2z\rho_z$, or
$1-2z\rho_z-2\rho_z^2+2z\rho_z^3=0$.
One easily checks that $\Pd(t,z;1)$ satisfies the \emm meromorphic
schema, of~\cite[Thm.~IX.9]{flajolet-sedgewick}, and the limit
behaviour of $X_n+Y_n$ follows.

\smallskip
 The study of the distance  between the endpoint and the first
diagonal is similar.  We first refine the enumeration by taking into account
the difference $X_n-Y_n$, using a new variable $z$. When establishing
the functional equation satisfied by $\Td(t,z;u)\equiv \Td(z;u)$, one
must note that walks ending on the right edge of the box have \gf\
$\Td(t,1/z;u)$. The equations finally read
\begin{eqnarray}
\left(1+{\frac {tu \left( 1-t^2 \right)  }{ \left( z-tu \right)  \left( u-tz \right) }}\right)
\Td(z;u)&=& \frac z{z-tu}+tz \Td(\bz; t\bz) - \frac{t^2}{u-tz}
\Td(z;tz), \label{2-sided-diff}\\
\Pd(t,z;u)&=& \Td(t,z;u)+\Td(t,\bz;u)-\Td(t,z;0)\nonumber
\end{eqnarray}
with $\bz=1/z$.
The kernel method provides a linear equation between $\Td(z;tz)$,
$\Td(\bz; t\bz)$, involving the (quadratic) series $U(z)\equiv U(t,z)$ that
cancels the kernel of~\eqref{2-sided-diff}.
 Replacing $z$ by $1/z$ gives a second linear equation, now involving  $\Td(z;tz)$,
$\Td(\bz; t\bz)$ and $U(\bz)$. Thus
both series $\Td(z;tz)$ and $\Td(\bz; t\bz)$ can be expressed in terms of
$U(z) $ and $U(\bz)$. It is  then straightforward to  obtain expressions for $\Td(t,z;u)$ and
$\Pd(t,z;u)$  in terms of
$U(z) $ and $U(\bz)$ (we  recommend using a formal algebra
system). In particular, each of these series is algebraic 
of degree 4.

 By an obvious symmetry argument, $\E(X_n-Y_n)=0$.
As most continuity theorems involve  \emm non-negative, random variables, we
consider $n+X_n-Y_n$, which is clearly non-negative, with mean
$n$. The associated 
\gf\ is $\Pd(tz,z;1)$. For $z$ in a neighborhood of $1$, the radius of
convergence $\rho_z\equiv \rho$ of this series
 is reached when 
$$
{\rho}^{7}{z}^{8}+{\rho}^{5} \left( {\rho}^{2}+2\,\rho-4 \right) {z}^{
6}-{\rho}^{3} \left( 4\,{\rho}^{2}+5\,\rho-5 \right) {z}^{4}+\rho\,
 \left( \rho-2+5\,{\rho}^{2} \right) {z}^{2}-2\,\rho+1
=0
,$$
and is again a simple pole of $\Pd(tz,z;1)$.
One easily checks that we are again in the meromorphic
schema of~\cite[Thm.~IX.9]{flajolet-sedgewick}, and the limit
behaviour of $n+X_n-Y_n$ (and consequently, of $X_n-Y_n$) follows.
\end{proof}

\section{Enumeration and asymptotic properties of 3-sided prudent
  walks}
\label{sec:3-sided}
\begin{Proposition}
\label{prop-sol-3sided}
The \gf\ of $3$-sided walks ending on the top  of their
 box  satisfies
$$
\Tt(t; u,tu) =  \sum_{k \ge 0} (-1)^k
\frac{\prod_{i=0}^{k-1}\left( \frac t{1-tq}-U(uq^{i+1})   \right)}
{\prod_{i=0}^{k}\left( \frac {tq}{q-t}-U(uq^{i})  \right)  }
  \left( 
1+ \frac{U(uq^k)-t}{t(1-tU(uq^k))} + \frac{U(uq^{k+1})-t}{t(1-tU(uq^{k+1}))}
\right)
$$
where 
$$
U(w)\equiv U(t;w)=\frac{1-tw+t^2+t^3w-\sqrt{(1-t^2)(1+t-tw+t^2w)(1-t-tw-t^2w)}}{2t}
$$
is the only power series in $t$ satisfying $(U-t)(1-tU)=twU(1-t^2)$,
and
$$q\equiv q(t)= U(t;1)=
\frac{1-t+t^2+t^3-\sqrt{(1-t^4)(1-2t-t^2)}}{2t}.
$$
The \gf \ $\Pt(t;u)$ of $3$-sided walks, counted by their length and
by the width of the  box, can be expressed rationally in terms of
$U(u)$ and $\Tt(t;u,tu)$ (see~\eqref{Ptu-expr}). When $u=1$, this gives the length \gf\ as
 $$
\Pt(t;1)=\frac 1{1-2t-t^2}\left(
2\,{t}^{2}q\,T (t;1,t) +{\frac { \left( 1+t \right)  
\left(2-t -{t}^{2}q  \right) }{1-tq}}\right)
-\frac 1 {1-t}.
$$
\end{Proposition}
\noindent Note that the series $U(t;1)$ is the algebraic series that
occurs in the solution of 2-sided walks (Proposition~\ref{prop-sol-2sided}).
\begin{proof}
In the equation~\eqref{eq:Rt} satisfied by $\Rt(u,w)$, the only
catalytic variable is $u$ (there is no occurrence of $\Rt(\cdot, w')$
with $w'\not=w$). Thus we can apply the standard kernel method:
setting $u=U(w)$ cancels the coefficient of $\Rt(u,v)$, and we are left with
\beq \label{eq:Rt-Tt}
tw\Rt(t,w)= \frac{U(w)-t}t \left( \frac 1 {1-tU(w)} +t\Tt(tw,w)\right).
\eeq
Recall that $\Tt$ is symmetric in its two (catalytic) variables. In particular,
$\Tt(tw,w)=\Tt(w,tw)$. 
The equation~\eqref{eq:Tt} satisfied by $\Tt(u,v)$ involves the series
$\Rt(t,u)$ 
and $\Rt(t,v)$. We use~\eqref{eq:Rt-Tt} to express them in terms of
$U(u)$, $U(v)$, $\Tt(u,tu)$ and $\Tt(v,tv)$, and obtain an
equation that involves only the series $\Tt$:
\begin{multline}
\left(1- \frac{tuv(1-t^2)}{(u-tv)(v-tu)}\right) \Tt(u,v)=
\label{eq:Tt-seul}\\
1+   \frac{U(u)-t}{t(1-tU(u))}
 + \frac{U(v)-t}{t(1-tU(v))}
-\left( \frac{tv}{v-tu} -U(u)\right) \Tt(u,tu)
-\left( \frac{tu}{u-tv} -U(v)\right) \Tt(v,tv).
\end{multline}
Now both $u$ and $v$ play catalytic roles. We want to cancel the
kernel of this new equation, namely the polynomial
$K(u,v)=(u-tv)(v-tu)-tuv(1-t^2)$, by an appropriate choice of
$v$. Thus $v$ will be a function of $u$ and $t$. As $K(u,v)$ is
\emm homogeneous, in $u$ and $v$,  the dependency of $v$ in $u$ is
extremely simple:  $K(u,v)$ vanishes for $v=qu$, where $q=U(t;1)$ only
depends on $t$
(of course,  $K(u,v)$ also vanishes for $v=u/q$, but this is not a
power series in $t$). 
This simplicity is crucial
to writing the solution in an explicit form. Replacing $v$ by $qu$ in~\eqref{eq:Tt-seul} gives 
$$
\Tt(u,tu)= -\frac{\frac t{1-tq}
  -U(uq)}{\frac{tq}{q-t}-U(u)}\Tt(uq,tuq)
+ \frac 1{\frac{tq}{q-t}-U(u)} \left( 
1+ \frac{U(u)-t}{t(1-tU(u))} + \frac{U(uq)-t}{t(1-tU(uq))}
\right).
$$
Observe that $tq/(q-t)=\frac{1-tq}{1-t^2}= 1+O(t)$, while $U(u)=
O(t)$. Hence $\frac{tq}{q-t}-U(u)$ is
invertible in the ring of power series in $t$ with coefficients in
$\qs[u]$. Moreover, $\frac{t}{1-tq}-U(uq)=O(t^3)$, so that we can
iterate the above equation indefinitely, replacing $u$ by $uq$, then
by $uq^2$, and so on. The net result is the
expression of $\Tt(u,tu)$ given in the proposition.

\medskip
We now seek an expression for $\Pt(t;u)$, which was given in terms of
$\Tt(u,u)$, $\Tt(u,0)$ and $\Rt(1,u)$ in  the third equation of
Lemma~\ref{lem:3sided}. 
We wish to express each of these series in terms of $U(u)$ (which is
known explicitly) and $\Tt(u,tu)$, which we have just determined.
To express $\Rt(1,u)$ (or, equivalently, $R(1,w)$),
we combine the case $u=1$ of~\eqref{eq:Rt} with~\eqref{eq:Rt-Tt}. 
For the other  two series, namely $\Tt(u,u)$ and $\Tt(u,0)$, we
specialize~\eqref{eq:Tt-seul} to $v=u$, and then to $v=0$
(using $U(0)=t$). 
Putting together the three pieces gives an expression for  $\Pt(t;u)$
in terms of $U(u)$ and $\Tt(u,tu)$, which can (for instance) be
written as follows: 
\begin{multline}
 \Pt(t;u)=
\frac 1{1-2t-t^2}\left(
2\,{t}^{2} U (u) T (u,tu)
+{\frac { \left( 1+t \right)  
\left(2-t -{t}^{2}U (u)  \right) }
{1-tU (u) }}\right)\\
-\frac{2(1-U(u))(1+t)(1-u)}{(1-2t-t^2)(1-t-tu-t^2u)}
\left({t}^{2}T (u,tu) +{\frac { \left( t+1 \right) 
 t}{1-tU (u) }}\right)
-\frac 1 {1-t}.\label{Ptu-expr}
\end{multline}
Note that the second term has a factor $(u-1)$: this makes the
specialization $u=1$ obvious, and gives the announced expression of $\Pt(t;1)$.
\end{proof}

\noindent{\bf Where is the group?} In the introduction
(Section~\ref{sec:contents}), we 
wrote that a group is associated with every linear equation with two
catalytic variables, and that the group associated with 3-sided walks
is infinite. What is this group? It is generated by two
transformations $\Phi$ and $\Psi$ that act on ordered pairs $(u,v)$
and leave $K(u,v)/u/v$ unchanged:
$$
\Phi(u,v)=\left(\frac{v^2}u,v\right),
\quad \Psi(u,v)=\left(u,\frac{u^2}v\right).
$$
It is easy to see that they generate an infinite group. Indeed, if one
repeatedly  applies 
$\Psi\circ\Phi$  to the pair $(u,v)$, one obtains all pairs $(v^{2i}/u^{2i-1},v^{2i+1}/u^{2i})$
 for $i \ge 0$.
The role of this group was first recognized in the study of Markov
chains in the quarter plane~\cite{fayolle-livre}. See also the more
recent and combinatorial
papers~\cite{mbm-petkovsek2,mbm-kreweras,mbm-mishna,mbm-xin,rechni,marni-rechni,Mishna-jcta}.
As announced in the introduction, we now proceed to show that the
series $\Pt(t;1)$ is not D-finite.

\begin{figure}[htb]
\includegraphics[height=3cm,width=3cm]{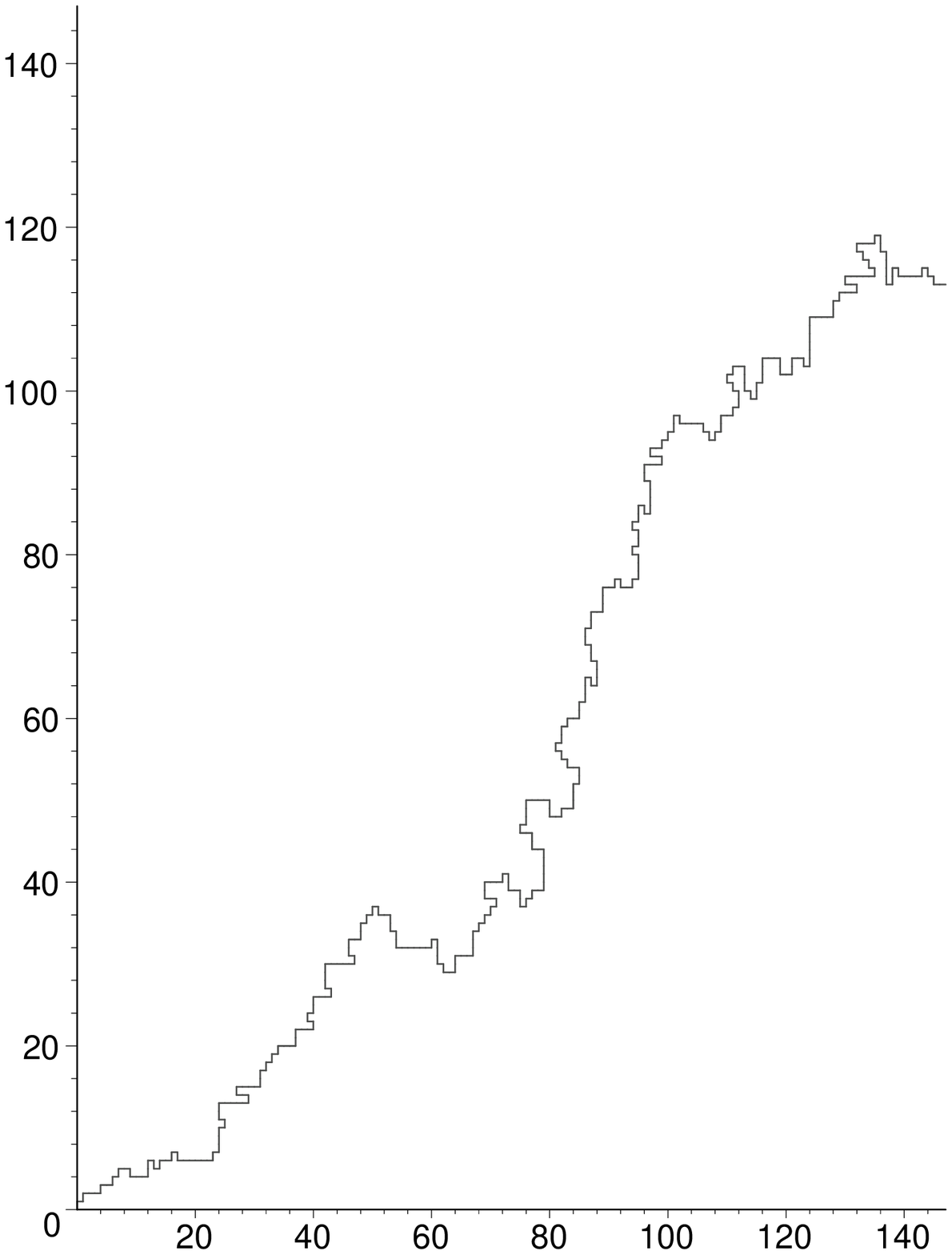}
\hskip 10mm
\includegraphics[height=3cm,width=3cm]{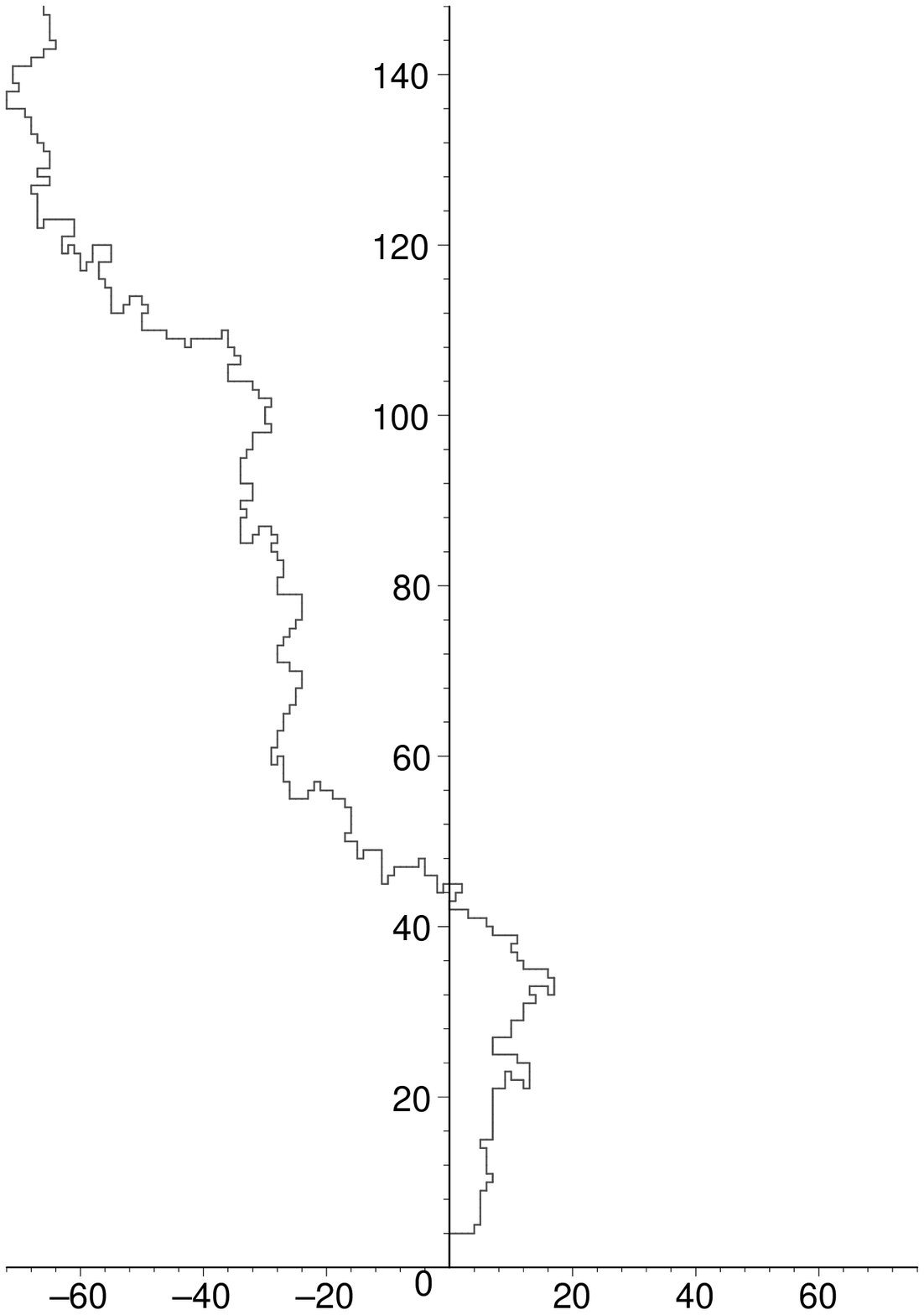}
\hskip 10mm
\includegraphics[height=3cm,width=3cm]{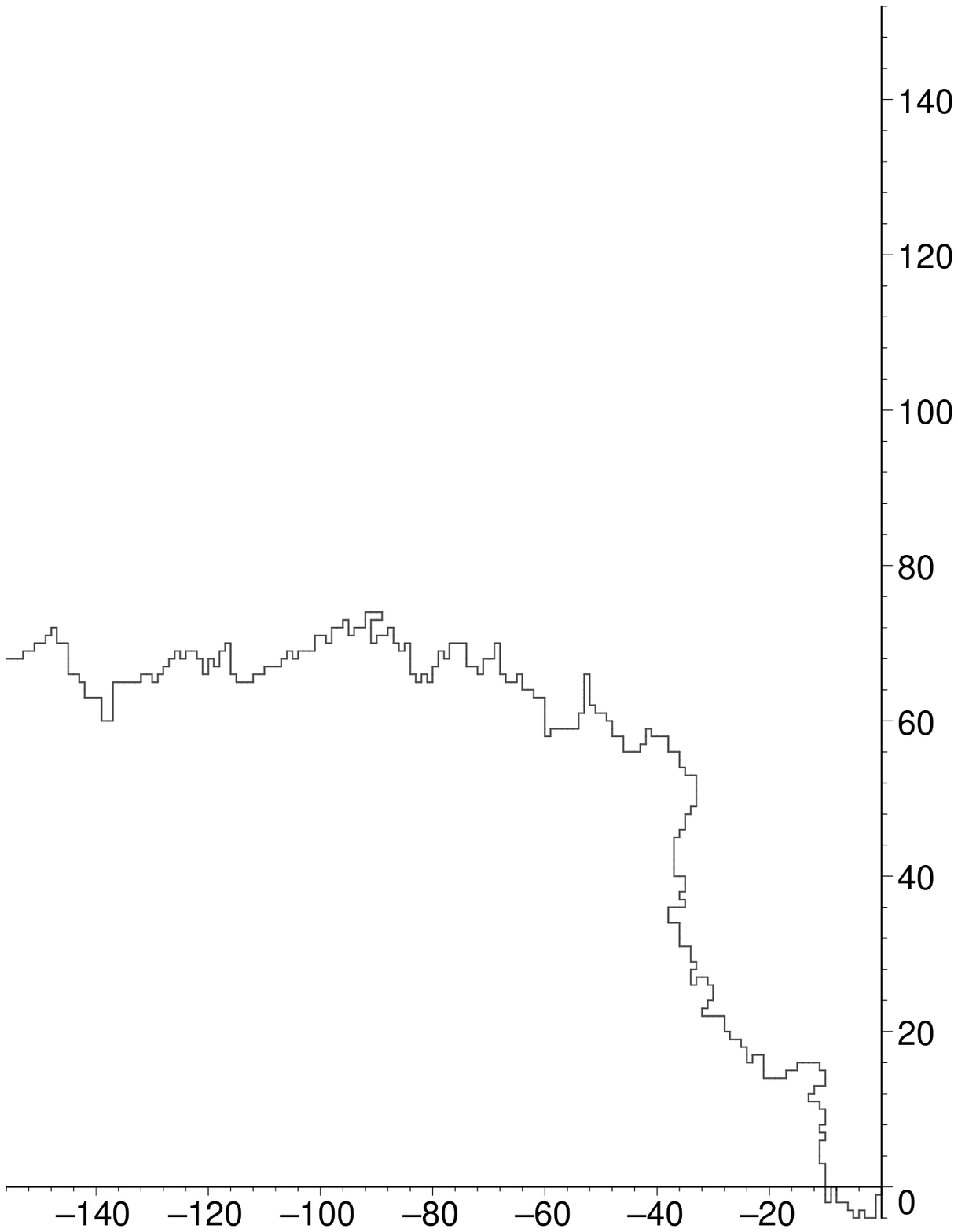} 
\caption{Random 3-sided  walks of length 400.}
\label{fig:random-3s}
\end{figure}

\begin{Proposition}[{\bf  Nature of the g.f. and asymptotic properties
    of 3-sided walks}]
\label{prop-nature-3sided}
The length \gf\ $\Pt(t;1)$ of $3$-sided walks is meromorphic in the
disk  $\D=\{t: |t|<t_c\}$, with $t_c=\sqrt 2 -1$.
 It has infinitely many poles in this disk, and thus cannot be D-finite. 
On the segment $(0, t_c)$, the poles are simple and form an increasing
sequence $(t_i)_{i\ge 0}$ that  tends to $t_c$.

The pole $t_0$ is the unique dominant
singularity of $\Pt(t;1)$ and coincides with the radius of
convergence $\rho$ of the series counting $2$-sided walks
(Proposition~\ref{prop-2-sides-asympt}). Thus  $t_0=\rho$ satisfies
$1-2\rho-2\rho^2+2\rho^3=0$, and the number of $n$-step $3$-sided walks is 
$$
p_n \sim \kappa \mu^n \quad \hbox{with} \quad \mu=\frac 1 \rho\simeq
2.48
$$
for some $\kappa >0$. 

Let $W_n$  denote the width
of the box of a random $n$-step $3$-sided walk. Then
$$
\E(W_n) \sim  \frak m  \, n,  \quad  \quad
\Var(W_n) \sim  \frak s ^2 \, n ,  
$$
where 
$$
 \frak m = \frac {1+\rho}{2(1+3\rho)}  \simeq 0.31
\quad \hbox{and} \quad
\frak s ^2= {\frac {3\rho\, \left( 1+\rho \right)  \left( 385-1148\,{\rho}^
{2}-494\,\rho \right) }{16 \left( {\rho}^{2}+\rho-1 \right)  \left( 3\,
\rho-1 \right) ^{3} \left( 1+3\,\rho \right) ^{3}}}
\simeq 1.41,
$$
 and the variable
$
\frac{W_n-  \frak m \, n}{\frak s \sqrt n}
$
converges in law to  a standard normal distribution.
\end{Proposition}
\begin{proof}
We start from the expression of $\Pt(t;1)$ given in
Proposition~\ref{prop-sol-3sided}. The polynomial $1-2t-t^2$ does not
vanish in $\D$ (although it vanishes at $t_c$),
 and we will show below that $q(t)$ is analytic in
$\D$, and that $1-tq(t)$ does not vanish in this domain. Hence most of
our analysis will focus on the series $\Tt(t;1,t)$. We will prove
that it is meromorphic in $\D$,
with a sequence of real positive simple poles $\rho=t_0 <t_1 <t_2<\cdots<t_c$
satisfying  
$$
\frac{t_iq(t_i)}{q(t_i)-t_i}=U(q(t_i)^i).
$$
Clearly, our first concern will be the series $U(u)$.

\medskip \noindent $\bullet$
{\bf The series $U(u)$.}
The quadratic equation that defines $U(u)$ gives
\beq\label{U-tilde}
U(u)=t+\tU(u)  \quad \hbox{with}  \quad 
\tU(u)= \frac{tuU(u)}{1-\frac{t\tU(u)}{1-t^2}},
\eeq
which shows that both $U(u)$ and $\tU(u)$ have coefficients in $\ns[u]$.
In particular, for all $t$ and $u$ such that  $U(|t|;|u|)$ converges,
$U(t;u)$ also converges and satisfies $|U(t;u)|\le U(|t|;|u|)$.

For $u=1$, we find that the radius of convergence of $U(t;1)=q(t)$ is
at $t_c$. Thus $U(t;1)$ is analytic in $\D$. Moreover, we
note that $U(t_c;1)=1$. The non-negativity of the coefficients of
$U(t;1)=q(t)$ implies that $|q(t)|<1$ for $t\in \D$. 

Consequently, for $t\in \D$ and $i\ge 0$, the series $U(t;q(t)^i)\equiv U(q^i)$ is
(absolutely) convergent, and thus analytic in $\D$. Moreover, 
$\left|tU(q^i)\right|< t_c\,U(t_c;1)=t_c<1$, so that $1-tU(q^i)$
does not vanish in $\D$.

\medskip \noindent $\bullet$
{\bf The numerator and denominator of $\Tt(t;1,t)$.}
Recall the expression of $\Tt(t;u,tu)$ given in
Proposition~\ref{prop-sol-3sided}. 
Note that $q(t)^i\rightarrow 0$ as $i\rightarrow \infty$, both as a
power series in $t$ (because $q(t)=O(t)$) and for every $t\in \D$
(because $|q(t)|<1$). In
particular, the term $tq/(q-t)-U(q^i)=t^2/(q-t)-\tU(q^i) $ converges
to $t^2/(q-t)$. 
Let 
\beq\label{D-def}
D(t)= \prod_{i\ge 0} \left(1- \frac{q-t}{t^2} \, \tU(q^i)\right).
\eeq
As $\frac{q-t}{t^2}=1+O(t)$ and $ \tU(q^i)=O(t^i)$, this is a well-defined series in
$t$.  We  write
\beq
\label{T-N-D}
\Tt(t;1,t)=\frac{N(t)}{D(t)}
\eeq
with
$$
N(t)= \sum_{k \ge 0} (-1)^k N_k(t)
$$
and 
\beq\label{Nk-def}
N_k(t)=
\left( \frac{q-t}{t^2}\right)^{k+1}
{\prod_{i=0}^{k-1}\left( \frac t{1-tq}-U(q^{i+1})   \right)}
{\prod_{i>k}\left(1-\frac{q-t}{t^2} \tU(q^{i})  \right)  }  T_k(t)
\eeq
with
$$
T_k(t)=
1+ \frac{U(q^k)-t}{t(1-tU(q^k))} + \frac{U(q^{k+1})-t}{t(1-tU(q^{k+1}))}.
$$
We will now prove that $D(t)$ and $N(t)$ are analytic in $\D$, so that
$\Tt(t;1,t)$ is meromorphic.

\medskip \noindent $\bullet$
{\bf The series $D(t)$ is analytic in $\D$.}
 As discussed above, every term of the product $D(t)$ is  analytic in
 $\D$. We still need to prove that the product converges in $\D$. 
For $t\in\D$ and  $|u|\le 1$, the equations~\eqref{U-tilde} imply that $|\tU(t;u)|\le
 |u| \tU(|t|;1)$. Consequently, 
\beq\label{ineq0}
|\tU(q^i)|\le |q(t)|^i \, \tU(|t|;1),
\eeq with $|q(t)|<1$, 
 so that the series $\sum_i |\tU(q^i)|$ is convergent
 for $t\in \D$. The same holds for the product $D(t)$.

\medskip \noindent $\bullet$
{\bf The series $N(t)$ is analytic in $\D$.} 
It follows from the properties of
$q, U$ and $\tU$  that every summand $N_k$
 (given by~\eqref{Nk-def}) is analytic in $\D$. We will prove the
convergence of the series $\sum_{k \ge 0} | N_k(t)|$ by bounding 
$N_k(t)$. Let us begin with $T_k(t)$.
First, we note that 
$$
\frac{U(t;u)-t}{t(1-tU(t;u))}= \frac{u (1-t^2) U(t;u)}{1-tU(t;u)}
$$ 
is uniformly bounded (by a constant) for $t\in \D$ and $|u|\le 1$: hence 
$T_k(t)$ is  uniformly bounded by a constant for $t\in\D$.

Let us now bound  the infinite product occurring
in $N_k(t)$. We write 
\begin{multline*}
  \left|\prod_{i>k}\left(1-\frac{q-t}{t^2} \tU(t;q^{i})  \right) \right|
\le
\exp\left(\sum_{i>k}\frac{|q-t|}{|t|^2} |\tU(t;q^{i})|   \right)\\
\le
\exp\left(\sum_{i>k}\frac{|q-t|}{|t|^2} \tU(|t|;1) |q(t)|^i   \right)
\le
\exp\left(\frac{|q-t|}{|t|^2}\frac{ \tU(|t|;1)|}{1-|q(t)|} \right)
< \infty.
\end{multline*}
The second inequality follows from~\eqref{ineq0}.

Hence it suffices to prove the convergence of
$$
\sum_k |M_k(t)| \quad \hbox{ with }\quad 
M_k(t)= \left( \frac{q-t}{t^2}\right)^{k+1}
{\prod_{i=0}^{k-1}\left( \frac t{1-tq}-U(q^{i+1})   \right)}
.
$$
Recall that $U(t;q(t)^k)\rightarrow U(t;0)=t$  as $k\rightarrow \infty$. Hence
$$
\frac{M_k(t)}{M_{k-1}(t)}
=  \frac{q-t}{t^2}\left( \frac t{1-tq}-U(q^{k})   \right)
\rightarrow   \frac{q(q-t)}{1-tq}.
$$
Due to the positivity of the coefficients of the series 
$q(t)=U(t;1)$ and $q(t)-t= \tU(t;1)$, the modulus of this ratio is
strictly bounded 
in $\D$ by the value it takes at   $t_c$, namely 1. Thus $
\left|{M_k(t)}/{M_{k-1}(t)}\right|$ converges, as $k$ grows, to a
limit that is less than 1:
the convergence of $\sum_k |M_k(t)| $ follows, and implies that $N(t)$
is analytic in $\D$. 

\medskip
Consequently,  $\Tt(t;1,t)=N(t)/D(t)$ is meromorphic in $\D$. In this
disk, all its singularities are poles, found among the zeroes of
$D(t)$.  The product form~\eqref{D-def} of
$D(t)$ leads us to study the zeroes of each  factor, or,
equivalently, the zeroes of $tq/(q-t)-U(q^i)$.

\medskip \noindent $\bullet$
{\bf Zeroes of $tq/(q-t)-U(q^i)$.} We fix $i\ge 0$ and focus on the
real interval 
$(0, t_c)$. On this interval, $U(t;q(t)^i)$ increases from $0$
to $1$.  The identity
$$
\frac{tq}{q-t}=\frac{1-tq}{1-t^2}=1- \frac{t\tU(t;1)}{1-t^2}
$$
shows that $tq/(q-t)$ decreases from 1 to $1/\sqrt 2<1$. 
Thus there exists a unique $t_i$ in $(0, t_c)$ such that
$t_iq(t_i)/(q(t_i)-t_i)= U(q(t_i)^i)$. Moreover, as $U(q^i)\ge
U(q^{i+1})$ for $t \in (0, t_c)$, one has $t_0<t_1<t_2<\cdots
<t_c$ and
\beq
\label{ineq2}
\frac{t_iq(t_i)}{q(t_i)-t_i}- U(q(t_i)^k) >0 \quad \hbox{for } k>i.
\eeq
 Each of these zeroes is simple, as the derivative of
$tq/(q-t)$ is negative. As $\Tt(t;1,t)$ is meromorphic in $\D$, its
poles are isolated, so that the increasing sequence
%
$(t_i)_i$ can only converge to $t_c$.

The equation satisfied by $\rho:=t_0$ reads $q(\rho)=2\rho$, which
yields the cubic equation of the proposition.

\medskip \noindent $\bullet$
{\bf Existence of infinitely many poles.} It remains to prove that each $t_j$, for
$j\ge 0$, is
actually a pole of $\Tt(t;1,t)=N(t)/D(t)$, that is, that $N(t_j)\not = 0$.
Recall the expression~\eqref{Nk-def} of $N_k$.  For $t=t_j$, one has
$U(q^j)=tq/(q-t)$, or, equivalently, $\tU(q^j)=t^2/(q-t)$. Hence
$$
N(t)=\sum_{k\ge j} (-1)^k N_k(t).
$$
We will show the following properties: for $t=t_j$,
\begin{itemize}
\item [(A)]
 $(-1)^jN_j(t)> 0$, and for $k>j$,  the signs of $(-1)^kN_k(t)$
alternate, starting from a positive sign:  $(-1)^{k-j-1}N_k(t) >
0$,
\item  [(B)] $|N_{j+1}(t)|> |N_{j+2}(t)| > |N_{j+3}(t)| >\cdots$
\end{itemize}
The combination of these two properties implies that $N(t_j)>0$, and
in particular that $t_j$ is indeed a pole of $\Tt(t;1,t)$.

Let us write $t=t_j$, and study the sign of $N_k(t)$. We begin with
the signs of $T_k(t)$ and $q-t$. We have already seen that  $U(u)-t=\tU(u)>0$
and $tU(u)<1$ for  $u \in (0,1)$. In particular, 
$$
T_k(t)\ge 1\quad \hbox{and }\quad  q-t>0.
$$
Let us move to the infinite product occurring in $N_k(t)$. Now for $i>k\ge j$,
$$
1- \frac{q-t}{t^2} \tU(q^i)=1-\frac{\tU(q^i)}{\tU(q^j)}>0
$$
as $q \in (0,1)$ and $\tU(u)$ is an increasing function of $u$, for
$0<u<1$.

We are left with the sign of $t/(1-tq)-U(q^{i+1})$, for $i\ge 0$. We
will prove that, still denoting $t=t_j$,
\beq\label{ineq}
U(q^{j+2})<\frac t{1-tq} <U(q^{j+1}).
\eeq
Given that the sequence $U(q ^i)$ decreases as $i$ grows, this gives
\begin{eqnarray}
\frac t{1-tq}- U(q^{i+1}) 
&<0 & \hbox{for } 0\le i \le j,\nonumber\\
& >0 & \hbox{for } i \ge j+1,\label{ineq1}
\end{eqnarray}
and concludes the proof of Property (A).

The key to prove~\eqref{ineq} is to observe that the function
$u\mapsto U(t;u)$ is increasing for $u \in (0,1)$, with an explicit
inverse, easily derived from the quadratic equation defining $U(u)$:
$$
u= \frac{(U(u)-t)(1-tU(u))}{tU(u)(1-t^2)}.
$$
Thus~\eqref{ineq} is equivalent to
$$
q^{j+2} <\frac{q (1-tq-t^2)}{(1-t^2)(1-tq)} <q^{j+1}.
$$
Recall that we write $t=t_j$. Hence $U(q^j)= tq/(q-t)$, so that
$$
q^j=\frac{(q-t-t^2q)}{q(1-t^2)(q-t)}.
$$
Hence we are left with proving that
$$
\frac{q(q-t-t^2q)}{q-t} <\frac{q (1-tq-t^2)}{1-tq}
<\frac{(q-t-t^2q)}{q-t},
$$
which is easily seen to hold: the first inequality boils down to
$q<1$, and the second one to $q-t-t^2q-t^3q>0$. But the
quadratic equation defining $q$ gives 
$$
q-t-t^2q-t^3q=tq(1+t)(1-2t)>0.
$$ 

It remains to prove Property (B), that is, for $k\ge j+2$,
$$
\frac{|N_k(t)|}{|N_{k-1}(t)|}=\frac{\left|\frac t{1-tq} -U(q^k)\right|}{\left|\frac {tq}{q-t}
    -U(q^k)\right|}\ \frac{|T_k(t)|}{|T_{k-1}(t)|} <1 .
$$
First,
note that the sequence $(U(q ^k))_k$ is decreasing. Hence the same holds for
$(T_k(t))_k$. Thus it suffices to show that for $ k\ge j+2$,
$$
{\left|\frac t{1-tq} -U(q^k)\right|}<{\left|\frac {tq}{q-t}
    -U(q^k)\right|}.
$$
Given~\eqref{ineq2} and~\eqref{ineq1},  we are left with proving
$$
\frac t{1-tq}<\frac {tq}{q-t},
$$
which boils down to $q^2<1$ and is immediate.

This concludes the proof of the properties of $\Tt(t;1,t)$ stated at
the beginning of the proof. These properties also hold for $\Pt(t;1)$,
as shown by the last equation of Proposition~\ref{prop-sol-3sided}.

\medskip \noindent $\bullet$
{\bf The asymptotic number of 3-sided walks.}
We have proved that $\Tt(t;1,t)$ and $\Pt(t;1)$ are meromorphic in
$\D$, with a
smallest real positive pole  at $t_0=\rho$. By Pringsheim's theorem,
$\rho$ is the radius of convergence of $\Pt(t;1)$.  In order to
prove that the coefficients of $\Pt(t;1)$  behave asymptotically as $\kappa \rho^{-n}$, we need to prove that there exists $\varepsilon >0$ such
that $\Pt(t;1)$ has no singularity
other than $\rho$ in the disk of radius $\rho+\varepsilon$.

As $\Pt(t;1)$ is meromorphic in $\D$, its poles are isolated. Thus we only have to prove that there is no other pole of modulus
$\rho$. 
Recall that all  poles $t$ of $\Pt(t;1)$ are roots of $tq/(q-t)=U(q^i)$
for some $i\ge 0$. Let $t$ be a complex number of modulus $\rho$, with
$t\not = \rho$. The positivity of the coefficients of $U(t;u)$ imply that
for $i\ge 0$,
$$
|U(t;q^i)|<U(\rho; q(\rho)^i)\le U(\rho; 1)
$$
while
$$
\left|\frac{tq}{q-t}\right|=\left|1- \frac{t\tU(t;1)}{1-t^2}\right|
>
1- \frac{\rho\tU(\rho;1)}{1-\rho^2}= U(\rho; 1).
$$
Hence $|U(q^i)|<|tq/(q-t)|$ for all $i$ and $t$ cannot be a pole.

\medskip \noindent $\bullet$
{\bf The width of the box.} The series $\Pt(t;u)$ that counts 3-sided
prudent walks by their length and the width of the box is given
by~\eqref{Ptu-expr}. A perturbation of the singularity analysis that
we have performed for  $\Pt(t;1)$ reveals that, for $u$ in some
neighborhood of $1$,   $\Pt(t;u)$ admits a unique dominant
singularity, which is a simple isolated pole obtained when
$U(u)=tq/(q-t)$, or
$$
t \left( t-1 \right) ^{3} \left( t+1 \right) ^{3}{u}^{2}+ \left( {t}^{
3}+2\,{t}^{2}-2\,t+1 \right)  \left( t-1 \right) ^{2} \left( t+1
 \right) ^{2}u-t \left( 1-2\,{t}^{2}-t+{t}^{4}+2\,{t}^{3} \right) 
.
$$
  We are again in the
meromorphic schema of~\cite[Thm.~IX.9]{flajolet-sedgewick}, and the limit
behaviour of $W_n$  follows.
\end{proof}

\section{Enumeration and asymptotic properties of triangular prudent  walks}
\label{sec:solution-triangle}
We now turn our attention to the triangular prudent walks of
Fig.~\ref{fig:triangle}.   Recall that the box of such a walk is a
triangle that points North.  If
this box has size $k$, we say that the walk \emm spans a box,
of size $k$. 
\begin{Proposition}\label{prop:triangle-simple}
  The number of triangular prudent walks that span a box of size $k$  is
$$
\tilde p _k=2^{k-1} (k+1)(k+2)!.
$$
More precisely, the number of triangular prudent walks that span a box
of size $k$, and end on the right edge of their box at distance $i$
from the North corner 
(and thus at distance $j=k-i$ from the South-West corner)
is
$$
\tilde r_{i,j}= \left\{
  \begin{array}{ll}
     \displaystyle \frac{2^k(k+2)!} 3 &\hbox{if } i=0 \hbox{ or } j=0,\\
 \displaystyle  \frac{2^k(k+2)!} 6 &\hbox{otherwise.}
  \end{array}\right.
$$
\end{Proposition}
\begin{proof}
  We specialize the functional equations of Lemma~\ref{lem:triangle} to
  $t=1$. Remarkably, the kernel of~\eqref{R-rec} reduces to
  1. Denoting $\tilde R(u,v)=\R(1;u,v)$, we have
$$ 
\tilde R(u,v)=1+2u\frac{v-2u}{v-u} \tilde R(u,u)+ 2v\frac{u-2v}{u-v} \tilde R(v,v). 
$$
This deprives us of our favourite tool (how could we cancel a kernel
that reduces to 1?), but still yields a simple solution. Let
$\tilde\R_k(u,v)$ be the homogeneous component of degree $k$ in
$\tilde R(u,v)$. With the notation $\tilde r_{i,j}$ used in the proposition,
$$
\tilde\R_k(u,v)= \sum_{i+j=k} \tilde r_{i,j} u^i v^j.
$$
The functional equation satisfied by $\tilde R(u,v)$ is equivalent to the
following recurrence of order 1:
$$
\tilde\R_k(u,v)= 2\, \frac{u(v-2u) \tilde\R_{k-1}(u,u)-v(u-2v)\tilde\R_{k-1}(v,v)}{v-u},
$$
with initial condition $\tilde\R_0(u,v)=1$. One readily checks that the
polynomial
$$
\tilde\R_k(u,v)= \frac{2^k(k+2)!} 6 \left( u^k+v^k+
  \frac{u^{k+1}-v^{k+1}}{u-v}\right)
$$
satisfies this recursion, and yields the values $\tilde r_{i,j}$ given
in the proposition.
The expression of $\tilde p_k$ then follows using the second equation
of Lemma~\ref{lem:triangle}.
\end{proof}
\noindent {\bf Remark.} The fact that the numbers $\tilde p_k$ grow faster than
exponentially is not unexpected. In particular, given a square of size $k$, 
the number of partially directed (1-sided) walks  of height $k$ that
fit in this square is $(k+1)^{k+2}$.  Indeed, such walks are
completely determined by choosing
the abscissas of the $k$ vertical steps, of the starting point and
of the endpoint.

 We now move to the length enumeration of triangular prudent walks.
\begin{Proposition}
\label{prop-sol-triangle}
Set $u=\frac{x(1-t)}{(1+tx)(1+t^2x)}$, where $x$ is a new variable.
The \gf\ of triangular prudent walks ending on the right edge of their
 box  satisfies
\beq\label{R-sol}
\R(t; u,tu) = (1+xt)(1+xt^2) \sum_{k \ge 0} \frac{t^{{k+1}\choose 2}
  \left( xt(1-2t^2)\right)^k }{(xt(1-2t^2);t)_{k+1}}\left(
\frac{xt^3}{1-2t^2}; t\right)_k
\eeq
where we have used the standard notation 
$$
(a;q)_n=(1-a)(1-aq)\cdots (1-aq^{n-1}). 
$$
The \gf \ of triangular prudent walks, counted by their length and the size
of the  box, is
$$
\Ptr(t;u)=1+\frac{6tu(1+t)}{1-t-2tu(1+t)}
\big(1+t\left(2u(1+t)-1\right)\R(t;u,tu)\big).
$$
\end{Proposition}

\noindent {\bf Comments}\\
\noindent 1. The parametrisation of $u$ in terms of the length variable $t$ and another
variable $x$ is just a convenient way to write
$\R(t;u,tu)$. Equivalently, we have
$$
\R(t; u,tu) = (1+Xt)(1+Xt^2) \sum_{k \ge 0} \frac{t^{{k+1}\choose 2}
  \left( Xt(1-2t^2)\right)^k }{(Xt(1-2t^2);t)_{k+1}}\left(
\frac{Xt^3}{1-2t^2}; t\right)_k
$$
where 
$$
X\equiv X(u)
=
\frac{1-t-ut-ut^2-\sqrt{(1-t)(1-t-2u   t-2ut^2+u^2t^2-u^2t^3)}}
{2ut^3}
$$
is the only power series in $t$ satisfying $u(1+tX)(1+t^2X)=X(1-t)$.  
In particular, the length  \gf\ of triangular prudent walks is
\beq\label{Ptr-length-sol}
\Ptr(t;1)=\frac{6t(1+t)}{1-3t-2t^2}
\big(1+t\left(1+2t\right)\R(t;1,t)\big)
\eeq
where
\beq\label{R1t-sol}
\R(t;1,t) = (1+Y)(1+tY) \sum_{k \ge 0} \frac{t^{{k+1}\choose 2}
  \left( Y(1-2t^2)\right)^k }{(Y(1-2t^2);t)_{k+1}}\left(
\frac{Yt^2}{1-2t^2}; t\right)_k
\eeq
and 
\beq\label{Y-sol}
Y=t X(1)
=
\frac{1-2t-t^2-\sqrt{(1-t)(1-3t-t^2-t^3)}}
{2t^2}.
\eeq
Note that 
\beq \label{Y-alg}
Y=\frac t {1-t} (1+Y)(1+tY).
\eeq

\noindent 2. The value of $\R(u,tu)$ is especially simple when $x=1/(1-2t^2)$,
   that is, for
$u=\frac{1-2t ^2}{(1-t^2)(1+2t)}$. Indeed, for this value of $u$,
   \begin{multline*}
     \R(t;u,tu) = \frac{(1+t-2t^2)(1-t^2)}{(1-2t^2)^2}
 \sum_{k \ge 0} \frac{t^{k+{{k+1}\choose 2}}
  }{(t;t)_{k+1}}\left(
\frac{t^3}{(1-2t^2)^2}; t\right)_k
\\= \frac{1-t}{t(1-2t)} \left( -1+
\prod_{m\ge 1} (1+t^m) \left(1-\frac{t^{2m+1}}{(1-2t^2)^2}\right)\right).
   \end{multline*}
The product form follows from the following
identity~\cite[Corollary~2.7]{andrews}:
\beq\label{phi20}
\sum_{n\ge 0} \frac{t^{{n+1} \choose 2} (a;t)_n}{(t;t)_n}= \prod_{m\ge
  1} (1+t^m) (1-at^{2m-1}).
\eeq
This can be used to prove that neither $\R(t;u,tu)$ nor  $\Ptr(t;u)$ are
D-finite. However, we will derive
from~(\ref{Ptr-length-sol}--\ref{R1t-sol})  the 
following finer result on the length \gf\ $\Ptr(t;1)$. 

\begin{figure}[thb]
\includegraphics[height=3cm,width=3cm]{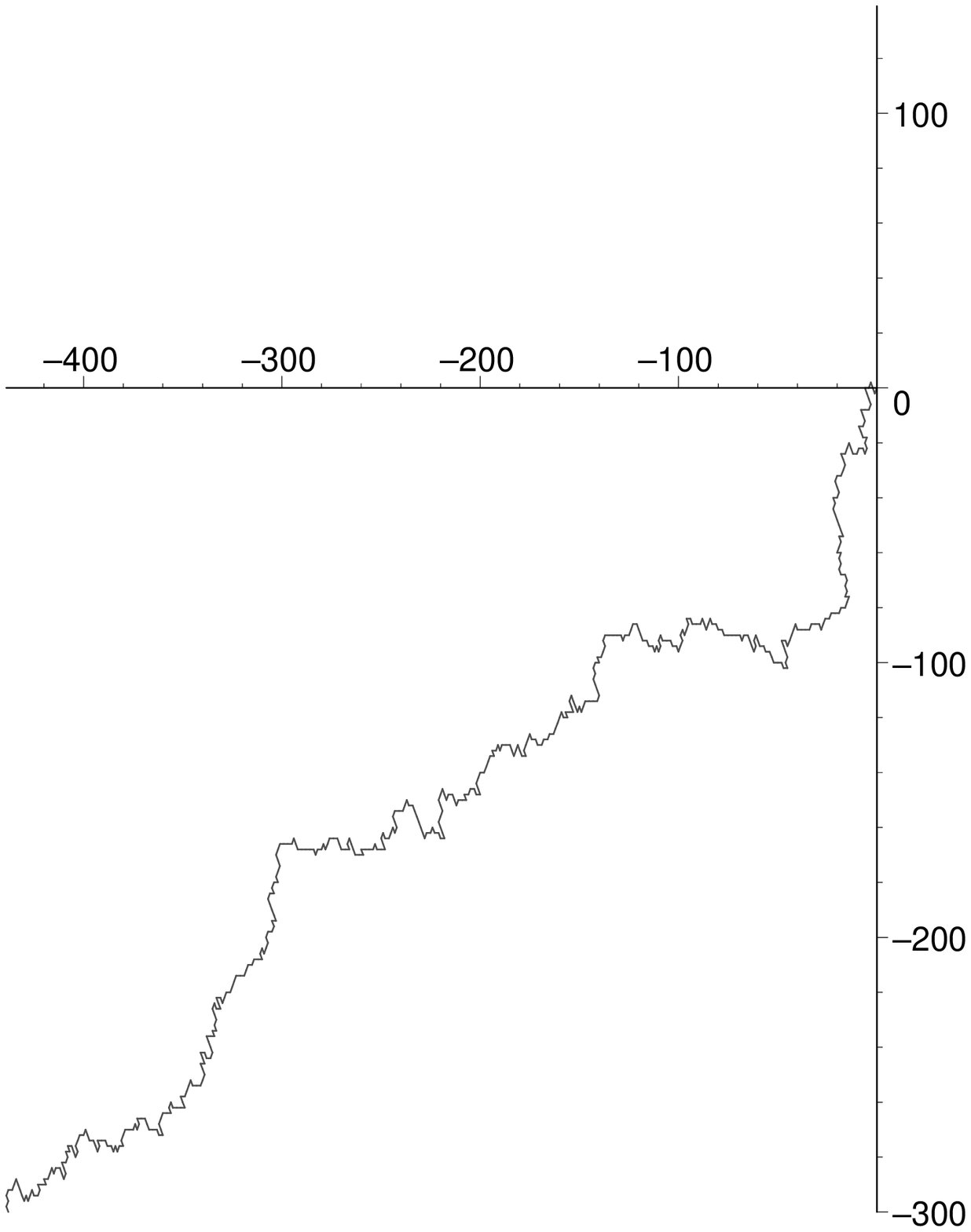}
\hskip 10mm
\includegraphics[height=3cm,width=3cm]{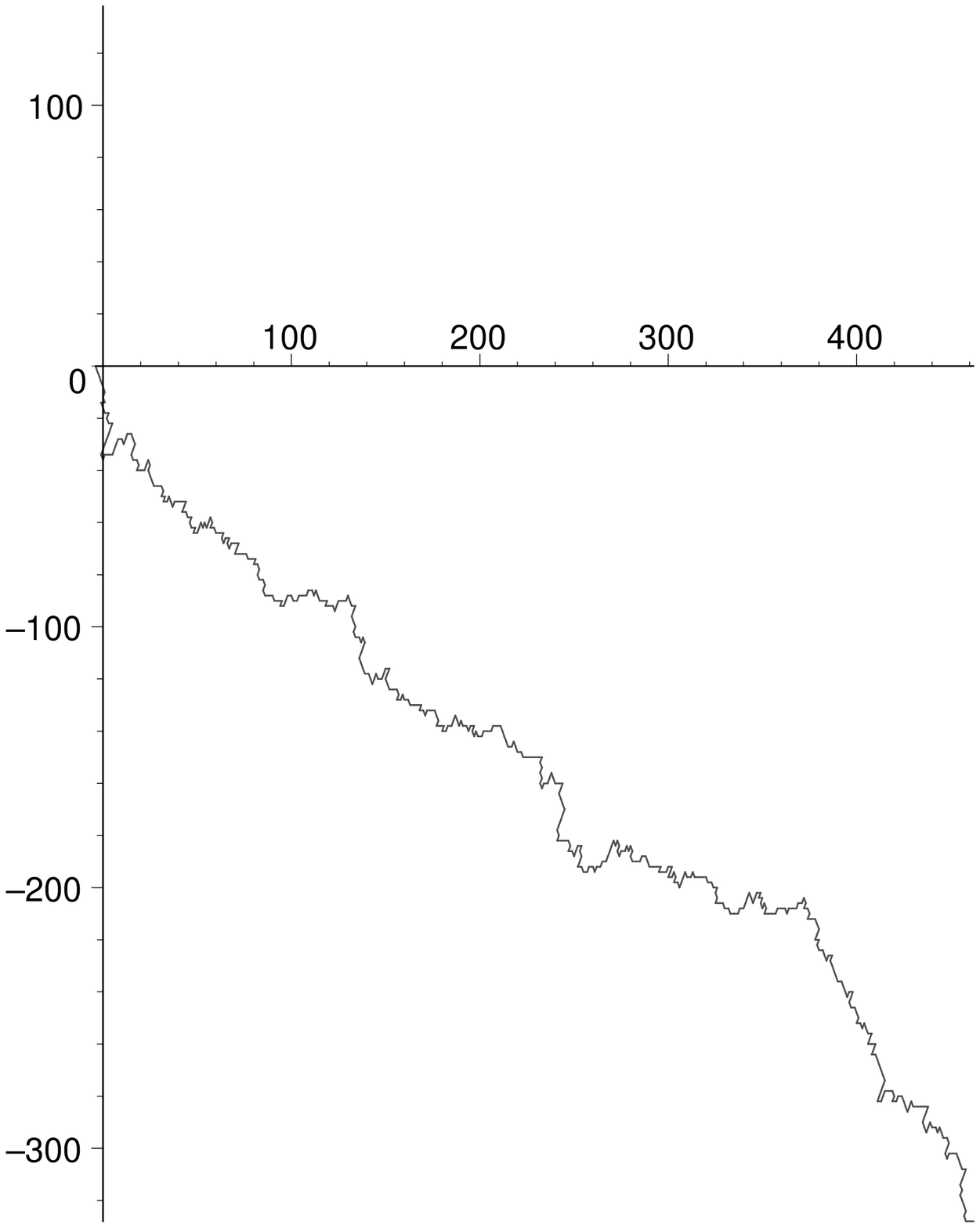}
\hskip 10mm
\includegraphics[height=3cm,width=3cm]{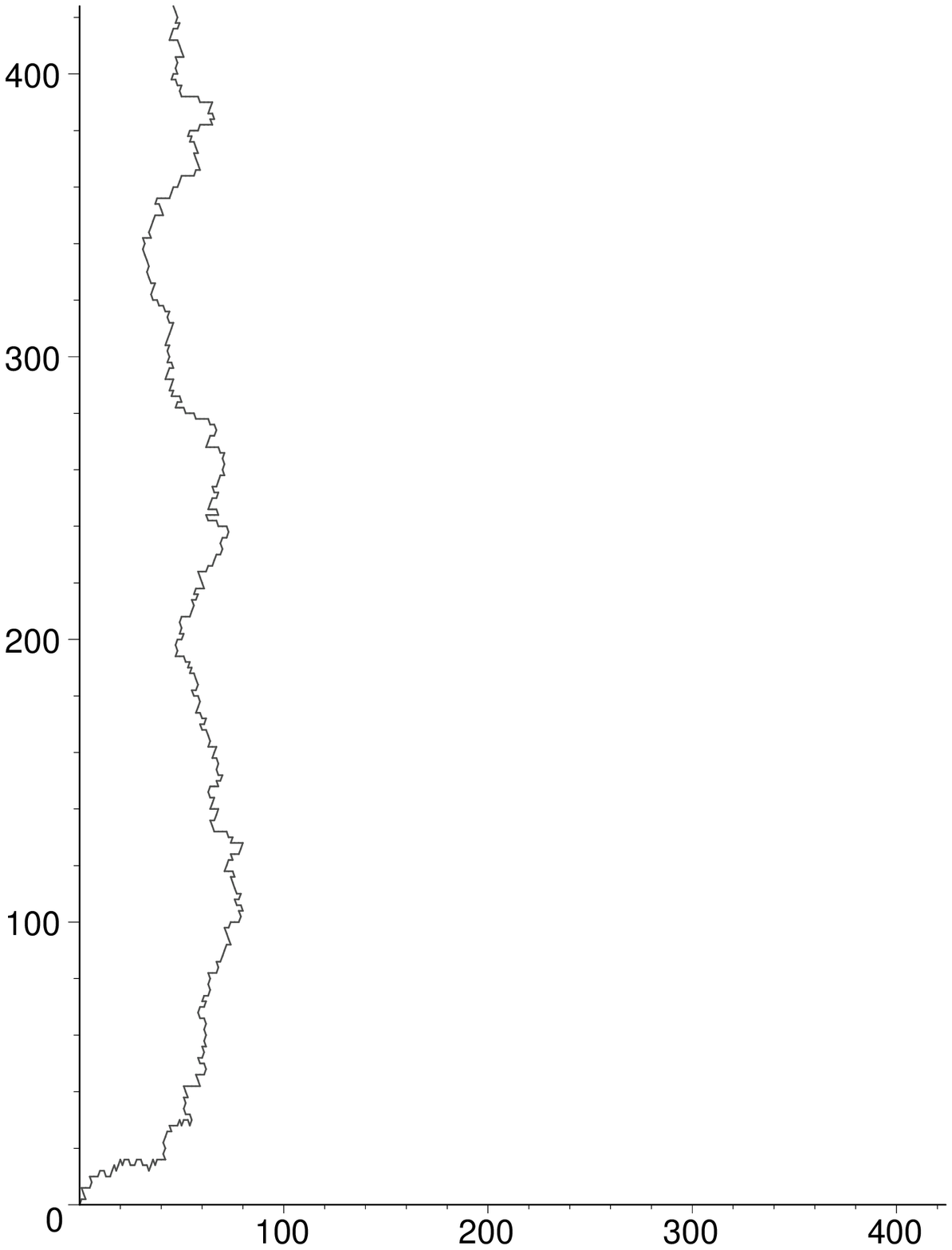} 
\caption{Random triangular prudent walks of length 500.}
\label{fig:random-t}
\end{figure}

\begin{Proposition}[{\bf Nature of the g.f. and asymptotic properties
    of triangular prudent walks}]
\label{prop-nature-triangle}
The length \gf\ $\Ptr(t;1)$ of triangular prudent walks is meromorphic in the
domain $\D=\{t:|t|<1\}\setminus [t_c, 1)$, where 
$t_c\simeq 0.295...$  is the real root of $1-3t-t^2-t^3$. 
In this domain, it has  infinitely many poles, so that it cannot be
D-finite. 
These poles accumulate on a portion of the unit circle, so that
$\Ptr(t;1)$ has a natural boundary.

The series  $\Ptr(t;1)$ has a
  unique dominant singularity, which is a simple pole $\rho= (\sqrt{17}-3)/4\simeq
  0.280$. 
Hence the number of triangular  prudent walks of length $n$ satisfies
$$
p_n \sim \kappa \left(\frac{3+\sqrt{17}}2\right)^n
$$
for some positive constant $\kappa$.

Let $S_n$  denote the size of the box  of a random $n$-step prudent
walk. Then the mean and  variance of $S_n$ satisfy
$$
\E(S_n) \sim  \frak m  \, n,  \quad  \quad
\Var(S_n) \sim  \frak s ^2 \, n ,  
$$
where 
$$
 \frak m = \frac 1 2 \left(1+\frac 1{\sqrt{17}}\right) 
\quad \hbox{and} \quad
\frak s ^2=\frac {12}{17\sqrt {17}},
$$
 and the variable
$
\frac{S_n-  \frak m \, n}{\frak s \sqrt n}
$
converges in law to  a standard normal distribution.

\end{Proposition}

\noindent \emm Proof of Proposition,~\ref{prop-sol-triangle}.
 The kernel of the functional equation~\eqref{R-rec} reads:
$$
K(u,v)= (u-tv)(v-tu)-tuv(1-t^2)(u+v).
$$
Again we want to cancel it by an appropriate choice of $u$ and
$v$. The kernel is not homogeneous in $u$ and $v$, as it was in the
case of 3-sided walks, but it has another interesting property:
 the curve $K(u,v)=0$ has genus 0,
and  thus admits a  rational parametrisation, namely
$$ 
K(U(x),U(tx))=0 \quad \hbox{ for } \quad U(x)
=\frac{x(1-t)}{(1+tx)(1+t^2x)}.
$$

Now take an indeterminate $x$, and set $u=U(x)$ and $v=U(tx)$. The
series $\R(u,v)$ is well-defined (as a series in $t$ with
coefficients in $\rs[x]$). Since the kernel vanishes, it follows
from~\eqref{R-rec} that
$$
1+tu(1+t)\frac{v-2tu}{v-tu} \R(u,tu)
+ tv(1+t)\frac{u-2tv}{u-tv} \R(tv,v)=0.
$$
Recall that $\R(u,v)=\R(v,u)$ for all $u$ and $v$. 
Denoting $\Phi(x)= \R(u(x),tu(x))$, this equation can be rewritten as
$$
\Phi(x)= \frac{(1+xt)(1+xt^2)}{1-xt(1-2t^2)} 
+ \frac{xt^2(1+xt)(1-2t^2-xt^3)}{(1+xt^3)(1-xt(1-2t^2))} \Phi(xt).
$$
Iterating it gives the value~\eqref{R-sol} of $\R(u,tu)$.

The expression of $\Ptr(t;u)$ in terms of $\R(u,tu)$ follows
from~\eqref{Ptr-R}, after specializing~\eqref{R-rec} to $v=u$ and then $v=0$.
\qed

\medskip
\noindent{\bf Where is the group?} We have solved a new linear
equation with two catalytic variables $u$ and $v$. Again, two
transformations  $\Phi$ and $\Psi$  leave $K(u,v)/u/v$ unchanged:
$$
\Phi(u,v)=\left(\frac{v^2}{u(1+v-t^2v)},v\right),
\quad \Psi(u,v)=\left(u,\frac{u^2}{v(1+u-t^2u)}\right).
$$
They generate an infinite group. Indeed, if one  repeatedly applies
$\Psi\circ\Phi$  to the pair $(U(x),U(tx))$, one obtains all pairs
$(U(t^{2i}x),U(t^{2i+1}x))$ for $i \ge 0$.
 We now proceed to show that the
series $\Ptr(t;1)$ is not D-finite.

\bigskip
\noindent \emm Proof of Proposition,~\ref{prop-nature-triangle}.
Recall the expression~\eqref{Ptr-length-sol} of $\Ptr(t;1)$. The
denominator $1-3t-2t^2$ vanishes at $\rho$ and will be responsible for
the simple dominant pole of $\Ptr(t;1)$. We will see that the series
$\R(t;1,t)\equiv\R(1,t)$, given by~\eqref{R1t-sol}, has a larger
radius of convergence than  $\Ptr(t;1)$. More precisely, we will prove
that $\R(1,t)$ is meromorphic in $\D$, with a unique real pole at
 $t_0\simeq 0.288 \in (\rho,t_c)$, where 
$4\,{t_0}^{4}+2\,{t_0}^{3}-6\,{t_0}^{2}-2\,t_0+1=0$, and infinitely many
non-real poles that accumulate on a portion of the unit circle.

\medskip \noindent 
$\bullet$ {\bf The series $Y(t)$.} 
The square root occurring in the  series $Y$ given by~\eqref{Y-sol} vanishes
at $t=1$, at $t=t_c\simeq 0.295$ and at two other points of modulus
$1.83...>1$. Hence $Y$ has radius $t_c$, but can be  analytically
continued  in the domain $\D$. In this domain, $Y(t)$ is
bounded. Moreover, it is easily seen to be increasing on the segment
$(-1,t_c)$. 



\medskip \noindent 
$\bullet$ {\bf The series  $\R(1,t)$ is meromorphic in $\D$.}
Let us write 
$$
R(1,t)= \frac{N(t)}{D(t)}
$$
where
$$D(t)= (Y(1-2t^2);t)_\infty = \prod_{i\ge 0} (1-Yt^i(1-2t^2))
$$
and
$$
N(t)=(1+Y)(1+tY) \sum_{k \ge 0} {t^{{k+1}\choose 2}
  \left( Y(1-2t^2)\right)^k }\left( 
\frac{Yt^2}{1-2t^2}; t\right)_k
{(Yt^{k+1}(1-2t^2);t)_\infty}.
$$
Clearly $N(t)$ and $D(t)$ are analytic in $\D$ (as $|t|<1$).
 The cancellation of $(1-2t^2)$ does not create any singularity, as 
$$ 
(1-2t^2)^k \left(\frac{Yt^2}{1-2t^2}; t\right)_k
=(1-2t^2-Yt^2)(1-2t^2-Yt^3)\cdots (1-2t^2-Yt^{k+1}).
$$
Hence $\R(1,t)$ is meromorphic in $\D$, and its
poles are found among the values of $t$ such that $Y(t)t^\ell(1-2t^2)=1$ for
some $\ell \in \ns$. All poles are simple, as $Y(t)t^\ell(1-2t^2)=1$
implies  $Y(t)t^m(1-2t^2)=t^{m-\ell}\not =1$ for $m\not = \ell$.

\medskip \noindent 
$\bullet$ {\bf Real poles of  $\R(1,t)$.}
A standard study of the functions $t\mapsto Y(t)$ and $t\mapsto
1/t^\ell/(1-2t^2)$  on the interval $(-1, t_c)$ reveals that the only
possible real pole  of  $\R(1,t)$ is at $t_0\simeq 0.288 <t_c$,
where $Y(t_0)=1/(1-2t_0^2)$. This implies 
$4\,{t_0}^{4}+2\,{t_0}^{3}-6\,{t_0}^{2}-2\,t_0+1=0$.

\begin{figure}[b!]
\begin{center}
\includegraphics[height=3cm] {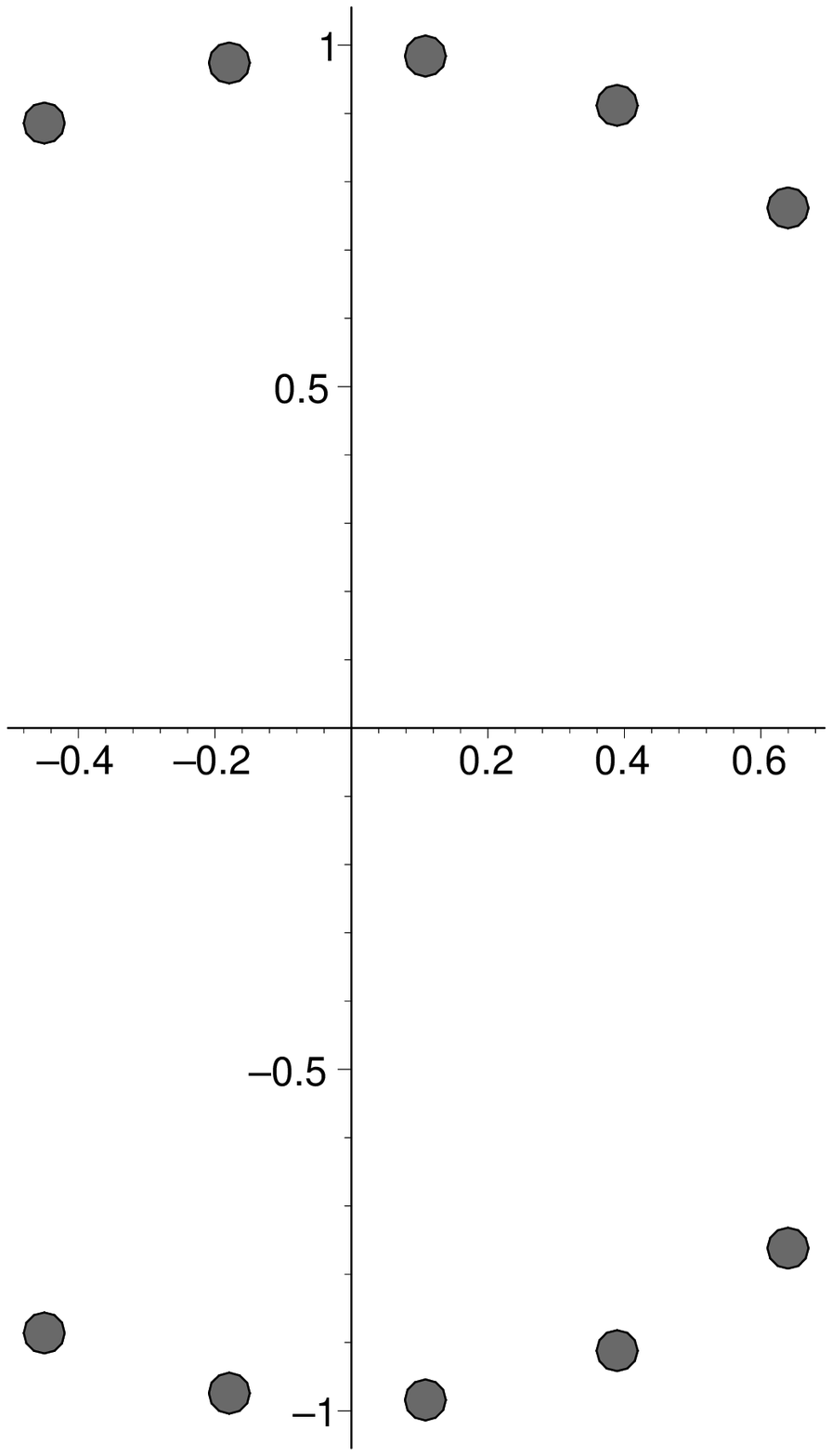}
\hskip 20mm
\includegraphics[height=3cm] {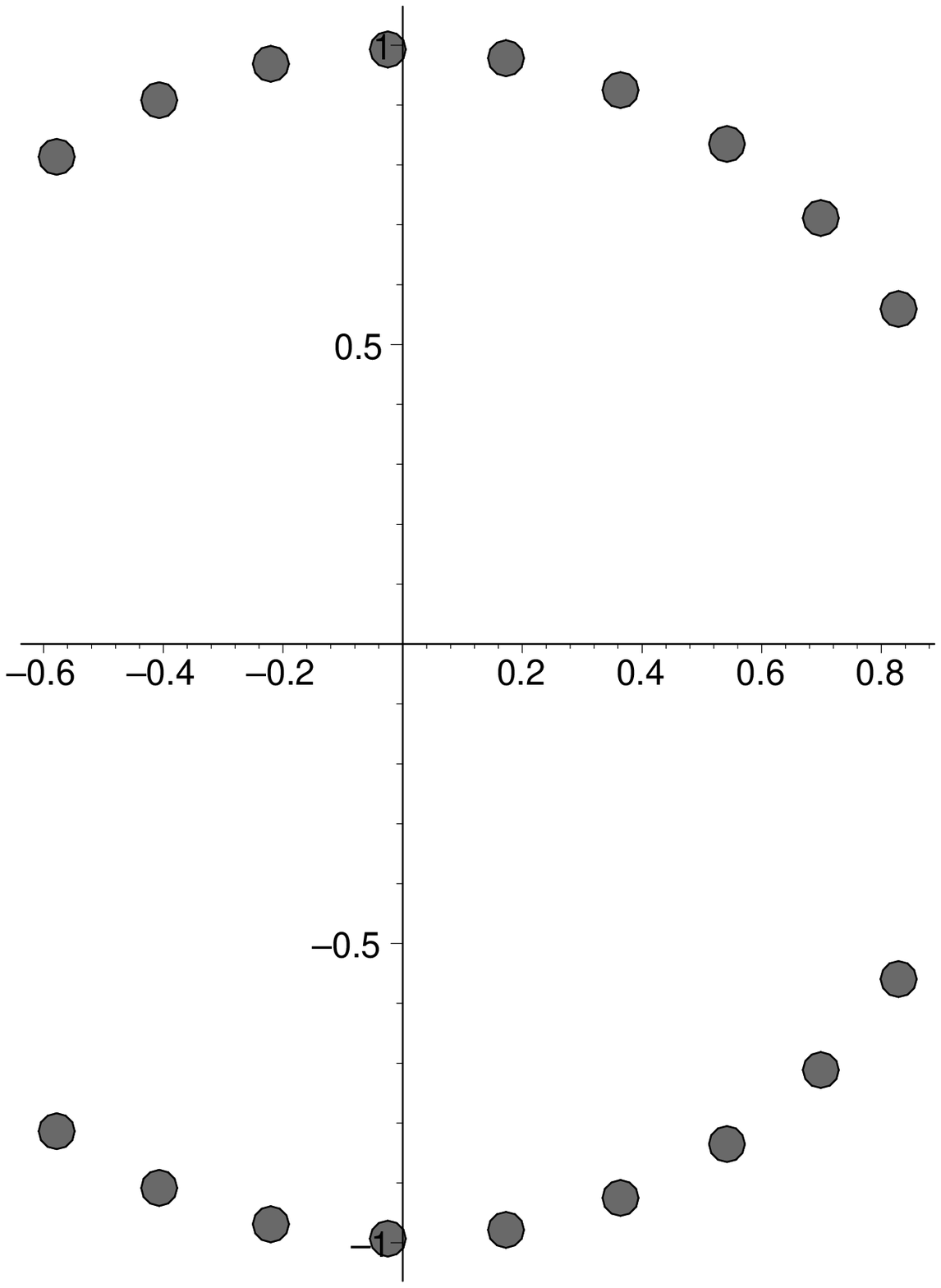}
\hskip 20mm
\includegraphics[height=3cm] {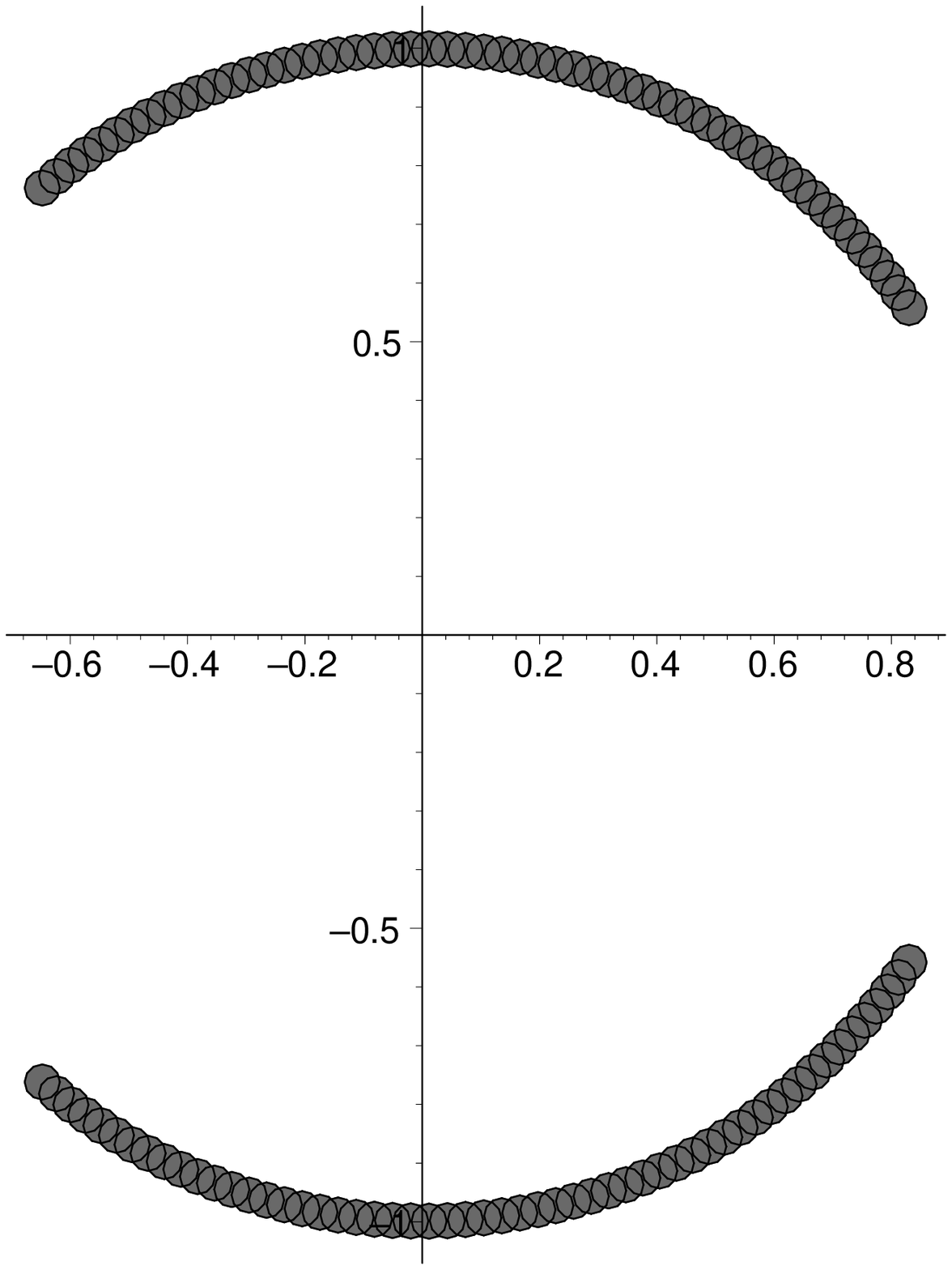}
\end{center}
\caption{The  roots of
  $Y(t)t^\ell(1-2t^2)=1$ lying inside the domain   $\D$ for $\ell=20$,
  $\ell=30$, $\ell=200$.  }
\label{fig:critical}
\end{figure}

\medskip \noindent 
$\bullet$ {\bf Existence of infinitely many poles.}
Let us call a \emm critical value, any $t\in\D$ satisfying  $Y(t)t^\ell(1-2t^2)=1$
for some $\ell$. We have just exhibited a real critical value $t_0$.
Fig.~\ref{fig:critical} shows the critical values obtained for
various values of $\ell$.
We will first prove that $\D$ contains infinitely many critical
values. Then, we will show that almost of them, and in particular
$t_0$, are poles of  $\R(1,t)$.

To prove the former point, we prove that there exist
$0<\theta_0<\theta_1<\pi$ such that  every point $e^{i\theta}$ with
$\theta_0<|\theta|<\theta_1$ is an accumulation point of critical
values. As $Y(\bar z)= \overline{Y(z)}$, it suffices to consider the
case $\theta>0$.
Let   $\theta\in(0,\pi)$ be of the form  $p \pi/q$ for two integers
$p$ and $q$. Let $\ell=2kq$ with $k\ge 0$. Take $t$ in the vicinity of
$e^{i\theta}$, that is $t=e^{i\theta}(1+s)$ with $s$ small. As $Y$ is
analytic in the neighbourhood of $e^ {i\theta}$,
 $$
Y(t)t^\ell(1-2t^2)= Y(e^{i\theta  }) (1-2e^{2i\theta})
\exp(\ell s + O(\ell s^2))(1+O(s)).
$$
This shows that there exists, in the neighbourhood of $e^{i\theta}$,  a
solution $t$ of $Y(t)t^\ell(1-2t^2)=1$  for every (large)
$\ell=2kq$. This solution reads $t=e^{i\theta}(1+s)$ with 
$$
s\sim -\frac 1 \ell \log \left((1-2e^{2i\theta})Y(e^{i\theta})\right).
$$
However,  $t$ will only be critical if it lies in $\D$, that
is, if $|t|<1$. But 
$$
|t|^2= 1-\frac 2 \ell \log \left(|1-2e^{2i\theta}|| Y(e^{i\theta})|\right)+o(\ell^{-1}),
$$
so that $t$ is critical (for $\ell$ large) if and only if $f(\theta):=|1-2e^{2i\theta}||
Y(e^{i\theta})|>1$. The function $f$ is continuous on
$(0,\pi)$. Moreover, at $t= e^{i\pi/2}$, one has $(1-2t^2)=3$ and
$Y(t)=-1+i\pm e^{-i\pi/4}$, so that $f(\pi/2)>1$. This proves the
existence of an interval $(\theta_0,\theta_1)$ where every point
$\theta$ satisfies $f(\theta)>1$ and is thus an accumulation
point of critical values. Fig.~\ref{fig:ftheta} shows a plot of $f(\theta)$,
and gives estimates for the best possible values $\theta_0\simeq 0.59$
and $\theta_1\simeq2.28$, in good agreement with
Fig.~\ref{fig:critical}.

\begin{figure}[hbt]
\begin{center}
\includegraphics[height=35mm,width=35mm]{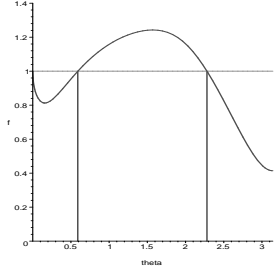}
\end{center}
\caption{The function $f(\theta):=|1-2e^{2i\theta}||
Y(e^{i\theta})|$.}
\label{fig:ftheta}
\end{figure}

Let $t\in \D$ satisfy $Y(t)t^\ell(1-2t^2) =1$. Then $t$ is a pole of
$\R(1,t)$ if and only if $N(t)\not = 0$.
But
$$
N(t)=(1+Y)(1+tY) \sum_{k \ge 0} {t^{{k+1}\choose 2} t^{-k\ell} }\left( 
\frac{t^{2-\ell}}{(1-2t^2)^2}; t\right)_k
{(t^{k+1-\ell};t)_\infty}.
$$
If $k<\ell$, the $k$th summand is zero because of the term
$(t^{k+1-\ell};t)_\infty$.  We thus take $k=  \ell+j$, with $j\ge
0$. We also rewrite $(1+Y)(1+tY)$ using~\eqref{Y-alg}. This gives
\begin{multline*}
N(t)=
t^{-1-{{\ell+1} \choose 2}} \frac{1-t}{1-2t^2}
\left(\frac{t^{2-\ell}}{(1-2t^2)^2};y\right)_\ell\ 
(t;t)_\infty\sum_{j\ge 0} \frac{t^{{j+1}\choose 2}}{(t;t)_j}\left(
\frac{t^2}{(1-2t^2)^2};t\right)_j
\\=
 t^{-1-{{\ell+1} \choose 2}} \frac{1-t}{1-2t^2}
\left(\frac{t^{2-\ell}}{(1-2t^2)^2};t\right)_\ell
(t;t)_\infty \prod_{m\ge 1} (1+t^m) \left(1 -\frac{t^{2m+1}}{(1-2t^2)^2}\right).
\end{multline*}
The final product expression relies on~\eqref{phi20}.

Could this product vanish for our critical value $t\in \D$?
As $|t|<1$,  the only factors
that might be zero are of the form $1-t^j /(1-2t^2)^2$, for $j \ge
2-\ell$. Each such factor has only a finite number of zeroes in
$\D$. In particular, those for which $j<0$, taken together, will only vanish for a
\emm finite, number of critical values, while there are infinitely
many  such values. For $j\ge 0$, the roots of $1-t^j /(1-2t^2)^2$ that lie in the unit disk
satisfy $|1-2t^2|<1$. That is, they are inside  the 'glasses' curve $|1-2t^2|=1$
plotted on Fig.~\ref{fig:lunettes}. In particular, they remain away from the
accumulation points we have exhibited on the unit circle, so that only
finitely many critical values are non-poles.

Could $t_0$, the unique critical value found on the segment $(-1,t_c)$,
not be a pole? For $t=t_0$, one has
$$
\begin{array}{ll}
  \frac{t^j}{(1-2t^2)^2}\ge \frac{1}{(1-2t^2)^2}>1 &\hbox{ for } j \le 0,\\
\frac{t^j}{(1-2t^2)^2}\le \frac{t}{(1-2t^2)^2}\simeq 0.4 <1 &\hbox{
  for } j\ge 1.
\end{array}$$
 Hence $N(t_0)\not =0$, so that $t_0$ is indeed a pole of $\R(1,t)$, and,
 by Pringsheim's theorem,  its radius of convergence.

\begin{figure}
\begin{center}
\includegraphics[height=35mm,width=35mm]{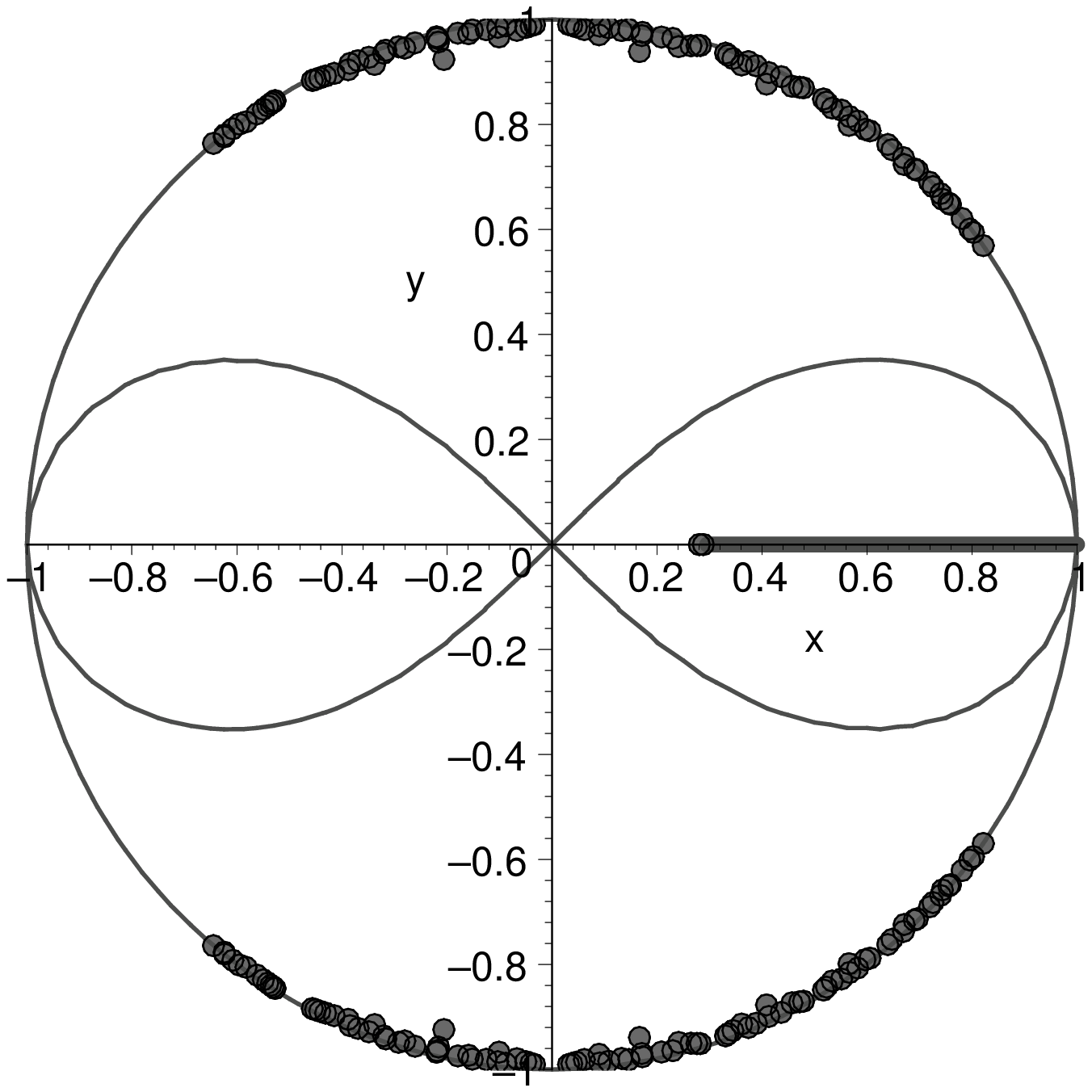}
\hskip 20mm
\includegraphics[height=35mm,width=35mm]{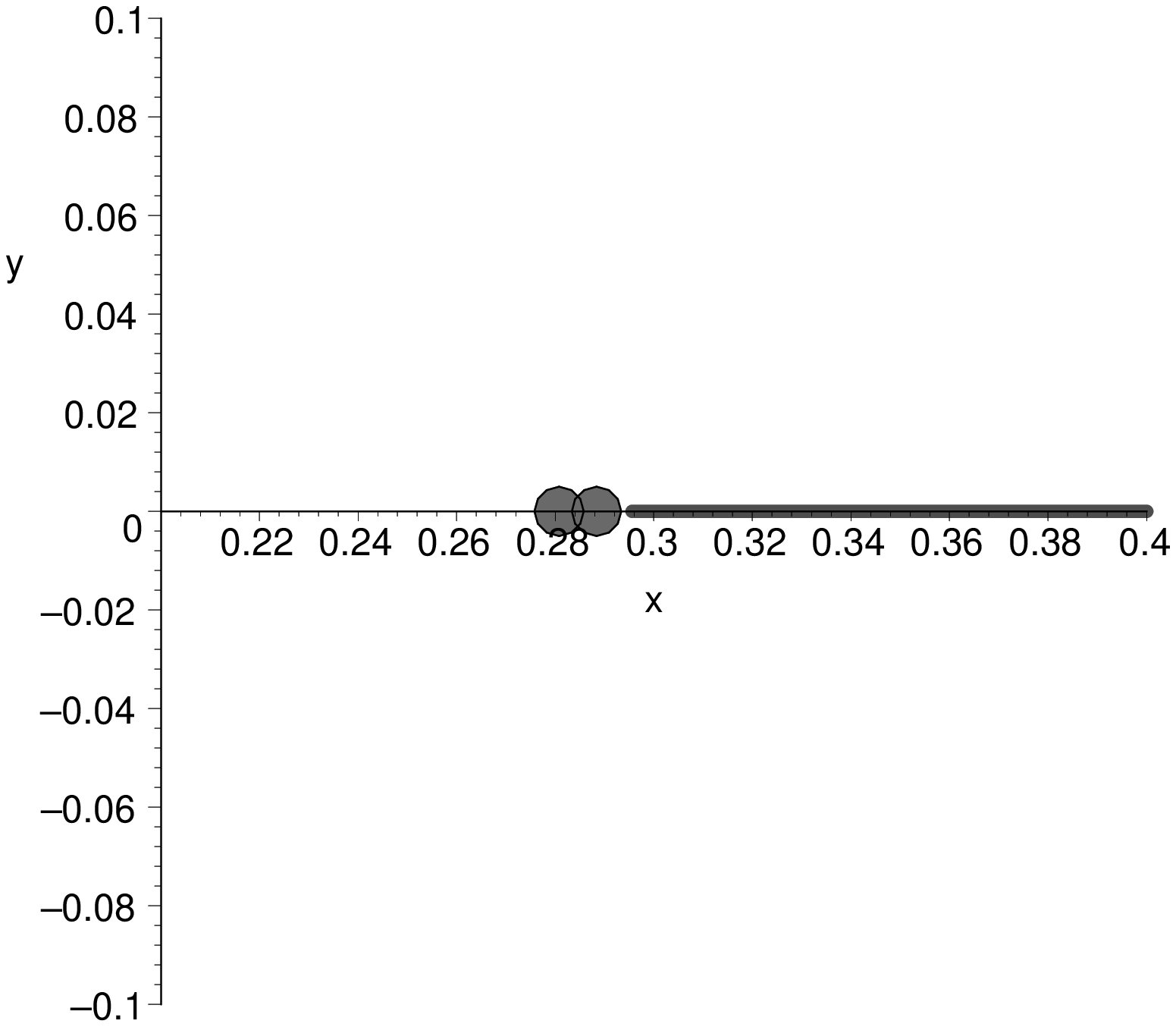}
\end{center}
\caption{The singular landscape of $\Ptr(t;1)$: this series is
  meromorphic inside the unit disk except on a cut $[t_c,1)$, with
  $t_c\simeq 0.295$. There 
  are two simple poles on the real axis before $t_c$, namely
  $\rho\simeq 0.280$ and $t_0\simeq 0.288$. Almost all critical values, which
  accumulate on a portion of the unit circle, are simple poles. The
  'glasses'  curve is  $|1-2t^2|=1$. To the right, a zoom in the
  neighbourhood of 0.28.}
\label{fig:lunettes}
\end{figure}

\medskip \noindent 
$\bullet$ {\bf The singular landscape of $\R(1,t)$ and $\Ptr(t;1)$.}
We can now conclude as to the nature of the series $\R(1,t)$:  this series
is meromorphic in the domain $\D$, with a  dominant pole at
$t_0$ and infinitely many poles that accumulate on a portion of the
unit circle, $\{e^{i\theta} : \theta_0 < |\theta|<\theta_1\}$.

In view of~\eqref{Ptr-length-sol}, the same holds for the series
$\Ptr(t;1)$, with 
$t_0$ replaced by $\rho$ as the denominator $1-3t-2t^2$ induces a new
pole $\rho<t_0$. This is illustrated in Fig.~\ref{fig:lunettes}. Clearly, $\Ptr(t;1)$ has no other pole of
modulus less that $t_0$. Hence the number of $n$-step triangular prudent
walks is, as announced, asymptotic to $\kappa \rho^{-n}$.

\medskip \noindent 
$\bullet$ {\bf The size of the box.}
Consider the expression of
$\Ptr(t;u)$ given in  Proposition~\ref{prop-sol-triangle}.
The singular pattern we have found for $u=1$ persists when $u$ is in a
neighborhood of 1. The radius of $R(u,tu)$ is reached when
$Xt(1-2t^2)=1$, but the pole of $\Ptr(t;u)$ given by the 
cancellation of $1-t-2tu(1+t)$ is smaller in modulus. We are again in the
meromorphic schema of~\cite[Thm.~IX.9]{flajolet-sedgewick}, and the limit
behaviour of $S_n$  follows.

\qed

\section{Random generation}
\label{sec:random}
\subsection{Generating trees and random generation}
A \emm generating tree, is a rooted tree with labelled nodes satisfying
the following property: if two nodes have the same label, the 
multisets of labels of their children are the same. 
Given a label $\ell$, the multiset of labels of the children of a node
labelled $\ell$ is denoted $\De(\ell)$. The rule that
gives  $\De(\ell)$ as a function of $\ell$  is called the
\emm rewriting rule, of the tree. 

Consider a class $\C$ of walks  closed by taking
prefixes: if $w$ belongs to $\C$, then so do all prefixes of
$w$ (this is the case for our classes of prudent walks). One can then
display the elements of $\C$ as the nodes of a  generating tree: 
 the empty walk sits at the root of the tree, and the 
children of the node labelled by $w$  are the walks of the form
$ws$ belonging to $\C$, where $s$ is a step. Fig.~\ref{fig:tree} (top) shows the
first few levels of the generating tree of 2-sided walks. An even simpler example is the tree of \emm directed walks, (walks with N and E
steps), which has the following rewriting rule:
$$
w \rightarrow 
\left\{
\begin{array}{lll}
wN \\
wS
\end{array}
\right. .
$$

\begin{figure}[pbth]
  \begin{center}
  \scalebox{0.42}{   \input{Figures/generating-tree.pstex_t}}
    \caption{The generating tree of 2-sided walks (top) and one of its good
    labellings (bottom).}
    \label{fig:tree}
  \end{center}
\end{figure}
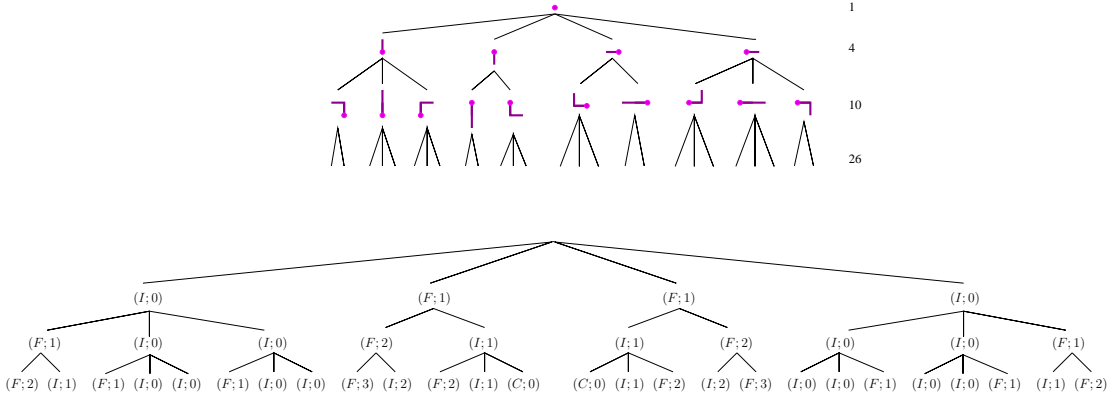

In this context, the (uniform) random generation of a walk of $\C$ of size $n$
is equivalent to the random generation of a path of length $n$ in the
tree, starting from the root. This path can be built recursively as follows:
one starts from the empty walk (the root of the tree), and builds the
walk step by step, choosing each new step \emm with the right
probability,. If at some point one has obtained a walk $w$ of
length $n-m$, with $m>0$,  then
the probability of choosing $s$ as the next step must be
$$
\GP(s|w)= \frac{\Ex(ws,m-1)}{\Ex(w,m)}
$$
where $\Ex(w,m)$ is the number of \emm extensions of $w$, of length
$m$, that is,  of walks $w'$ of length $m$ such that the concatenation
$ww'$ belongs to $\C$. Clearly, $\Ex(w,m)=0$ if $w\not \in \C$.
For $w \in \C$, the numbers $\Ex(w,m)$ can be computed inductively: 
$$
\Ex(w,m)=
\left\{
\begin{array}{llll}
1 & \hbox{if } 
m=0,\\
\sum_{ws \in \De(w)} \Ex(ws,m-1) &\hbox{otherwise.}
\end{array}
\right.
$$
This is the basic principle of  recursive random
generation, initiated in~\cite{nijenhuis-wilf-book}.
The (time and space) complexity of the  procedure that generates an
$n$-step walk of $\C$ depends on how
efficiently one solves  the following two problems:
\begin{itemize}
\item [$(1)$] compute the numbers $\Ex(w,m)$, for all walks $w$ of
  length $n-m$,
  \item[$(2)$] given a walk $w$ of length less than $n$, decide which
  steps $s$ are \emm   admissible,, that is, are such that   $ws$
  belongs to $\De(w)$.
\end{itemize}
We will explain how both tasks can be  achieved by introducing
  \emm good labellings, of the class $\C$.  A function $L$ defined on
  $\C$ is a  good labelling if for all $w \in \C$, the multiset 
$\{L(w'): w' \in \De(w)\}$ depends only on $L(w)$.
 This means that the tree obtained by replacing each node 
$w$ of the generating tree by $L(w)$ is itself a generating
tree. We will illustrate this
  discussion with the class  of 1-sided walks. 

\noindent{\bf (1) The numbers  $\Ex(w,m)$.} 
When $\C$ is the class of 1-sided walks, these numbers only depend on
the \emm direction, (horizontal or   vertical, $H$ or $V$ for short) of
  the final step of $w$ (and of course on $m$).   Denote
  this direction by $L(w)$. Then it is easy to see that $L$ is
  a  good labelling:   the  tree obtained by replacing each node  
$w$ of the generating tree by $L(w)$ is itself a generating
tree, with rewriting rule
\beq\label{rule-VH}
V \rightarrow 
\left\{
\begin{array}{lll}
V\\
H\\
H
\end{array}
\right. ,
\quad 
H \rightarrow 
\left\{
\begin{array}{lll}
V\\
H
\end{array}
\right. .
\eeq
(By convention, the label of the empty walk is set to $V$.)

For a general class $\C$, any good labelling $L$  allows us to
compute the numbers $\Ex(w,m)$. Indeed, they can be   rewritten as
$\Ex(L(w),m)$, with 
\beq\label{Ex-rec-L}
\Ex(\ell,m)= \left\{
\begin{array}{llll}
1 & \hbox{if }  m=0,\\
\sum_{\ell' \in \De(\ell)} \Ex(\ell',m-1) &\hbox{otherwise.}
\end{array}
\right.
\eeq
In the case of partially directed walks, we have
$$
\Ex(V,m)=\Ex(V,m-1)+ 2\,\Ex(H,m-1),
\quad 
\Ex(H,m)=\Ex(V,m-1)+\Ex(H,m-1).
$$ 
A good labelling is \emm compact, if it takes its values in $\zs^k$,
with $k$ fixed. In this case, it is usually easy to determine
$\De(\ell)$, given $\ell$. In order to fasten  the computation of the
extension numbers, one tries to minimize the number of distinct labels $L(w)$
occurring in the first $n$ levels of the generating tree. This 
solves Problem~(1) above.

\noindent{\bf (2) Admissible steps.}
We are left with the second problem: determine the admissible
steps. Moreover, we also need to decide, for every admissible step
$s$, which of the elements of $\De(L(w))$ is the label of $ws$. That
is, we need a way to \emm correlate,  new  steps with  new
values of the label. 
In our  random generation procedures, this will be achieved thanks to
a second good labelling $P$ having the following properties:
\begin{itemize}
\item[--] it refines the first one; that is, $P(w)=(L(w) | M(w))$, and the
rewriting rule of $P$ extends
 that of $L$,
\item[--]  the last step of  $w$ is $S(P(w))$, for a simple function $S$.
\end{itemize}
For instance, in the case of 1-sided walks, we 
keep the label $V$ for walks ending with a vertical step, but
introduce two versions of $H$, namely $(H\ |\ 1)$ (for a final $E$
step) and $(H\ |-1)$ (for a final $W$ step), 
and refine the rule~\eqref{rule-VH} as follows:
$$
V \rightarrow 
\left\{
\begin{array}{lll}
V\\
(H \ |\  1)\\
(H \ | -1)
\end{array}
\right. ,
\quad 
(H\  |\  1) \rightarrow 
\left\{
\begin{array}{lll}
V\\
(H\ |\  1)
\end{array}
\right.,
\quad 
(H\  |- 1) \rightarrow 
\left\{
\begin{array}{lll}
V\\
(H\ | -1)
\end{array}
\right. .
$$
The last step function is given by $S(V)=N$, $S(H \ |\ 1)=E$, $S(H\ | -1)=W$.
This new rewriting rule is  less compact than~\eqref{rule-VH},
and not needed for the 
calculation of the extension numbers (why compute and store
separately $\Ex((H\ |\  1),m)$ and $\Ex((H\  |- 1),m)$ while they are
equal?). However, 
it \emm is, needed for the generation stage.

In  our random generation procedures below we use the $L$-labels to
compute the extension numbers, and the $P$-labels to
determine the admissible steps and associate with them the correct new label.
Accordingly, the algorithm involves two procedures:
\begin{itemize}
\item The procedure $\Ex(\ell, m)$  computes the extension numbers
  using~\eqref{Ex-rec-L}, 
\item 
 The generation procedure reads (denoting
  $p[1]=\ell$ for a  $P$-label $p=(\ell| m)$):

\begin{quote}
\begin{tt}
gen:=proc(n)\\
{\rm \bf Initialisation}
w:=[]: p:=P(w): $\ell$:=p[1]: m:=n:\\
while m>0 do 

[p$_1$,$ \ldots$, p$_k$]:= $\De$(p):

for i from 1 to k do $\ell_i$:=p$_i$[1] od:

U:=Uniform(0,1): d:=U*$\Ex$($\ell$,m):

i:=1:

while d> $\Ex$($\ell_1$,m-1) +$\cdots$ +$\Ex$($\ell_i$,m-1) do
i:=i+1: od:

p:=p$_i$: s:=S(p): w:=ws: \\
od:
\end{tt}
\end{quote}
\end{itemize}

Since all our procedures have this form, the following sections only
describe the 
labellings $L$ and $P$ that we use for each class of walks. In all
cases, the number of children  of a $P$-label is bounded, so that the
generation procedure is achieved in linear time.
We first discuss the simple case of 2-sided
walks, and then the more complex case of general prudent walks (on the square
lattice). Once the principles underlying the corresponding labellings
are understood, it is not difficult to adapt them to 3-sided 
prudent walks
and  to triangular prudent walks. 

\subsection{Two-sided prudent walks}
A non-empty 2-sided walk $w$ can be of three different types, $I$,
$C$ or $F$, depending on the nature of its final step $s$:
\begin{itemize}
  \item[$(I)$] either  $s$ is an \emm inflating, step, which moves the
    North or East edge of the    box; 
\item[$(C)$] or  $s$ walks along the North or East edge of the
  box, moving  the endpoint \emm  closer, to the NE corner, 
\item[$(F)$] or  $s$  walks along the North or East edge, moving
  the endpoint \emm further, away from the NE corner.
\end{itemize}
We label $w$ by
$L(w)=(T;i)$, where $T$ is the type of $w$, and $i$ is the
distance of the endpoint of $w$ to the NE corner of its 
box.
It is not hard to realize that $L$ is a good labelling: the labels of
the children of a node 
$w$ of the generating tree only depend on $L(w)$. These labels are
described by the following rewriting rule:
$$
(I;i) \rightarrow 
\left\{
\begin{array}{lll}
(I;i) \\
(F; i+1) \\
(C;i-1) & \hbox{if } i >0\\
(I;0) & \hbox{if } i=0
\end{array}
\right. ,
\quad (C;i)  \rightarrow 
\left\{
\begin{array}{lll}
(I;i)\\
(C;i-1) & \hbox{if } i >0\\
(I;0) & \hbox{if } i=0
\end{array}
\right. ,
\quad (F;i)  \rightarrow 
\left\{
\begin{array}{lll}
(I;i) \\
(F;i+1) 
\end{array}
\right.
.
$$
We do not define the label of the empty walk. Two of its four
children have label $(I;0)$ (corresponding to North and
East steps), while the other two (corresponding to South and West
steps) have label $(F;1)$. The top of the  label based tree is shown
in Fig.~\ref{fig:tree} (bottom).  

The number of extensions of length $m$ of a walk labelled $\ell$ can now
be computed using~\eqref{Ex-rec-L}. For a walk of length $m$ and
label $(T;i)$, one has $i\le m$. Hence $O(n)$ different labels occur
in the first $n$ levels of the tree, and the calculation of the extension
numbers $\Ex(\ell,m)$ needed to generate a random
walk of length $n$ requires $O(n^2)$  operations.

The $L$-label $L(w)$ does not determine the last step $s$ of
$w$.  For the generation stage, we refine it by
defining $P(w)=(L(w) \ | \ s)$. The last step function is of course
$S(P(w))=s$. The generating tree based on the labelling $P$ has
rewriting rule:
$$
(I;i 
\ | \ s) \rightarrow 
\left\{
\begin{array}{llrll}
(I;i &|& s ) \\
(F; i+1 &|& 3-s ) \\
(C;i-1 &|& 1-s ) &\hbox{if } i >0\\
(I;0 &|& 1-s) & \hbox{if } i=0
\end{array}
\right. ,
$$
$$
(C;i\ | \ s)  \rightarrow 
\left\{
\begin{array}{llrl}
(I;i &|& 1-s)  \\
(C;i-1 &|& s )& \hbox{if } i >0\\
(I;0 &|&   s)& \hbox{if } i=0
\end{array}
\right. ,
\quad (F;i\ | \ s ) \rightarrow 
\left\{
\begin{array}{llr}
(I;i &|& 3-s )\\
(F;i+1 &|&  s)
\end{array}
\right.
.
$$
We have written $s=0$ (resp. $1$, $2$, $3$) for a North
(resp. East, South, West) step.
The four children of the root have labels $(I;0 \ |\ 0)$, $(I;0 \ |\
1)$, $(F;1 \ |\ 2)$,   $(F;1 \ |\ 3)$.

\bigskip
\noindent {\bf Remark.} As shown in~\cite{duchi}, the language on the
alphabet $\{N,S,E,W\} $ that naturally encodes 2-sided walks is \emm
non-ambiguous context-free,. This implies that one can generate these walks using
another type of recursive approach, based on factorisations of the
walks~\cite{flajoletcalculusINRIA}. This approach achieves a
better complexity than our naive step-by-step construction:
only  $O(n)$ operations are needed in the precalculation
stage. Moreover, it can be optimized further  using the principles of
\emm Boltzmann approximate-size sampling,~\cite{boltzmann}. As the
other classes of prudent walks we have studied cannot be described by
context-free grammars, we have focussed above on the step-by-step recursive
approach. Independently,
we have also implemented the Boltzmann approach, which produced
the second and third walks of Fig.~\ref{fig:random-2s}.

\subsection{General prudent walks on the square lattice}
\label{sec:generation-prudent}
A non-empty prudent walk $w$ can be of two different types, $I$ or $A$,
depending on the nature of its final step $s$:
\begin{itemize}
  \item[$(I)$] either  $s$ is an \emm inflating, step, which moves one of the
  edges of the  box,
\item[$(A)$] or $s$ walks \emm along, one of the edges of the box.
\end{itemize}
Consider the last edge of the box that has moved while
constructing $w$.  Say it is
the North edge (the other cases are treated similarly after  a
rotation). Then $w$ ends on this edge. If the final step $s$ 
is inflating (North),
or walks along the top edge  in \emm clockwise, direction (East), 
let $i$ (resp. $j$) be the distance between the endpoint
of $w$ and the NE (resp. NW)  corner of the box.  Otherwise,
$s$ walks along the top edge in a counterclockwise direction (West):
 exchange the roles of $i$ and $j$. The top edge has length
$i+j$. Let $h$ the other dimension of the box. We label $w$ by 
$L(w)=(I; \{i,j\}, h)$ if $w$ is inflating, by $L(w)=(A; i, j, h)$ otherwise. 
If $i=j$, the set $\{i,j\}$ has to be understood as the multiset where
$i$ is repeated twice.

It is not hard to realize that the labels of the children
of a node  $w$  only depend on $L(w)$. These labels are
described by the left part of the  rewriting rules given below (ignoring
for the moment what follows the vertical bar).
As $i, j$ and $h$ are bounded by $n$, the calculation of the extension
numbers needed to  generate a random walk of length $n$ will require
$O(n^4)$  operations. 

The label $L(w)$ does not determine the last step $s$ of
$w$.  For the generating stage, we refine it  into a $P$-label as follows:
\begin{itemize}  
\item[--] 
for a walk of type $I$, we keep track of the \emm ordered, pair  $(i,j)$
and of  the last step $s$; hence the $P$-label reads $(I;\{i,j\},h \,|\, i,j,s)$,
 \item[--] for a walk of type $A$, we keep track of the last step $s$
 and of its  direction $d$  around the box ($d=1$ if $s$ goes in
a clockwise direction around the box, $d=-1$ otherwise). The refined
 label thus reads $(A; i,j,h\, |\,s,d)$.
\end{itemize}
The combined rewriting rule reads:
$$\begin{array}{l}
(I;\{i,j\},h \ | \  i,j,s) \rightarrow
\left\{\begin{array}{llrll}
(I;\{i,j\},h+1 & |&  i,j,s )\\
(A;i-1,j+1,h  &|&  s+1,1 ) & \hbox{if } i >0\\
(A;j-1,i+1,h  &|&  s-1,-1 ) & \hbox{if } j >0\\
(I;\{0,h\}, j+1  & |&  h,0,s+1 ) & \hbox{if }i=0\\
(I;\{0,h\}, i+1  &|&  0,h,s-1 ) & \hbox{if }j=0
\end{array}\right. ,\\
\\
(A;i,j,h \ | \ s,d ) \rightarrow
\left\{\begin{array}{llrl}
(I;\{i,j\},h+1 &|&  i \frac {\, d\, }{}j, j \frac {\, d\, }{}i , s-d)
\\
(A;i-1,j+1,h   &|&  s,d  ) & \hbox{if } i >0\\
(I;\{0,h\}, j+1  &|&   h \frac {\, d\, }{}0,  0 \frac {\, d\, }{} h,  s ) & \hbox{if } i=0
\end{array}\right. ,
\end{array}$$
with
\beq\label{notationij}
i \frac {\, d\, }{}j =
\left\{
\begin{array}{ll}
  i &\hbox{if } d=1,\\
j &\hbox{if } d=-1.
\end{array}\right.
\eeq
Again, we have written $s=0$ (resp. $1$, $2$, $3$) if $s$ goes North
(resp. East, South, West), and these values are taken modulo 4.
The four children of the root have labels $(I;\{0,0\},1 \ |\ 0,0,s)$
for $s=0, 1, 2,3$.


\subsection{Three-sided prudent walks}
 
A non-empty 3-sided walk can be of five different types, $I_v$,
$I_h$,  $A$, $C$, or $F$,
depending on the nature of its final step $s$:
\begin{itemize}
  \item[$(I_v)$] either  $s$ is a \emm vertical,  inflating step (it moves
    the North  edge of the  box),
 \item[$(I_h)$] or  $s$ is a \emm horizontal, inflating step (it moves
    the East or West  edge of the  box),
  \item[$(A)$] or $s$ walks \emm along, the top edge,
\item[$(C)$] or $s$ walks upwards on a vertical edge, moving the
  endpoint of the walk \emm   closer, to the top end of the edge,
\item[$(F)$] or $s$ walks downwards on a vertical edge, moving the
  endpoint \emm   further, away from the top end of the edge.
\end{itemize}
For a walk $w$ of type $I_v$ or $A$, we denote by $i$
and $j$ the  distances between the endpoint of the walk and the NE and
NW corners, with the same convention as for general prudent walks
(Section~\ref{sec:generation-prudent}). Otherwise, $i$ denotes the
distance between the 
endpoint of $w$ and the top end of the vertical edge where $w$ ends,
 and $j$ denotes the width of the box. The
label of a walk $w$ of type 
$T$ is $L(w)=(I_v; \{i,j\})$ if $T=I_v$ and  $(T;i,j)$  otherwise.
Then $L$ is a good labelling, described by the first part of 
the  rewriting rules given below.
The calculation of the extension numbers needed to  generate a random
walk of length $n$  requires $O(n^3)$  operations.

The $L$-labelling does not determine the last step $s$ of
$w$.  For the generating stage, we refine it as follows:
\begin{itemize}
  \item[--] for a walk of type $I_v$, we keep track of the ordered
  pair  $(i,j)$, 
  \item[--] for a walk of type $A$ or $I_h$, we keep track of the last
  step $s$, 
\item[--] for a walk of type  $C$ or $F$, we keep track of the direction $d$
  of the last step around the box:
  $d=1$ if the last step goes in a clockwise 
direction around the box or consists only of South steps, $d=-1$ otherwise. 
\end{itemize}
 The last step is explicitly given by the refined label for types $A$ and
$I_h$, and is otherwise obtained via the last step function $S$ defined
by $S(I_v ;\ldots) =S(C; \ldots)=0$, $S(F;\ldots)=2$.
The generating rules read, with the notation~\eqref{notationij}:
$$
\begin{array}{ccl}
(I_v; \{i,j\} \ |\  i,j \ ) &\rightarrow &
\left\{
\begin{array}{llrll}
(I_v; \{i,j\} &|& i,j  ) \\
(A;i-1,j+1  &|& 1) & \hbox{if } i >0\\
(A;j-1,i+1 &|& 3) & \hbox{if } j >0\\
(I_h;0, j+1  &|& 1) & \hbox{if } i=0\\
(I_h;0, i+1  &|& 3) & \hbox{if } j=0
\end{array}
\right. ,\\
\\
 (I_h;i,j  \ |\ s) &\rightarrow &
\left\{
\begin{array}{llrll}
(I_h;i,j+1 &|& s) \\
(F; i+1,j &|&2-s)\\
(C;i-1,j  &|& s-2) & \hbox{if } i >0\\
(I_v;  \{0,j\}   &|& j\frac{\,s-2\,}{}0 , 0\frac{\,s-2\,}{} j  ) & \hbox{if } i=0
\end{array}
\right.,\\
\\
  (A;i,j  \ |\ s) &\rightarrow &
\left\{
\begin{array}{llrll}
(I_v; \{i,j\}  &|& j\frac{\,s-2\,}{}i ,i\frac{\,s-2\,}{}j ) \\
(A;i-1,j+1  &|& s) & \hbox{if } i >0\\
(I_h;0, j+1  &|& s) & \hbox{if } i=0
\end{array}
\right.,\\
\\
 (C;i,j  \ | d) &\rightarrow &
\left\{
\begin{array}{llrll}
(I_h;i,j+1 &|& 2+d) \\
(C;i-1,j &|&  d) & \hbox{if } i >0\\
(I_v;  \{0,j\}  &|& j\frac{\,d\,}{}0 ,0\frac{\,d\,}{}j   ) & \hbox{if } i=0
\end{array}
\right.,\\
\\
 (F;i,j  \ |\ d) &\rightarrow &
\left\{
\begin{array}{llrll}
(F; i+1,j &|& d) &\\
(I_h;i,j+1 &|& 2-d) & 
\\
(I_h;i,1 &|& 3)   &\hbox{if } j =0  
\end{array}
\right. .
\end{array}
$$
 The four children of the root have
labels $(I_h; 0,1| 1)$, $(I_h; 0,1| 3)$, $(I_v; \{0,0\} | 0,0) $ and $(F;1, 0|1)$.


\subsection{Triangular prudent walks}
The principles that underlie the generation of triangular prudent walks are
the same as for prudent walks on the square lattice (Section~\ref{sec:generation-prudent}). The type of a
walk $w$ depends on whether its last step $s$ is \emm inflating, (type $I$) or walks
\emm along, the  box (type $A$). Consider the last edge of the
              box that has changed while constructing
$w$. Assume it is the right edge (the other cases are treated
similarly after a 
rotation). If $s$ points up (NW or NE), then $i$ (resp. $j$) denotes 
the distance between the endpoint of $w$ and the top vertex
(resp. bottom right vertex) of the box. Otherwise, the roles of $i$ and
$j$ are exchanged. Define $L(w)=(T;i,j)$, where $T$ is the
type of $w$. Then $L$ is a good labelling. From this, one computes the
extension numbers  needed to generate a walk of length $n$ in
time $O(n^3)$.  

During the generation stage, we also keep track of the last step
$s$ of $w$, and of its direction $d$ around the box: if the last edge that has
changed is the right one,
then $d=1$ if $s$ goes East of South-East,  $d=-1$
otherwise. The other cases are treated similarly after a rotation. The
combined rewriting rules read: 
$$\begin{array}{ccl}
(I;i,j \ | \ s, d ) &\rightarrow &
\left\{
\begin{array}{llrl}
(I;i,j+1  &|& s,d) \\
(I;j,i+1  &|& s-d,-d) \\
(A;j-1,i+1  &|& s-2d, -d)\\
(A;i-1,j+1  &|& s+d, d) & \hbox{if } i >0\\
(I;0,j+1  &|& s+d, -d) & \hbox{if } i=0\\
(I;j,1 &|& s+2d, d) & \hbox{if } i=0
\end{array}
\right. ,
\end{array}$$
$$\begin{array}{ccl}
 (A;i,j \ | \ s,d ) &\rightarrow& 
\left\{
\begin{array}{llrl}
(I;i,j+1  &|& s-d,d) \\
(I;j, i+1  &|& s-2d,-d) \\
(A;i-1,j+1  &|& s, d) & \hbox{if } i >0\\
(I;0,j+1  &|& s, -d) & \hbox{if } i=0\\
(I;j,1 &|& s+d, d) & \hbox{if } i=0\\
\end{array}
\right..
\end{array}
$$
Again, the steps are encoded by $0, 1, \ldots, 5$ in clockwise
order, and considered modulo 6. The NW step is encoded by $0$. The six
children of the root have labels $ (I;0 ,1\  | \ s,d)$, with $(s,d)=
(0,-1),( 1,1), (2,-1), ( 3,1), (4,-1 ), (5,1 )$.

\section{Some questions and perspectives}

The enumeration of general prudent walks remains an open
problem, although we have established a functional equation defining their
generating function (Lemma~\ref{lem:4-sided}). Recall that this g.f. is
expected to be non-D-finite, and to have the same radius of
convergence as the g.f.s of 2-sided and 3-sided walks
(Propositions~\ref{prop-2-sides-asympt} and~\ref{prop-nature-3sided})~\cite{guttmann}. 

It may be worth investigating  the number of 
prudent walks that span a given box. This number is indeed
simple in the case of triangular prudent walks
(Proposition~\ref{prop:triangle-simple}), even though the length \gf\
is not D-finite. 

Regarding random generation, we think that we have optimized the
 step-by-step recursive approach. However, it only provides walks with
a few hundred steps. We believe that a better understanding of our
 results, and of the combinatorics of prudent walks, should lead to
 more efficient sampling procedures. 

All the walks we have been able to count
move away from the origin at a positive speed. Would it be possible to
introduce a non-uniform distribution on these walks that would make
them significantly more compact? Recall that the end-to-end distance
in an $n$-step self-avoiding walk is expected to grow like $n^{3/4}$ only.

In the same spirit, it should be possible to study \emm kinetic, versions of
prudent walks, where at each tick of the clock, the walk 
takes, with equal probability, one of the available prudent steps. The
resulting distribution on $n$-step walks is no longer
uniform.  This  version is
actually the one that was studied in~\cite{turban-debierre1}
and~\cite{santra}. Clearly, the random 
generation of these kinetic walks can be performed in linear time, and
does not require any precomputation (in terms of generating trees, one
chooses uniformly one of the children). Random trajectories differ
significantly from the uniform case (Fig.~\ref{fig:kinetic}). Their behaviour has very recently been 
explained by Beffara \emm et al.,~\cite{beffara-prudent}.

\begin{figure}[pbth]
 \begin{center}
\includegraphics[height=3cm,width=3cm]{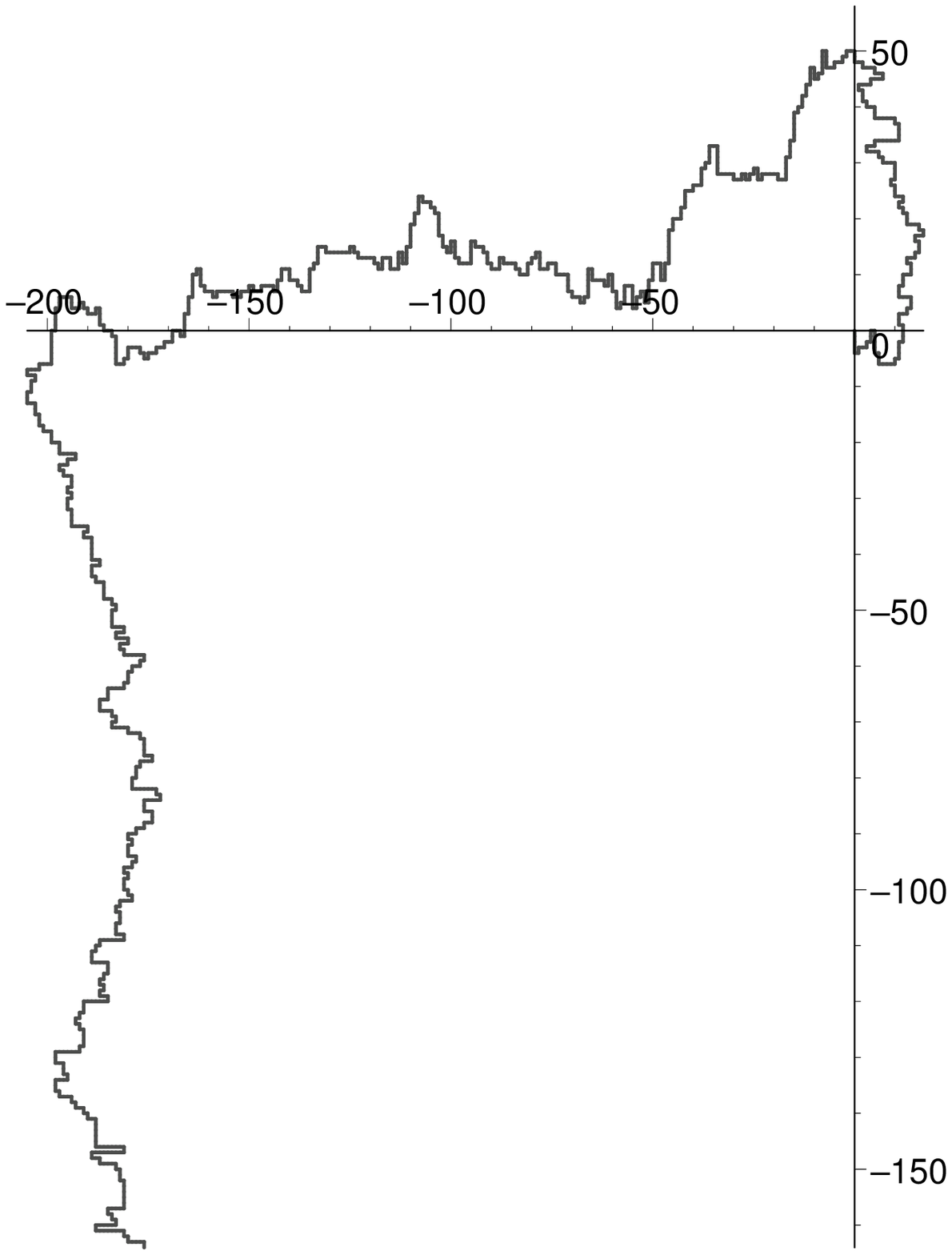}
\hskip 10mm
\includegraphics[height=3cm,width=3cm]{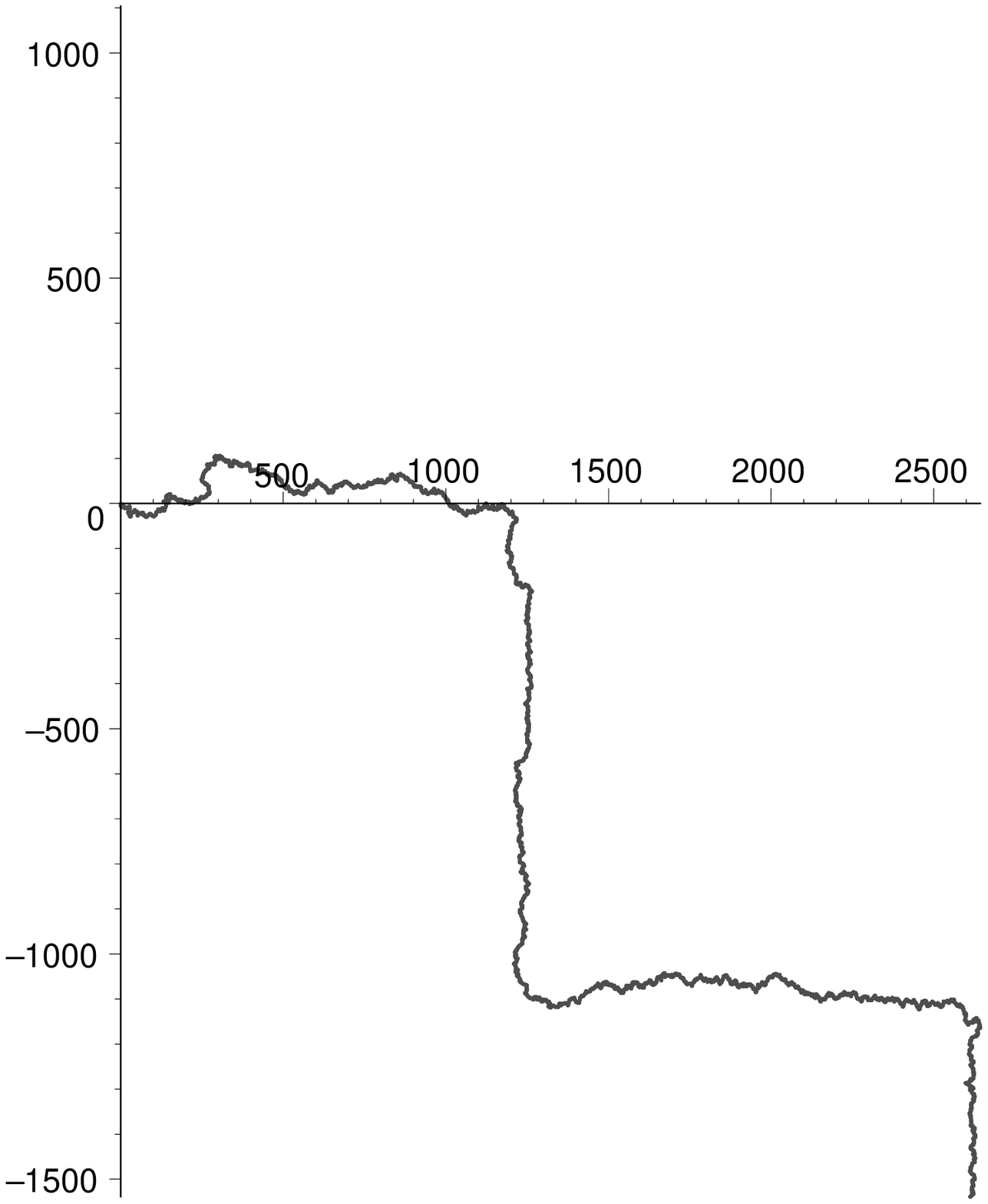}
\hskip 10mm
\includegraphics[height=3cm,width=3cm]{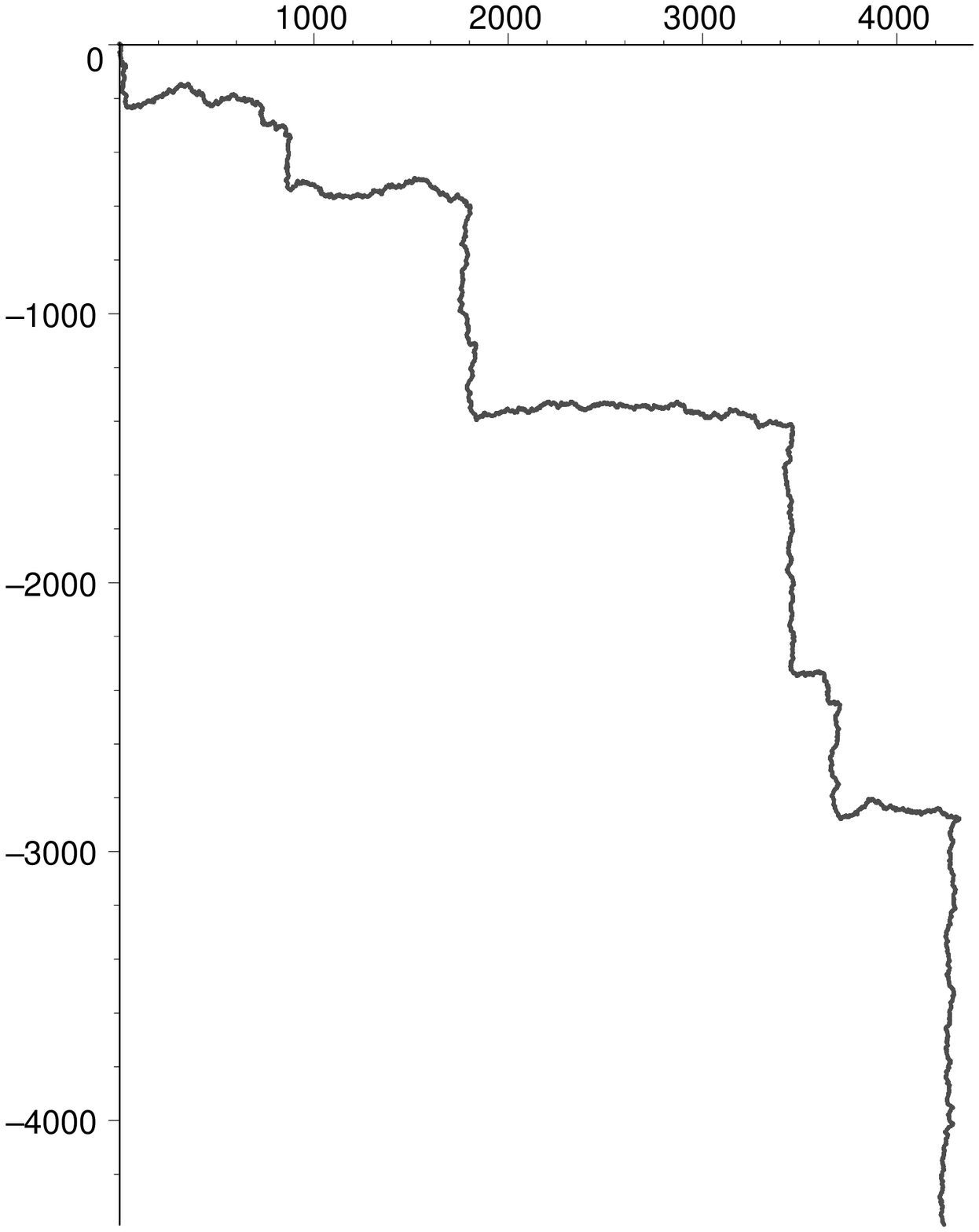}
    \caption{Kinetic prudent walks with  1000, 10000 and 20000 steps.}
    \label{fig:kinetic}
  \end{center}
\end{figure}


\bigskip
\noindent
{\bf Acknowledgements.} I am indebted to Tony Guttmann for 
discussions on prudent walks, which gave a new impulse to my interest
in these objects.

\bibliographystyle{plain}
\bibliography{biblio}

\end{document}

%% file: Figures/triangle-example.pstex_t
\begin{picture}(0,0)%
\includegraphics{triangle-example.pstex}%
\end{picture}%
\setlength{\unitlength}{4144sp}%
\begingroup\makeatletter\ifx\SetFigFont\undefined%
\gdef\SetFigFont#1#2#3#4#5{%
  \reset@font\fontsize{#1}{#2pt}%
  \fontfamily{#3}\fontseries{#4}\fontshape{#5}%
  \selectfont}%
\fi\endgroup%
\begin{picture}(6660,1620)(1048,-1714)
\put(6526,-511){\makebox(0,0)[lb]{\smash{{\SetFigFont{14}{16.8}{\familydefault}{\mddefault}{\updefault}$v$}}}}
\put(7111,-1141){\makebox(0,0)[lb]{\smash{{\SetFigFont{14}{16.8}{\familydefault}{\mddefault}{\updefault}$u$}}}}
\end{picture}%

%% file: Figures/measures.pstex_t
\begin{picture}(0,0)%
\includegraphics{measures.pstex}%
\end{picture}%
\setlength{\unitlength}{4144sp}%
\begingroup\makeatletter\ifx\SetFigFont\undefined%
\gdef\SetFigFont#1#2#3#4#5{%
  \reset@font\fontsize{#1}{#2pt}%
  \fontfamily{#3}\fontseries{#4}\fontshape{#5}%
  \selectfont}%
\fi\endgroup%
\begin{picture}(14448,2202)(778,-1966)
\put(1576,-1951){\makebox(0,0)[lb]{\smash{{\SetFigFont{17}{20.4}{\familydefault}{\mddefault}{\updefault}{\color[rgb]{0,0,0}Two-sided}%
}}}}
\put(2521, 29){\makebox(0,0)[lb]{\smash{{\SetFigFont{17}{20.4}{\familydefault}{\mddefault}{\updefault}{\color[rgb]{0,0,0}$u$}%
}}}}
\put(5176, 29){\makebox(0,0)[lb]{\smash{{\SetFigFont{17}{20.4}{\familydefault}{\mddefault}{\updefault}{\color[rgb]{0,0,0}$v$}%
}}}}
\put(6436, 29){\makebox(0,0)[lb]{\smash{{\SetFigFont{17}{20.4}{\familydefault}{\mddefault}{\updefault}{\color[rgb]{0,0,0}$u$}%
}}}}
\put(9226, 29){\makebox(0,0)[lb]{\smash{{\SetFigFont{17}{20.4}{\familydefault}{\mddefault}{\updefault}{\color[rgb]{0,0,0}$w$}%
}}}}
\put(7246,-1951){\makebox(0,0)[lb]{\smash{{\SetFigFont{17}{20.4}{\familydefault}{\mddefault}{\updefault}{\color[rgb]{0,0,0}Three-sided}%
}}}}
\put(12691, 29){\makebox(0,0)[lb]{\smash{{\SetFigFont{17}{20.4}{\familydefault}{\mddefault}{\updefault}{\color[rgb]{0,0,0}$v$}%
}}}}
\put(13951, 29){\makebox(0,0)[lb]{\smash{{\SetFigFont{17}{20.4}{\familydefault}{\mddefault}{\updefault}{\color[rgb]{0,0,0}$u$}%
}}}}
\put(12826,-1951){\makebox(0,0)[lb]{\smash{{\SetFigFont{17}{20.4}{\familydefault}{\mddefault}{\updefault}{\color[rgb]{0,0,0}Four-sided}%
}}}}
\put(15211,-1006){\makebox(0,0)[lb]{\smash{{\SetFigFont{17}{20.4}{\familydefault}{\mddefault}{\updefault}{\color[rgb]{0,0,0}$w$}%
}}}}
\put(11116,-781){\makebox(0,0)[lb]{\smash{{\SetFigFont{17}{20.4}{\familydefault}{\mddefault}{\updefault}{\color[rgb]{0,0,0}$u$}%
}}}}
\end{picture}%

%% file: Figures/generating-tree.pstex_t
\begin{picture}(0,0)%
\includegraphics{generating-tree.pstex}%
\end{picture}%
\setlength{\unitlength}{4144sp}%
\begingroup\makeatletter\ifx\SetFigFont\undefined%
\gdef\SetFigFont#1#2#3#4#5{%
  \reset@font\fontsize{#1}{#2pt}%
  \fontfamily{#3}\fontseries{#4}\fontshape{#5}%
  \selectfont}%
\fi\endgroup%
\begin{picture}(15525,5548)(-1439,-5294)
\put(-1439,-5236){\makebox(0,0)[lb]{\smash{{\SetFigFont{12}{14.4}{\familydefault}{\mddefault}{\updefault}{\color[rgb]{0,0,0}$(F;2)$}%
}}}}
\put(-854,-5236){\makebox(0,0)[lb]{\smash{{\SetFigFont{12}{14.4}{\familydefault}{\mddefault}{\updefault}{\color[rgb]{0,0,0}$(I;1)$}%
}}}}
\put(-224,-5236){\makebox(0,0)[lb]{\smash{{\SetFigFont{12}{14.4}{\familydefault}{\mddefault}{\updefault}{\color[rgb]{0,0,0}$(F;1)$}%
}}}}
\put(361,-5236){\makebox(0,0)[lb]{\smash{{\SetFigFont{12}{14.4}{\familydefault}{\mddefault}{\updefault}{\color[rgb]{0,0,0}$(I;0)$}%
}}}}
\put(901,-5236){\makebox(0,0)[lb]{\smash{{\SetFigFont{12}{14.4}{\familydefault}{\mddefault}{\updefault}{\color[rgb]{0,0,0}$(I;0)$}%
}}}}
\put(4501,-5236){\makebox(0,0)[lb]{\smash{{\SetFigFont{12}{14.4}{\familydefault}{\mddefault}{\updefault}{\color[rgb]{0,0,0}$(F;2)$}%
}}}}
\put(5086,-5236){\makebox(0,0)[lb]{\smash{{\SetFigFont{12}{14.4}{\familydefault}{\mddefault}{\updefault}{\color[rgb]{0,0,0}$(I;1)$}%
}}}}
\put(5626,-5236){\makebox(0,0)[lb]{\smash{{\SetFigFont{12}{14.4}{\familydefault}{\mddefault}{\updefault}{\color[rgb]{0,0,0}$(C;0)$}%
}}}}
\put(3286,-5236){\makebox(0,0)[lb]{\smash{{\SetFigFont{12}{14.4}{\familydefault}{\mddefault}{\updefault}{\color[rgb]{0,0,0}$(F;3)$}%
}}}}
\put(3871,-5236){\makebox(0,0)[lb]{\smash{{\SetFigFont{12}{14.4}{\familydefault}{\mddefault}{\updefault}{\color[rgb]{0,0,0}$(I;2)$}%
}}}}
\put(1531,-5236){\makebox(0,0)[lb]{\smash{{\SetFigFont{12}{14.4}{\familydefault}{\mddefault}{\updefault}{\color[rgb]{0,0,0}$(F;1)$}%
}}}}
\put(2116,-5236){\makebox(0,0)[lb]{\smash{{\SetFigFont{12}{14.4}{\familydefault}{\mddefault}{\updefault}{\color[rgb]{0,0,0}$(I;0)$}%
}}}}
\put(2656,-5236){\makebox(0,0)[lb]{\smash{{\SetFigFont{12}{14.4}{\familydefault}{\mddefault}{\updefault}{\color[rgb]{0,0,0}$(I;0)$}%
}}}}
\put(14086,-5236){\rotatebox{360.0}{\makebox(0,0)[rb]{\smash{{\SetFigFont{12}{14.4}{\familydefault}{\mddefault}{\updefault}{\color[rgb]{0,0,0}$(F;2)$}%
}}}}}
\put(13501,-5236){\rotatebox{360.0}{\makebox(0,0)[rb]{\smash{{\SetFigFont{12}{14.4}{\familydefault}{\mddefault}{\updefault}{\color[rgb]{0,0,0}$(I;1)$}%
}}}}}
\put(12871,-5236){\rotatebox{360.0}{\makebox(0,0)[rb]{\smash{{\SetFigFont{12}{14.4}{\familydefault}{\mddefault}{\updefault}{\color[rgb]{0,0,0}$(F;1)$}%
}}}}}
\put(12286,-5236){\rotatebox{360.0}{\makebox(0,0)[rb]{\smash{{\SetFigFont{12}{14.4}{\familydefault}{\mddefault}{\updefault}{\color[rgb]{0,0,0}$(I;0)$}%
}}}}}
\put(11746,-5236){\rotatebox{360.0}{\makebox(0,0)[rb]{\smash{{\SetFigFont{12}{14.4}{\familydefault}{\mddefault}{\updefault}{\color[rgb]{0,0,0}$(I;0)$}%
}}}}}
\put(8146,-5236){\rotatebox{360.0}{\makebox(0,0)[rb]{\smash{{\SetFigFont{12}{14.4}{\familydefault}{\mddefault}{\updefault}{\color[rgb]{0,0,0}$(F;2)$}%
}}}}}
\put(7561,-5236){\rotatebox{360.0}{\makebox(0,0)[rb]{\smash{{\SetFigFont{12}{14.4}{\familydefault}{\mddefault}{\updefault}{\color[rgb]{0,0,0}$(I;1)$}%
}}}}}
\put(7021,-5236){\rotatebox{360.0}{\makebox(0,0)[rb]{\smash{{\SetFigFont{12}{14.4}{\familydefault}{\mddefault}{\updefault}{\color[rgb]{0,0,0}$(C;0)$}%
}}}}}
\put(9361,-5236){\rotatebox{360.0}{\makebox(0,0)[rb]{\smash{{\SetFigFont{12}{14.4}{\familydefault}{\mddefault}{\updefault}{\color[rgb]{0,0,0}$(F;3)$}%
}}}}}
\put(8776,-5236){\rotatebox{360.0}{\makebox(0,0)[rb]{\smash{{\SetFigFont{12}{14.4}{\familydefault}{\mddefault}{\updefault}{\color[rgb]{0,0,0}$(I;2)$}%
}}}}}
\put(11116,-5236){\rotatebox{360.0}{\makebox(0,0)[rb]{\smash{{\SetFigFont{12}{14.4}{\familydefault}{\mddefault}{\updefault}{\color[rgb]{0,0,0}$(F;1)$}%
}}}}}
\put(10531,-5236){\rotatebox{360.0}{\makebox(0,0)[rb]{\smash{{\SetFigFont{12}{14.4}{\familydefault}{\mddefault}{\updefault}{\color[rgb]{0,0,0}$(I;0)$}%
}}}}}
\put(9991,-5236){\rotatebox{360.0}{\makebox(0,0)[rb]{\smash{{\SetFigFont{12}{14.4}{\familydefault}{\mddefault}{\updefault}{\color[rgb]{0,0,0}$(I;0)$}%
}}}}}
\put(-1124,-4651){\makebox(0,0)[lb]{\smash{{\SetFigFont{12}{14.4}{\familydefault}{\mddefault}{\updefault}{\color[rgb]{0,0,0}$(F;1)$}%
}}}}
\put(361,-4651){\makebox(0,0)[lb]{\smash{{\SetFigFont{12}{14.4}{\familydefault}{\mddefault}{\updefault}{\color[rgb]{0,0,0}$(I;0)$}%
}}}}
\put(2116,-4651){\makebox(0,0)[lb]{\smash{{\SetFigFont{12}{14.4}{\familydefault}{\mddefault}{\updefault}{\color[rgb]{0,0,0}$(I;0)$}%
}}}}
\put(3556,-4651){\makebox(0,0)[lb]{\smash{{\SetFigFont{12}{14.4}{\familydefault}{\mddefault}{\updefault}{\color[rgb]{0,0,0}$(F;2)$}%
}}}}
\put(5086,-4651){\makebox(0,0)[lb]{\smash{{\SetFigFont{12}{14.4}{\familydefault}{\mddefault}{\updefault}{\color[rgb]{0,0,0}$(I;1)$}%
}}}}
\put(361,-4021){\makebox(0,0)[lb]{\smash{{\SetFigFont{12}{14.4}{\familydefault}{\mddefault}{\updefault}{\color[rgb]{0,0,0}$(I;0)$}%
}}}}
\put(13771,-4651){\rotatebox{360.0}{\makebox(0,0)[rb]{\smash{{\SetFigFont{12}{14.4}{\familydefault}{\mddefault}{\updefault}{\color[rgb]{0,0,0}$(F;1)$}%
}}}}}
\put(12286,-4651){\rotatebox{360.0}{\makebox(0,0)[rb]{\smash{{\SetFigFont{12}{14.4}{\familydefault}{\mddefault}{\updefault}{\color[rgb]{0,0,0}$(I;0)$}%
}}}}}
\put(10531,-4651){\rotatebox{360.0}{\makebox(0,0)[rb]{\smash{{\SetFigFont{12}{14.4}{\familydefault}{\mddefault}{\updefault}{\color[rgb]{0,0,0}$(I;0)$}%
}}}}}
\put(9091,-4651){\rotatebox{360.0}{\makebox(0,0)[rb]{\smash{{\SetFigFont{12}{14.4}{\familydefault}{\mddefault}{\updefault}{\color[rgb]{0,0,0}$(F;2)$}%
}}}}}
\put(7561,-4651){\rotatebox{360.0}{\makebox(0,0)[rb]{\smash{{\SetFigFont{12}{14.4}{\familydefault}{\mddefault}{\updefault}{\color[rgb]{0,0,0}$(I;1)$}%
}}}}}
\put(12286,-4021){\rotatebox{360.0}{\makebox(0,0)[rb]{\smash{{\SetFigFont{12}{14.4}{\familydefault}{\mddefault}{\updefault}{\color[rgb]{0,0,0}$(I;0)$}%
}}}}}
\put(4366,-4021){\makebox(0,0)[lb]{\smash{{\SetFigFont{12}{14.4}{\familydefault}{\mddefault}{\updefault}{\color[rgb]{0,0,0}$(F;1)$}%
}}}}
\put(8281,-4021){\rotatebox{360.0}{\makebox(0,0)[rb]{\smash{{\SetFigFont{12}{14.4}{\familydefault}{\mddefault}{\updefault}{\color[rgb]{0,0,0}$(F;1)$}%
}}}}}
\end{picture}%